\documentclass{amsart}
\usepackage[utf8]{inputenc}
\usepackage{amssymb}
\usepackage{amsmath}
\usepackage{amsthm}
\usepackage{hyperref}
\usepackage[capitalise]{cleveref}
\usepackage{graphicx}
\usepackage{bbm}
\usepackage{nameref}
\usepackage{verbatim}
\usepackage{csquotes}
\usepackage{natbib}
\usepackage{mathtools}
\usepackage{tikz}
\usepackage{tikz-cd}
\usetikzlibrary{matrix,arrows,backgrounds} 
\usetikzlibrary{babel}

\DeclarePairedDelimiter\floor{\lfloor}{\rfloor}

\title{Exponentially Mixing SRB Measures are Bernoulli}
\author{Amadeus Maldonado}
\date{}

\numberwithin{equation}{section}

\newtheorem{theorem}{Theorem}[section]
\newtheorem{proposition}[theorem]{Proposition}
\newtheorem{lemma}[theorem]{Lemma}

\newtheorem{corollary}[theorem]{Corollary}

\theoremstyle{remark}

\theoremstyle{definition}
\newtheorem{definition}[theorem]{Definition}

\newtheorem{mainthm}{Theorem}

\crefname{lemma}{Lemma}{Lemmas}
\crefname{definition}{Definition}{Definitions}
\crefname{corollary}{Corollary}{Corollaries}
\crefname{theorem}{Theorem}{Theorems}

\newcommand{\R}{\mathbb{R}}

\newcommand{\Z}{\mathbb{Z}}

\newcommand{\ep}{\varepsilon}

\begin{document}
	
	\maketitle

	\begin{abstract}
		We prove two results for $C^{1+\alpha}$ diffeomorphisms of a compact manifold preserving an SRB measure $\mu$.
		First, if $\mu$ is exponentially mixing, then it is Bernoulli.
		Second, if $\mu$ is an exponential volume limit constructed in \cite{ExpLim}, then it is also Bernoulli.
    \end{abstract}

	\tableofcontents
	
	\section{Introduction}

    \subsection{Main theorems}

    Let $f \colon M \to M$ be a $C^{1+\alpha}$ diffeomorphism of a compact manifold preserving a Borel probability measure $\mu$.
     We say that $(f,\mu)$ is \emph{exponentially mixing} if there exist $r>0$, $\beta>0$, $C>0$ such that for all $g,h \in C^r(M)$ and for all $n >0$,
    \begin{equation}
    \label{H1}
        \left|\int g\circ f^n h d\mu - \int g d\mu \int h d\mu \right|
        \leq C \|g\|_{C^r} \|h\|_{C^r} e^{-\beta n}.
    \end{equation}
    In this case, we say that $(f,\mu)$ is exponentially mixing for $C^r$ observables. By a standard interpolation argument, which can be found in Appendix B of \cite{ExpMix}, it implies that $(f,\mu)$ is exponentially mixing for $C^{r'}$ observables, for any $r'>0$.
    The drawback to making $r'>0$ small is that $\beta$ will potentially decrease.
    For the rest of the paper, when saying exponential mixing, we mean with respect to Lipschitz observables.

In \cite{ExpMix}, D. Dolgopyat, A. Kanigowski and F. Rodriguez Hertz prove that exponentially mixing smooth measures are Bernoulli. Exponential mixing is initially used to show the existence of at least one positive Lyapunov exponent, applicable to any nonatomic Borel probability measure. Adding the smoothness assumption, exponential mixing enables them to show the equidistribution of unstable manifolds on exponentially small disjoint balls covering most of the space. Following this, they construct fake center-stable foliations on these balls, with absolutely continuous holonomies between unstable manifolds. This construction allows them to apply the tools developed in \cite{OrnWeiss} to conclude that the measure is Bernoulli.

The smoothness of $\mu$ plays a crucial role in their arguments due to its compatibility with geometric structures. For instance, for non-smooth measures, it is challenging to $L^1$-approximate the indicator function of an exponentially small ball by a Lipschitz function with controlled norm, since mass might be concentrated near the boundary.

For general Borel measures, additional challenges arise with the constructed fake center-stable foliations, which offer some kind of local product structure with the Lebesgue measure on unstable manifolds. Relating this to the measure $\mu$ suggests a natural consideration of SRB measures, which are absolutely continuous with respect to Lebesgue along unstable manifolds. In this paper, we extend the main result of \cite{ExpMix} to SRB measures:

    \begin{mainthm}\label{thmA}
		Let $f \colon M \to M$ be a $C^{1+\alpha}$ diffeomorphism of a compact manifold and $\mu$ a $f$-invariant, exponentially mixing SRB measure.
		Then $(M,\mu,f)$ is Bernoulli.
	\end{mainthm}
    Let $m$ be the normalized Riemannian volume measure on $M$, not necessarily $f$-invariant. We say that \emph{volume is almost exponentially mixing} for $C^r$ observables if there exists $r>0,\beta>0,C>0$ such that for all $g,h \in C^r(M)$ with $\int h dm = 0$ and for all $n>0$,
    \begin{equation*}
        \left|\int g \circ f^n h dm\right| \leq C \|g\|_{C^r}\|h\|_{C^r}e^{-\beta n}.
    \end{equation*}
    S. Ben Ovadia and F. Rodriguez Hertz prove in \cite{ExpLim} that, if volume is almost exponentially mixing, there exists an $f$-invariant Borel measure $\mu$ such that for all $g,h \in C^r(M)$,
    \begin{equation}
        \label{H2}
        \left|\int g \circ f^n h dm - \int g d\mu \int h dm\right| \leq C \|g\|_{C^r}\|h\|_{C^r}e^{-\beta n}.
    \end{equation}
    Moreover, if $\mu$ is nonatomic, then it is an ergodic SRB measure.
    In this case, we will call $\mu$ the \emph{limit SRB measure}. They also show that $\mu$ has some equidistribution of unstables property (see Proposition 4.2 in \cite{ExpLim}).
    This led them to conjecture that the limit SRB measure $\mu$ is Bernoulli.
    In fact, we prove that:
    \begin{mainthm}\label{thmB}
		Let $f \colon M \to M$ be a $C^{1+\alpha}$ diffeomorphism of a compact manifold and suppose that volume is almost exponentially mixing. Let $\mu$ be the limit SRB measure.
		Then $(M,\mu,f)$ is Bernoulli.
	\end{mainthm}

    We remark that \cref{thmA} and \cref{thmB} are, a priori, different Theorems.
    This is because we do not know if the limit measure constructed in \cite{ExpLim} is exponentially mixing.

    \subsection{Historical background}

    Let $p = (p_1,..,p_k)$ be a probability vector.
    The measure preserving system $ (\{1,..,k\}^{\mathbb{Z}},\sigma, p^{\mathbb{Z}})$ where $\sigma$ is the shift map is called a \emph{Bernoulli shift}.
    They are a model for `randomness' in dynamical systems.
    Given a measure preserving system $(X,\mu,T)$, we say that it is \emph{Bernoulli} if it is measurably isomorphic to some Bernoulli shift.
    
    A. Kolmogorov defined entropy in \cite{Kolmogorov} as a dynamical invariant in order to distinguish different measure preserving systems.
    It is well known that two systems with the same entropy may be different.
    However, D. Ornstein showed in \cite{OrnEnt} that two Bernoulli shifts with the same entropy are measurably isomorphic.
    Therefore, showing that some class of systems is Bernoulli would also show that entropy is a full invariant in such class of systems.
    On top of that, the Bernoulli property is the strongest ergodic property since it implies the Kolmogorov property (K property), mixing of all orders, mixing, weak mixing and ergodic.

    Let $M$ be a compact Riemannian manifold and $f \colon M \to M$ a $C^{1+\alpha}$ diffeomorphism.
	Y. Pesin showed in \cite{Pes} that, if $f$ preserves a weakly mixing, hyperbolic smooth measure, then it is Bernoulli.
    For measures with some Lyapunov exponents equal to zero, not even the K property is enough to show Bernoulli. Indeed, in \cite{KnotBernKatok} A. Katok constructs a volume preserving diffeomorphism in an $8$-dimensional manifold which satisfies the K property, but is not Bernoulli.
    In \cite{KnotBern4}, A. Kanigowski, F. Rodriguez Hertz and K. Vinhage construct an example in dimension $4$. In dimension $2$, the K property implies hyperbolicity (no zero exponents), and Pesin's Theorem applies. It is unknown whether the K property implies Bernoulli in dimension $3$.

	For systems not preserving volume, one looks for other natural invariant measures.
	An important class of measures are the so called Sinai Ruelle Bowen (SRB) measures.
    Their precise definition is given in \cref{SRBsection}.
    They were first constructed and studied for Axiom A attractors
    (see \cite{S2}, \cite{R1}, \cite{R2}, \cite{B2}, \cite{K1} and, for an overview, \cite{YoungSurvey})
    and later verified to exist in various settings.
    F. Ledrappier showed in \cite{Led} that hyperbolic, weak mixing, SRB measures are Bernoulli, extending Y. Pesin's result.
    
    We have so far talked about qualitative properties of measure preserving systems.
    Exponential mixing is a quantitative property that requires a smooth structure. 
    It obviously implies mixing, but its connection with stronger ergodic properties was not known until recently.
    In \cite{KanigBern}, A. Kanigowski shows that homogeneous systems with the K property are Bernoulli (see Theorem 1.1 in \cite{KanigBern} for a detailed statement).
    He uses the fact that such systems are exponentially mixing (see \cite{Howe}, \cite{Moore} and \cite{Oh}) and, to the best of our knowledge, it is the first instance that exponential mixing is used to show the Bernoulli property.
    In \cite{ExpMix}, they show that exponential mixing of a smooth measure by itself is enough to the Bernoulli property.
    \cref{thmA} reinforces that connection, extending it to SRB measures.
    
    Section 10 of \cite{ExpMix} provides examples of Bernoulli systems arising from several semisimple extensions of volume preserving Anosov diffeomorphisms.    
    Theorem 4.1 of \cite{DDKN}, which does not require smoothness of the measures considered, guarantees that the resulting skew product systems will be exponentially mixing, provided that a certain large deviations estimate holds. 
    Considering an SRB measure on the base instead, the same construction provides new examples of Bernoulli systems by \cref{thmA}.
    This is because the product measure is an SRB measure for the skew product system, as we verify in \cref{SRBskew}.

    One may also study statistical properties of continuous time systems.
    Let $f^t \colon M \to M$ be a $C^{1+\alpha}$ flow preserving a Borel probability measure $\mu$. 
    It is \emph{exponentially mixing} for $C^r$ observables if there exist $r>0$, $\beta>0$, $C>0$ such that for all $g,h \in C^r(M)$ and for all $t >0$,
    \begin{equation}
    \label{H3}
        \left|\int g\circ f^t h d\mu - \int g d\mu \int h d\mu \right|
        \leq C \|g\|_{C^r} \|h\|_{C^r} e^{-\beta t}.
    \end{equation}
    We say that the flow $(M,\mu,f^t)$ is \emph{Bernoulli} if for each fixed $t_0>0$, $(M,\mu,f^{t_0})$ is Bernoulli.
    If \eqref{H3} holds for all $t>0$, it in particular holds for $t \in t_0\mathbb{N}$.
    In particular, if $(M,\mu,f^t)$ is an exponentially mixing flow, then for each $t_0>0$, $(M,\mu,f^{t_0})$ is exponentially mixing.
    If $\mu$ is an SRB measure for the flow $f^t$, then it is also an SRB measure for the diffeomorphism $f^{t_0}$ for each $t_0>0$, since they share the same unstable manifolds.
    Therefore, the following is a Corollary of \cref{thmA}.
    \begin{mainthm}
		Let $f^t \colon M \to M$ be a $C^{1+\alpha}$ flow of a compact manifold and $\mu$ a $f$-invariant, exponentially mixing SRB measure.
		Then $(f^t,\mu)$ is Bernoulli.
	\end{mainthm}

	\subsection{Overview of the proof}
	
	The proof heavily relies on Pesin theory.
	We therefore set up notation and state many of the necessary results needed.
	In particular, we prove specific estimates not easily found in the literature that may be useful in other contexts dealing with Pesin theory.
	Some of them are change of coordinates for Lyapunov charts in \cref{changeofcoordinateslyap} and explicit log H\"older continuity of the unstable Jacobain on Pesin sets.
	
	We apply the machinery developed by D. Ornstein and B. Weiss in \cite{OrnWeiss} to verify the Bernoulli property.
	For our context, we reformulate it in terms of most unstable manifolds  equidistributing at time $\floor{\varepsilon n}$ along disjoint sets $B_i$ that cover most of $M$ and have the property that any two points $x,y \in B_i$ stay $\varepsilon$-close until time $n$. This is the content of \cref{EBB}.
	
	We therefore need to pass from \eqref{H1} or \eqref{H2} to equidistribution of unstable manifolds.
	This is an argument that goes back to G. Margulis (see \cite{Margulis2004}) and consists of two steps:
	We first thicken an unstable manifold $W$ by an exponentially narrow tube $C$ and $L^1$-approximate the indicator function $\mathbbm{1}_C$ by a Lipschitz function $\varphi$ with controlled Lipschitz norm $\|\varphi\|_{\text{Lip}}$.
	Secondly, we use some form of local product structure to approximate the measure conditioned on the tube $C$ to the measure conditioned on  the unstable $W$.
	
	The first step is standard for smooth measures.
	However, for a general Borel measure $\mu$, it is unclear how to well $L^1(\mu)$-approximate $\mathbbm{1}_C$ by $\varphi$, since $C$ may have mass accumulated around its boundary.
	To overcome this difficulty, we construct in \cref{goodbox} a disjoint family of exponentially small boxes that cover most of $M$ and have little measure accumulated on a quantified neighborhood of their boundaries.
	They are also mostly covered by pieces of unstable manifolds.
	This part is done for any Borel measure $\mu$.
	Similar problems of covering $M$ by boxes with dynamical meaning and good properties near their boundaries have also been present in works on symbolic dynamics such as \cite{Sarig},\cite{LimaSarig} and \cite{AraujoLimaPoletti}.
	
	For the second step, we use the previously constructed boxes to build fake center stable manifolds on them, similar to what was done in \cite{ExpMix}.
	The main difference is that following their construction, showing that these foliations cover most of $M$ is challenging if the measure is not smooth.
	Our construction is done by pulling back good foliated boxes at time $n$ and intersecting them with good boxes at time $0$.
	Since good boxes cover most of $M$, the fake center stable foliation constructed will also cover most of $M$. This is done in \cref{fakecsconstruction}.
	
	Given two unstable manifolds $W_1,W_2$ in a good box, the fake center stable holonomy $\pi^{cs} \colon W_1 \to W_2$ is a $C^{1+\alpha}$ map with Jacobian close to $1$.
	This gives us an approximate local product structure for center stable saturated sets with respect to the unstable Lebesgue measure.
	This is where the SRB property is crucial, since we can approximate the measure conditioned on an unstable by the normalized unstable Lebesgue measure (see (P8) of \cref{goodbox}).
	
	All these allow us to show equidistribution of most unstable manifolds along exponentially small boxes $C$ in \cref{MainLemma}, as long as they're not too small.
	This is still not sufficient because the sets $B_i$ are too thin along the unstable direction, despite covering the entirety of the center stable direction. 
	To overcome this, in \cref{MainLemma2} we break up $C$ by sets $B_i$ thin along the unstable direction and show that equidistribution of an unstable along the box implies equidistribution on most of the $B_i$.
	This is because $f^{\floor{\varepsilon n}}W$ is really long along the unstable direction.
	Therefore, a typical connected component of $f^{\floor{\varepsilon n}}W \cap C$ crosses some $B_{i_0}$ inside the box $C$ if, and only if, it crosses all of the $B_i$.

 	\subsection{Notation}
	
	Given two sequences $a_n$,$b_n$, the notation $a_n \ll b_n$ means $\lim \frac{a_n}{b_n} = 0$.

	For a subset  $A \subset X$ of a metric space with metric $d$, we denote by $\mathcal{N}_t(A) = \{x \in X \ | \ d(x,A) < t\}$ and $\partial_t A = \{ x \in X \ | \ d(x,\partial A) < t\}$. 
	
	If $X,Y$ are two metric spaces and $A \subset X \times Y$ with $d_{X\times Y} = \max \{d_X, d_Y\}$, then $\mathcal{N}_{(s,t)}(x',y') = \{(x,y) \in X\times Y \ | \ d_X(x',x) < s \ \text{and} \ d_Y(y',y) < t\}$.

	For $x,y$ nonzero real numbers and $a>1$, we denote by $x \sim_a y$ to mean $\frac{x}{y} \in (a^{-1},a)$.

	\subsection{Acknowledgments}
	
	I would like to thank my PhD advisor Aaron Brown for his continued guidance throughout the making of this paper.
	I am also grateful to Snir Ben Ovadia, Dmitry Dolgopyat, Adam Kanigowski and Federico Rodriguez Hertz for kindly explaining many relevant arguments in their papers and for the many conversations regarding this paper.
	 In particular, I thank Ben Ovadia for introducing me to this project, and Dolgopyat and Kanigowski for their valuable comments on an earlier version of this paper.
	 
	I am also thankful for the fruitful conversations I had regarding this paper with Seljon Akhmedli, Keith Burns, Solly Coles and Kurt Vinhage.
	
	This material is based upon work supported by the National Science Foundation under Grant
No. DMS-2020013 and DMS-2400191.

	\section{Preliminaries}
	
	Here we write down preliminaries and fix notation for the paper.
	We will closely follow the notation and exposition from section 2 of \cite{ExpMix}.

	\subsection{Measure theory}

	Let $(X,\mathcal{B},\mu)$ be a probability space.
	A \emph{measurable finite partition} of $X$ is a collection of pairwise disjoint measurable sets $\mathcal{P} = \{P_1,...,P_k\}$ such that $\mu(\bigcup_{i=1}^k P_i) = 1$.
	We say that a property holds for \emph{$\varepsilon$-almost every atom} if the union of the atoms such that the property does not hold has measure less than $\varepsilon$.
	
	Given two finite measurable partitions
	$\mathcal{P} = \{P_1,...,P_k\}$ and $\mathcal{Q} = \{Q_1,...,Q_k\}$ of $X$, 
	we introduce the following distance between them
	\begin{equation*}
		\overline{d}(\mathcal{P},\mathcal{Q})
		= \sum_{i=1}^k \mu(P_i \Delta Q_i).
	\end{equation*}
	For $(\mathcal{P}^s)_{s=1}^S = (\{P^s_1,...,P^s_k\})_{s=1}^S$ and 
	$(\mathcal{Q}^s)_{s=1}^S = (\{\mathcal{Q}^s_1,...,\mathcal{Q}^s_k\})_{s=1}^S$
	sequences of measurable finite partitions of $X$, we define their distances to be
	\begin{equation*}
		\overline{d}\left((\mathcal{P}^s)_{s=1}^S,(\mathcal{Q}^s)_{s=1}^S\right)
		= \frac{1}{S} \sum_{s=1}^S \overline{d}(\mathcal{P}^s,\mathcal{Q}^s).
	\end{equation*}
	If $\mathcal{P} = \{P_1,...,P_k\}$ is a measurable partition of $(X,\mu)$ 
	and $\mathcal{Q} = \{Q_1,...,Q_k\}$ is a measurable partition of $(Y,\nu)$,
	then we denote by $\mathcal{P} \sim \mathcal{Q}$ if $\mu(P_i) = \nu(Q_i)$, for $i=1,...,k$.
	For $(\mathcal{P}^s)_{s=1}^S = (\{P^s_1,...,P^s_k\})_{s=1}^S$ and 
	$(\mathcal{Q}^s)_{s=1}^S = (\{\mathcal{Q}^s_1,...,\mathcal{Q}^s_k\})_{s=1}^S$
	sequences of measurable finite partitions on $(X,\mu)$ and $(Y,\nu)$ respectively,
	we define their distance to be
	\begin{equation*}
		\overline{d}\left((\mathcal{P}^s)_{s=1}^S,(\mathcal{Q}^s)_{s=1}^S\right)
		= \underset{\overline{\mathcal{Q}}^1 \sim \mathcal{Q}^1,...,
		\overline{\mathcal{Q}}^S \sim \mathcal{Q}^S}{\inf}
		\overline{d}\left((\mathcal{P}^s)_{s=1}^S,(\overline{\mathcal{Q}}^s)_{s=1}^S\right).
	\end{equation*}	
	For $A \subset X$ of positive measure, we denote the \emph{conditional measure on $A$} by $\mu_A$. It is defined to be
	\begin{equation*}
		\mu_A(B) = \frac{\mu(A\cap B)}{\mu(A)} \ \text{for} \ B \subset X.
	\end{equation*}
	If $\mathcal{P} = \{P_1,...,P_k\}$ is a measurable partition on $(X,\mu)$,
	we denote by $\mathcal{P}_A = \{P_1 \cap A,...,P_k \cap A\}$ 
	the induced partition on $(A,\mu_A)$.
	
	Now let $T \colon X \to X$ be a measurable map preserving the measure $\mu$.
	\begin{definition}
		\label{vwB}
		A finite partition $\mathcal{P} = \{P_1,...,P_k\}$ 
		is said to be \textit{very weak Bernoulli}
		if for every $ \varepsilon>0$, there exists $N_0 \in \mathbb{N}$ such that
		for $N_0\leq N_1 \leq N_2$, $N_0 \leq S$ and $ \varepsilon$-almost every atom
		$A$ of $\bigvee_{i = N_1}^{N_2}T^i(\mathcal{P})$, we have
		\begin{equation*}
			\overline{d}\left(
				(T^{-s}\mathcal{P})_{s = 0}^S,(T^{-s}\mathcal{P}_A)_{s=0}^S\right)
				< \varepsilon
		\end{equation*}
	\end{definition}
	Let $(\mathcal{P}_n)_{n \geq 1}$ be a sequence of finite measurable partitions of $X$.
	We say that $\mathcal{P}_n$ \emph{converges to the point partition} if for every measurable set $A \subset X$ and $\varepsilon>0$,
	for each $n$ large enough, there exists a  $\mathcal{P}_n$-measurable set $A_n$ such that $\mu(A\Delta A_n) < \varepsilon$.
	The following standard theorem can be found as Theorem A and Theorem B in \cite{OrnWeiss}.
	\begin{theorem}
		\label{vwBimpliesB}
		Let $(\mathcal{P}_n)_{n \geq 1}$ be a sequence of finite measurable partitions of $X$ converging to the point partition.
		If $\mathcal{P}_n$ is  very weak Bernoulli for each $n \geq 1$, then $(X,\mu,T)$ is Bernoulli.
	\end{theorem}
	
	A map $\theta\colon (X,\mu) \to (Y,\nu)$ is \emph{$\varepsilon$-almost measure preserving} if there exists $E \subset X$ with $\mu(E)< \varepsilon$ such that for any $A \subset X\setminus E$,
	\begin{equation*}
		\nu(\theta A) \sim_{1+\varepsilon} \mu(A).
	\end{equation*}
	
	We say that a probability space $(X,\mu,\mathcal{B})$ is a \emph{Lebesgue space} if $X$ is a Polish space with $\mathcal{B}$ the Borel sigma algebra.
	The following standard result can be found in \cite{Rohlin} and will be helpful for constructing $\varepsilon$-almost measure preserving maps.
	\begin{lemma}
		Any two atomless Lebesgue spaces are measurably isomorphic.
	\end{lemma}
	Let $d$ be a metric on $X$. For a finite partition $\mathcal{P} = \{P_1,...,P_k\}$ of $X$, we denote by $\partial \mathcal{P} = \bigcup_{i =1}^{k} \partial P_i$.
	We say that $\mathcal{P}$ is a \emph{regular partition} if for every $\varepsilon>0$, there exists $\delta>0$ such that $\mu(\mathcal{N}_{\delta}(\partial \mathcal{P}))< \varepsilon$.
	The following statement is Corollary 2.21 of \cite{ExpMix}, which is proven in Lemma 2.4 of \cite{OrnWeiss}.
	\begin{corollary}
		\label{vwBnice}
		Let $f \colon (X, \mathcal{B}, \mu, d) \to (X, \mathcal{B}, \mu, d)$ be ergodic and $\mathcal{P}$ be a
		regular partition of $X$. Suppose that for every $ \varepsilon >0$, there exists
		$N, \tilde{N} \in \mathbb{N}$ such that for every
		$N'\geq N, \varepsilon$-almost every atom $A \in \bigvee_N^{N'} f^i \mathcal{P}$ and every $S \geq \tilde{N}$,
		there exists an $ \varepsilon$-measure 
		preserving map $\theta = \theta(N,S,A) \colon (A,\mu_A) \to (X,\mu)$
		satisfying for $ \varepsilon$-almost every $x \in (A,\mu_A)$,
		\begin{equation}
			\label{close}
			\frac{1}{S}\#\{ i \in \{0,...,S-1\} \ : \ d(f^i(x), f^i(\theta(x))) < \varepsilon \} > 1 - \varepsilon.
		\end{equation}
		Then, $\mathcal{P}$ is very weak Bernoulli.
	\end{corollary}

	Now let $\mathcal{P}$ be a (possibly infinite) partition of $X$.
	It is said to be \textit{$\mu$-measurable} if
	there exists $p \colon X \to [0,1]$ measurable with respect to the sigma algebra $\mathcal{B}$ completed by $\mu$ such that
	for $\mu$-almost every $x \in X$,
	there exists $t_x \in [0,1]$ satisfying $\mathcal{P}(x) = p^{-1}(y_x)$.

	The next lemma is an immediate  application of Markov's inequality and is used multiple times in the paper,
	 which is why we decided to state it separately with the appropriate quantifiers.
	\begin{lemma}
		\label{AppendixLemma}
		Let $(X,\mathcal{B},\mu)$ be a probability space,
		$\varepsilon_A, \varepsilon_B > 0$ and
		$A,B \in \mathcal{B}$ with 
		\begin{equation*}
			\mu(A)< \varepsilon_A \ \text{and}  \ \mu(B) > 1- \varepsilon_B.
		\end{equation*}
		Let $\{B_i\}_{i \in I}$ be a measurable partition of $B$, with $I$ some indexing set
		and $\{\mu_{B_i}\}_{i \in I} \cup \{\mu_{B^c}\}$ be a disintegration of $\mu$
		subordinate to the measurable partition $\mathcal{P} = \{B_i\}_{i \in I} \cup \{B^c\}$.
		Given $ \varepsilon>0$, let 
		\begin{equation*}
			I' = \{ i \in I \ | \ \mu_{B_i}(A) \leq \varepsilon \}
		\end{equation*}
		then 
		\begin{equation}
			\label{trashout}
			1 - \varepsilon_B - \frac{ \varepsilon_A}{ \varepsilon} \leq \mu\left(\bigcup_{i \in I'} B_i \right).
		\end{equation}
	\end{lemma}

	\begin{proof}
	By the definition of disintegration of measure, we may write the measure of $A$ as the integral of its conditional measures
	\begin{equation*}
		\mu(A) = \int \mu_{\mathcal{P}(x)}(A) d\mu(x).
	\end{equation*}
	Applying Markov's inequality to $x \mapsto \mu_{\mathcal{P}(x)}(A)$,
	\begin{equation*}
		\mu(x \in X \ | \ \mu_{\mathcal{P}(x)}(A) > \varepsilon) <\frac{1}{ \varepsilon} \mu(A). 
	\end{equation*}
	For $i \in I \setminus I'$, we have $B_i \subset \{x \in X \ | \ \mu_{\mathcal{P}(x)}(A) > \varepsilon\}$. Therefore
	\begin{equation*}
		\mu\left(\bigcup_{i \in I \setminus I'}B_i\right) \leq \frac{1}{ \varepsilon}\mu(A).
	\end{equation*}
	Thus, we may conclude the lemma, since
	\begin{equation*}
		1 - \varepsilon_B < \mu(B) = \mu\left(\bigcup_{i \in I'}B_i\right) +
		\mu\left(\bigcup_{i \in I \setminus I'} B_i\right) \leq
		\mu\left(\bigcup_{i \in I'} B_i\right) + \frac{ \varepsilon_A}{ \varepsilon}.
	\end{equation*}

	\end{proof}

	\subsection{Pesin theory}

	For completeness, we will include Pesin theory results in this subsection.
	We closely follow the exposition in Section 2 of \cite{ExpMix}.
	Most results can be found in 
	\cite{BarPes} and \cite{Brown2016}.

	Let $M$ be a compact, connected Riemannian manifold of dimension $d$ and $f \colon M \to M$ a $C^{1+ \alpha}$ diffeomorphism.
	The following definition is stated as Definition 2.1 in \cite{ExpMix}.
	
	\begin{definition}
	\label{regular}
		Let $\lambda,\delta>0$. A point $x \in M$ is \emph{$(\lambda,\delta)$-Lyapunov regular} if for $k \in \mathbb{Z}$, there exists a splitting
		\begin{equation*}
			T_{f^kx}M =E^u(f^kx) \oplus E^{cs}(f^kx)
		\end{equation*}
		and numbers $\mathfrak{R}_{\delta}(f^kx)$ satisfying for $n \in \mathbb{Z}$,
		\begin{enumerate}
			\item $\mathfrak{R}_{\delta}(f^{k+n}x) \leq e^{\delta |k|} \mathfrak{R}_{\delta}(f^nx)$;
			\item $Df^kE^i(x) = E^i(f^kx)$, for $i \in \{cs,u\}$;
			\item if $v \in E^{cs}(f^kx)$, then
				\begin{equation*}
					\|Df^nv\| \leq e^{\delta n}\mathfrak{R}_{\delta}(f^kx) \|v\|;
				\end{equation*}
			\item if $v \in E^u(f^kx)$ and $n \geq 0$, then
				\begin{equation*}
					\|Df^{-n}v\| \leq \mathfrak{R}_{\delta}(f^{k-n}x) e^{(-\lambda+\delta) n}\|v\|;
				\end{equation*}
			\item $\angle(E^{cs}(f^kx),E^u(f^kx)) \geq \mathfrak{R}_{\delta}(f^kx)^{-1}$.
		\end{enumerate}
	\end{definition}
	Let $\mu$ be an ergodic $f$ invariant Borel probability measure on $M$. The following theorem is stated as Theorem 2.2 in \cite{ExpMix}.
	\begin{theorem}
		Suppose that the system $(M,f,\mu)$ has positive Lyapunov exponents and let $\lambda>0$ be the smallest positive Lyapunov exponent.
		For any $\delta>0$, the set of $(\lambda,\delta)$-Lyapunov regular points has full $\mu$ measure.
		Moreover, $\mathfrak{R}_{\delta}$ defined on that set can be chosen to be measurable.
		
	\end{theorem}	
	From now on, $\lambda>0$ is fixed to be the smallest Lyapunov exponent of the ergodic system $(M,f,\mu)$.
	We will denote the set of $(\lambda,\frac{\delta}{4})$-Lyapunov regular points by $LyapReg(\delta)$ and by $LyapReg = \bigcup_{\delta >0} LyapReg(\delta)$.
	
	We will now define the Lyapunov norms on $T_xM$ for $x \in LyapReg(\delta)$. 
	For $u \in E^u(x)$, let
	\begin{equation*}
		|u|'^2_{x,\delta} = \sum_{m \leq 0} \|D_xf^mu\|^2 e^{-2 \lambda m - 2\delta |m|}
	\end{equation*}
	and for $v \in E^{cs}(x)$,
	\begin{equation*}
		|v|'^2_{x,\delta} = \sum_{m \geq 0} \|D_xf^mv\|^2 e^{-2 \delta |m|}.
	\end{equation*}
	They define inner products on $E^u(x),E^{cs}(x)$ and are extended to $T_xM = E^{cs}(x) \oplus E^u(x)$ by defining these spaces to be orthogonal.
	
	Let $C_1>0$ be such that, for all $p \in M$, $\hat{f}_p = \text{exp}_{fp} \circ f \circ \text{exp}_p$ is defined on the ball of radius $\frac{1}{C_1}$ in $T_pM$ and
	\begin{equation*}
		\|(D_{\hat{z}_1}\hat{f}_p)^{-1} - (D_{\hat{z}_2}\hat{f}_p)^{-1}\| \leq C_1|\hat{z}_1-\hat{z}_2|^{\alpha}, 
		\ \text{if} \ |\hat{z}_1|,|\hat{z}_2|\leq \frac{1}{C_1}.
	\end{equation*}
	For $p \in LyapReg(\delta)$, let
	\begin{equation*}
		\mathfrak{r}_{\delta}(p) = \left( (\alpha\delta)^{-1} \frac{(\mathfrak{R}_{\alpha\frac{\delta}{4}}(p))^2}{\sqrt{1-e^{-\alpha\delta}}}2^{\frac{1-\alpha}{2}}C_1\right)^{\frac{-2}{\alpha}}.
	\end{equation*}
	We may assume that $\mathfrak{r}_{\delta}$ is bounded above by $\frac{1}{C_1}$ (see Lemma A.2 in \cite{ExpMix} and the comment below it).
	By (1) of \cref{regular}, for $x \in LyapReg(\delta)$,
	\begin{equation}
	\label{SlowExp}
		e^{-\delta} \mathfrak{r}_{\delta}(x) \leq \mathfrak{r}_{\delta}(fx) \leq e^{\delta} \mathfrak{r}_{\delta}(x).
	\end{equation}
	For $\tau>0$, let 
	\begin{equation}
		\label{pesinset}
		{\mathcal{P}}_{\tau} = {\mathcal{P}}^{\delta}_{\tau} = \{ x \in M \ | \ \tau \leq \mathfrak{r}_{\delta}(x) \}.
	\end{equation}
	We call $\mathcal{P}_{\tau}$ a \emph{Pesin set}.
	Note that $\mu(\mathcal{P}_{\tau}) \rightarrow 1$ as $\tau \rightarrow 0$.

	\subsubsection{Lyapunov Charts}

	The following statement is Lemma 2.4 of \cite{ExpMix}
	\begin{lemma}
		\label{LyapCharts}
		There are $\alpha_2,\delta_0>0$ such that for every $\delta<\delta_0$ and $x \in LyapReg(\delta)$,
		there exists a linear map $L_{x,\delta} \colon \mathbb{R}^d \to T_xM$ satisfying
		\begin{enumerate}
			\item[(L1)] $L_{x,\delta}$ is an isometry between the standard metric on $\mathbb{R}^d$ and $T_xM$ with the Lyapunov norm $| \cdot |'_{x,\delta}$.
            Moreover, $L_{x,\delta}(\mathbb{R}^i) = E^i(x)$, for $i=u,cs$;
			\item[(L2)] for $\tau>0$ small enough, the map $x \mapsto L_{x,\delta}$ is $\alpha_2$-H\"older continuous on $\mathcal{P}_{\tau}$;
			\item[(L3)] $\max\{\|L_{x,\delta}\|,\|L^{-1}_{x,\delta}\|\} \leq \mathfrak{r}_{\delta}^{-1}(x)$.
		\end{enumerate}
		Define the Lyapunov chart at $x$ by
		\begin{equation*}
			h_x = h_{x,\delta} = exp_x \circ L_{x,\delta}
		\end{equation*}
		and let
		\begin{equation*}
		 \tilde{f}_x = \tilde{f}_{x,\delta} = h^{-1}_{fx} \circ f \circ h_x.
		\end{equation*}
		Then,
		\begin{enumerate}
			\item $h_{x}(0) = x$;
			\item $\max\{\text{Lip}(h_{x}),\text{Lip}(h_x^{-1})\} \leq \mathfrak{r}_{\delta}^{-1}(x)$;
			\item $\mathcal{N}_{\mathfrak{r}_{\delta}(x)}(0) \subset domain(\tilde{f}_x)$ 
				and $\mathcal{N}_{\mathfrak{r}_{\delta}(x)}(0) \subset domain(\tilde{f}^{-1}_x)$;
			\item $e^{\lambda - \delta}|v| \leq |D_0 \tilde{f}_{x,\delta}v|$ for $v \in \mathbb{R}^u$ 
				and $|D_0\tilde{f}_{x,\delta}v| \leq e^{\delta} |v|$ for $v \in \mathbb{R}^{cs}$;
			\item $\text{H\"ol}_{\alpha_2}(D\tilde{f}_{x,\delta}) \leq \delta$,
				$Lip(\tilde{f}_{x,\delta} - D_0\tilde{f}_{x,\delta}) \leq \delta$
				and $Lip(\tilde{f}^{-1}_{x, \delta} - D_0 \tilde{f}^{-1}_{x,\delta}) \leq \delta$.
		\end{enumerate}
	\end{lemma}
	The following definition is similar to definition 6.2 of \cite{ExpMix}.

\begin{definition}
\label{horizontal}
    Let $C = h_z(\mathcal{N}^u_{\xi}(c^u)\times \mathcal{N}^{cs}_r(c^{cs}))$,
    with $\mathcal{N}^u_{\xi}(c^u)\times \mathcal{N}^{cs}_r(c^{cs}) \subset \mathcal{N}_{\mathfrak{r}_{\delta}(x)}(0)$.
    Let $W \subset C$ be a submanifold and $\tilde{W} = h_z^{-1}(W)$.
    We say that $W$ is \emph{$(\gamma,A)$-mostly vertical} 
    if $\tilde{W}$ is the graph of a $C^{1+\gamma}$ function $\eta \colon \mathcal{N}^{cs}_{r}(c^{cs}) \to \mathcal{N}^u_{\xi}(c^u)$
    with 
    \begin{equation*}
        \| D \eta\|_{C^{1+\gamma}} \leq A.
    \end{equation*}
    Similarly, it is \emph{$(\gamma,A)$-mostly horizontal} if it's
    the graph of a $C^{1 + \gamma}$ function $\eta \colon \mathcal{N}_{\xi}^u(c^u) \to \mathcal{N}_r^{cs}(c^{cs})$ satisfying
    \begin{equation*}
        \| D \eta\|_{C^{1+\gamma}} \leq A.
    \end{equation*}
\end{definition}

\subsubsection{Change of coordinates for Lyapunov charts}
\label{changeofcoordinateslyap}
	It will also be important for this paper to have estimates for the change of Lyapunov coordinates.
	\begin{lemma}
		\label{changecoordinates}
		For $\tau>0$ small enough, there exists $C = C(\tau)>0$ such that
		for $p,q \in \mathcal{P}_{\tau}$ and $z \in \mathcal{N}_{\tau}(0) \subset \mathbb{R}^d$,
		\begin{equation}
			\label{changecoordinateseq}
			\|D_z(h_q^{-1} \circ h_p) - Id\| \leq C(d_M(p,q)^{\alpha_2} + |z|)
		\end{equation}	
	\end{lemma}
	\begin{proof}
		The map $h_q^{-1} \circ h_p$ is given by composing the following maps
		\begin{equation*}
			\label{changeofcoordproof}
			\mathbb{R}^d \xrightarrow{L_p} T_pM \xrightarrow{\text{exp}_p} M \xrightarrow{\text{exp}_q^{-1}} T_qM \xrightarrow{L_q^{-1}} \mathbb{R}^d.
		\end{equation*}
		Differentiating we get
		\begin{equation}
			\label{changeofcoordinates1}
			D_z(h_q^{-1} \circ h_p) = L_q^{-1}\circ (D_{\exp_q^{-1}\circ \text{ exp}_p\circ L_p(z)}\text{exp}_q)^{-1}\circ D_{L_p(z)}\text{exp}_p\circ L_p.
		\end{equation}
		Fix an orthonormal local frame in a neighborhood $U$ of $p$ 
		so that 
		\begin{align*}
			U \times \mathbb{R}^d &\to Gl_d(\mathbb{R})\\
			(x,v) & \mapsto D_v \text{exp}_x
		\end{align*}
		is a well defined map.
		Since $M$ is a smooth Riemannian manifold, the above map is smooth.
		Since $Gl_d(\mathbb{R})$ is a smooth Lie group, then
		\begin{align*}
			U \times \mathbb{R}^d \times U \times \mathbb{R}^d &\to Gl_d(\mathbb{R})\\
			(x,v,y,u) &\mapsto (D_u \text{exp}_y)^{-1} \circ D_v \text{exp}_x
		\end{align*}
		is also smooth.
		Now since the above function maps to the identity if $x = y$ and $u = v = 0$, then there exists $C_1>0$ such that
		\begin{equation}
			\label{expholder}
			\|(D_u \text{exp}_y)^{-1} \circ D_v \text{exp}_x - Id\| \leq C_1 (d_M(x,y)+d_{\mathbb{R}^d}(v,0)+d_{\mathbb{R}^d}(u,0)).
		\end{equation}
		Therefore, writing $v = L_p(z)$ and $u=  \text{exp}_q^{-1}\circ \text{exp}_p(v)$, we may rewrite \eqref{changeofcoordinates1} as
		\begin{equation}
			\label{changeofcoordinates2}
			L_q^{-1}\circ (Id + E(p,q,z))\circ L_p = L_q^{-1} \circ L_p + L_q^{-1} \circ E(p,q,z) \circ L_p,
		\end{equation}
		with $\|E(p,q,z)\| \leq C_1(d_M(p,q) +|v|+|u|)$ by \eqref{expholder}.
		Now by (L3) of \cref{LyapCharts} and once again using the fact that $Gl_d(\mathbb{R})$ is a smooth Lie group,
		there exists $C_2>0$ such that for $x,y \in \mathcal{P}_{\tau}$,
		\begin{equation}
			\label{LHolder}
			\|L_y^{-1} \circ L_x - Id\| \leq C_2 d_M(x,y)^{\alpha_2}.
		\end{equation}
		By (L3) of \cref{LyapCharts},
		\begin{equation}
			\label{normv}
			|v| \leq \tau^{-1} |z|.
		\end{equation}
		Therefore,
		\begin{equation}
			\label{normu}
			|u| = d_M(q,\text{exp}_p(v)) \leq d_M(p,q) +d_M(p,\text{exp}_p(v)) \leq d_M(p,q) + \tau^{-1}|z|.
		\end{equation}
		Putting together \eqref{changeofcoordinates2},\eqref{LHolder},\eqref{normv},\eqref{normu} and (L3) of \cref{LyapCharts}, we get
		\begin{align*}
			\|D_z(h_q^{-1} \circ h_p) - Id\| \leq &
			C_2d_M(p,q)^{\alpha_2} \\
			&+ \tau^{-2}C_1(d_M(p,q)+\tau^{-1}|z| + d_M(p,q)+\tau^{-1}|z|),
		\end{align*}
		which implies the Lemma if $p,q \in U$.
		The Lemma follows by taking a finite cover of $M$ by open sets with orthonormal frames defined on them.
	\end{proof}
	
	\begin{corollary}
	\label{graphid}
		Let $\tau>0$ be small and $A>0$. 
		There exists $\varepsilon = \varepsilon(\tau,A)>0$ satisfying the following.
		Let $p,q \in \mathcal{P}_{\tau}$ with $d(p,q)<\varepsilon$, $0<\xi,r<\varepsilon$ and
		$\eta_p \colon \mathcal{N}^u_{\xi} \to \mathcal{N}^{cs}_r$ be a $C^{1+\gamma}$ function satisfying $\|D\eta_p\| \leq A$.
		Let $\eta_q \colon \mathcal{N}^u_{\xi'}(h_q^{-1}\circ h_p(p)) \to \mathbb{R}^{cs}$ be the $C^{1+\gamma}$ function whose graph is $h_q^{-1}\circ h_p(\text{graph}(\eta_p))$, for a possibly smaller $\xi'$.
		Then, there exists $K = K(A,\tau)$ such that
		\begin{equation*}
			\|D\eta_q\| \leq \|D\eta_p\| +K(d(p,q)^{\alpha_2}+\xi+r).
		\end{equation*}
		The same result holds for vertical submanifolds.
	\end{corollary}
	\begin{proof}
		The proof follows a standard graph transform argument.
		Let $\varphi = h_q^{-1} \circ h_p$,
		$\psi = h_p^{-1} \circ h_q$. 
		Write $\varphi = Id + E_{\varphi}$ and $\psi = Id + E_{\psi}$
		 with $E_{\varphi} = (E^u_{\varphi},E^{cs}_{\varphi})$ and $E_{\psi} = (E^u_{\psi},E^{cs}_{\psi})$.
		 
		 For fixed $A>0$, there exists $\varepsilon_0>0$ small enough such that if $\|D_zE_{\varphi}\|<\varepsilon_0$ for all $z \in \mathcal{N}^u_{\xi}\times\mathcal{N}^{cs}_r$, then $\eta_q$ is well defined.
		 Moreover, there exists $K'$ such that 
		 \begin{equation}
		 \label{boundKp}
		 	\|D\eta_q\| \leq K'.
		 \end{equation}
		 By \cref{changecoordinates}, there exists $\varepsilon>0$ such that if $d(p,q),r,\xi<\varepsilon$, then the above is satisfied. For $x \in \mathcal{N}_{\xi'}^u(\varphi(p))$,
		 $(x,\eta_q(x)) \in \varphi(\text{graph}(\eta_p))$.
		 Therefore,
		 \begin{equation}
		 \label{graphtransform}
		 	\psi(x,\eta_qx) =
		 	(x+E^u(x,\eta_qx),\eta_qx+E^{cs}_{\psi}(x,\eta_qx))
		 	\in \text{graph}(\eta_p).
		 \end{equation}
		 This means that
		 \begin{equation*}
		 	\eta_p(x+E^u_{\psi}(x,\eta_qx)) 
		 	= \eta_qx +E^{cs}_{\psi}(x,\eta_qx).
		 \end{equation*}
		 Substituting in \eqref{graphtransform}, applying $\varphi$ and looking at the second coordinate,
		 \begin{equation*}
		 	\eta_qx
		 	= \eta_p(x+E^u_{\psi}(x,\eta_qx))+E^{cs}_{\varphi}(\psi(x,\eta_qx)).
		 \end{equation*}
		 By \cref{changecoordinates} and \eqref{boundKp},
		 \begin{equation*}
		 \begin{split}
		 	\|D\eta_q\|
		 	\leq & \|D\eta_p\|(1+\|D E^{u}_{\psi}\|(1+K'))
		 	+ \|D E^{cs}_{\varphi}\|\|D\psi\|(1+K')\\
		 	\leq & \|D\eta_p\| + (1+K')(\|D\eta_p\|+\|D\psi\|)\|D E_{\varphi}^{cs}\|.
		 \end{split}
		 \end{equation*}
		 By \cref{changecoordinates}, $\|D\psi\| \leq 1+C(\varepsilon^{\alpha_2}+2\varepsilon)$.
		 Letting $K = (1+K')(1+C(\varepsilon^{\alpha_2}+2\varepsilon)+A)$ and using \cref{changecoordinates} proves the Corollary.
	\end{proof}

	For $x \in LyapReg(\delta)$,
	we extend $\tilde{f}_{x,\delta}$ to all $\mathbb{R}^d$
	by making it linear 
	outside of $\mathcal{N}_{2\mathfrak{r}_{\delta}(0)}$.
	For $n >0$, denote by 
	\begin{equation*}
		\tilde{f}^{(n)}_{x,\delta} = \tilde{f}_{f^{n-1}x,\delta} \circ ... \circ \tilde{f}_{x,\delta} \ \text{and} \  
		\tilde{f}^{(-n)}_{x,\delta} = (\tilde{f}_{f^{-n}x,\delta})^{-1} \circ ... \circ (\tilde{f}_{f^{-1}x,\delta})^{-1}.
	\end{equation*}

	\subsubsection{Unstable Manifolds}

	We now construct the unstable foliation on Lyapunov charts.
	Given $x \in M$, define
	\begin{equation*}
		\tilde{W}^u(x) = \left\{ y \in \mathbb{R}^d \ | \ \underset{n \rightarrow +\infty}{\limsup} \frac{1}{n} \log |\tilde{f}^{(-n)}_{x,\delta}(y)| < 0 \right\}.
	\end{equation*}
	The following standard Theorem is stated as in Lemma 2.8 of \cite{ExpMix}.
	\begin{theorem}[Unstable Manifold Theorem]
		\label{unstable}
		There are $C_0,\alpha_4>0$ such that, for $x \in LyapReg(\delta)$, 
		$\tilde{W}^u(x)$ is the graph of a $C^{1+\alpha_4}$ function
		$\eta^u = \eta^u_x \colon \mathbb{R}^u \to \mathbb{R}^{cs}$ such that
		\begin{enumerate}
			\label{unstablemanifold}
			\item $\eta^u(0) = 0$;
			\item $D_0\eta^u = 0$;
			\item $\|\eta^u\|_{C^{1+\alpha_4}} \leq C_0$.
		\end{enumerate}
		Moreover, for $z_1,z_2 \in \tilde{W}^u(x)$ and $n > 0 $,
		\begin{equation*}
			|\tilde{f}^{(-n)}_x(z_1) - \tilde{f}^{(-n)}_x(z_2)| 
			\leq e^{(-\lambda+\sqrt{\delta})n}|z_1-z_2|.
		\end{equation*}
		
	\end{theorem}

	Given $R>0$, we define
	\begin{equation*}
		\tilde{W}^u_{R}(x) = \text{graph}(\eta^u_x|_{\mathcal{N}^u_R(0)})
		\ \text{and} \ W^u_{R}(x) = h_x(\tilde{W}^u_R(x)).
	\end{equation*}
	We define the \textit{unstable manifold at x} by
	\begin{equation*}
		W^u(x) = W^u_{\mathfrak{r}_{\delta}(x)}(x),
	\end{equation*}
	which is often called the \textit{local} unstable manifold at $x$ in the literature.
	But, since we don't use the global unstable manifolds,
	we'll drop the \textit{local} in this paper.
	For $x \in \text{LyapReg}(\delta)$ and $y \in W^u(x)$ and $n \geq 0$, we have
	\begin{equation}
		\label{UnstableContraction}
		d(f^{-n}x,f^{-n}y) \leq \mathfrak{r}_{\delta}(x)^{-1}e^{(-\lambda+\sqrt{\delta})n}d(x,y).
	\end{equation}
	See equation (2.13) in \cite{ExpMix} and the paragraph right before it for more details, recalling that our unstable manifolds always have size at least $\mathfrak{r}_{\delta}(x)$.
	
	Now for $\tau>0$ small enough,
	if $x,y \in \mathcal{P}_{\tau}$ are close sufficiently close to each other,
	we can view $W^u(y)$ as the graph of some function in the Lyapunov chart at $x$.
	The following Lemma is a consequence \cref{graphid} and H\"older continuity of the unstable distribution on Pesin sets.

	\begin{lemma}
		\label{unstablegraph}
		Let $\tau>0$ be small. 
		There exists $\alpha_5>0, \varepsilon = \varepsilon(\tau)>0, C = C(\tau)>0$
		such that for $x,y \in \mathcal{P}_{\tau}$ with $d(x,y) < \varepsilon$,
		$h_x^{-1}(W^u(y))$ is the graph of a $C^{1+\alpha_5}$ function
		$\eta^u_{x,y} \colon \mathcal{N}^u_{\frac{\tau}{2}}(0)\subset \mathbb{R}^u \to \mathbb{R}^{cs}$ such that,
		\begin{equation*}
			\|\eta^u_{x,y}\|_{C^{1+\alpha_5}} \leq C.
		\end{equation*}
		Moreover if $y_1,y_2 \in \mathcal{P}_{\tau}$ are such that $d(x,y_1),d(x,y_2)< \varepsilon$
		and $z_1,z_2 \in \mathcal{N}^u_{\frac{\tau}{2}}(0)$, then
		\begin{equation*}
			\label{unstablefol}
			\|D_{z_1}\eta^u_{x,y_1} -D_{z_2}\eta^u_{x,y_2}\| 
			\leq C|(z_1,\eta^u_{x,y_1}(z_1))-(z_2,\eta^u_{x,y_2}(z_2))|^{\alpha_5}.
		\end{equation*}
	\end{lemma}
\subsubsection{Log H\"older continuity of the unstable Jacobian on Pesin sets}
\label{logholder}
	
	For $x \in \text{LyapReg}$ and $w \in W^u(x)$, let $J^u(w) = |\text{Jac}(D_wf|_{E^u(w)})|$.
	\begin{lemma}
	\label{JuHol}
	There exists $C>0$ and $\alpha_5>0$ such that for $\delta>0$ small enough, $x \in \text{LyapReg}(\delta)$ and $z \in W^u(x)$,
		\begin{equation*}
			|\log J^u(x)-\log J^u(z)|
			\leq C\mathfrak{r}_{\delta}(x)^{-2}d(x,z)^{\alpha_5}.
		\end{equation*}
	\end{lemma}
	\begin{proof}	
		Let $d$ be a metric on the Grassmanian bundle of $u$-dimensional subspaces $\text{Gr}_u(TM)$.
		Since $f$ is $C^{1+\alpha}$, then $Df$ induces an $\alpha$-H\"older map on $\bigwedge^uTM$.
		Therefore, the map $v_1\wedge ... \wedge v_u \mapsto |Dfv_1\wedge ... \wedge Dfv_u |$ is also H\"older continuous.
		 When restricted to norm $1$ elements, it defines a H\"older continuous map $\text{Gr}_u(TM) \to \mathbb{R}$ by $E \mapsto  |\text{Jac}(Df|_E)|$.
		Since $\log$ is Lipschitz when restricted to compact sets, there exists $C_1>0$ such that
		\begin{equation}
		\label{JacHol}
			|\log|\text{Jac}(Df|_E)| - \log|\text{Jac}(Df|_F)|| \leq C_1d(E,F)^{\alpha}.
		\end{equation}
		For $x \in \text{LyapReg}(\delta)$ and $z \in W^u(x)$, let $h_x^{-1}z = \overline{z} = (\overline{z}^u,\overline{z}^{cs}) $ and $\overline{E}^u(\overline{z}) = D_zh^{-1}_x(E^u(z))$.
		Then $\overline{E}^u(\overline{z}) = D_{\overline{z}^u}\eta^u_x(\mathbb{R}^u)$. By (2) and (3) of \cref{unstable},
		\begin{equation*}
			\|D_{\overline{z}^u}\eta^u_x\| \leq C_0|\overline{z}^u|^{\alpha_4}.
		\end{equation*}
		This implies that, when measuring distance between subspaces of $\mathbb{R}^d$ by their angle,
		\begin{equation}
		\label{JacHol2}
			d(\mathbb{R}^u,\overline{E}^u(z))
			\leq C_0|\overline{z}|^{\alpha_4}.
		\end{equation}
		Recall that $h_x = \text{exp}_x \circ L_{x,\delta}$ with $\text{max}\{\|L_{x,\delta}\|,\|L_{x,\delta}^{-1}\|\} \leq \mathfrak{r}^{-1}_{\delta}(x)$. In particular, $Dh_x$ induces a smooth map on $\text{domain}(h_x) \times \mathbb{R}^d \to TM$ which is $C_2 \mathfrak{r}_{\delta}^{-1}$-bi-Lipschitz on its image, where $C_2>0$ only depends on $M$. Therefore, there exists $C_3 >0$ depending only on $M$ and our choice of metric $d$ satisfying
		\begin{equation}
			\label{JacHol3}
			d(E^u(x),E^u(z)) \leq C_3\mathfrak{r}_{\delta}(x)^{-1} (|\overline{z}| + d(\mathbb{R}^u,\overline{E}^u(\overline{z}))).
		\end{equation}
		The Lemma now follows from $\text{Lip}(h^{-1}_x) \leq \mathfrak{r}_{\delta}(x)^{-1}$ together with \eqref{JacHol},\eqref{JacHol2} and \eqref{JacHol3}.
	\end{proof}		
	\begin{corollary}
	\label{JacHolderCor}
	For $\delta>0$ small enough, there exists $C' = C'(\delta)$ such that for $x \in \text{LyapReg}(\delta)$ and $z \in W^u(x)$,
		\begin{equation}
			\label{JacHolder}
			|\log \text{Jac}(D_xf^{-n}|_{E^u(x)})
			-\log \text{Jac}(D_zf^{-n}|_{E^u(z)})|
			\leq C' \mathfrak{r}_{\delta}(x)^{-(2+\alpha_5)}d(x,z)^{\alpha_5}.
		\end{equation}
		In particular, given $\varepsilon>0$ and $\tau>0$, there exists $\xi>0$ such that if $x \in \mathcal{P}_{\tau}$ and $z \in W^u(x)$ with $d(x,z)<\xi$, then
		\begin{equation}
		\label{JacHolder2}
			\text{Jac}(D_xf^{-n}|_{E^u(x)})
			\sim_{1+\varepsilon} \text{Jac}(D_zf^{-n}|_{E^u(z)}).
		\end{equation}
	\end{corollary}
	\begin{proof}
	By the chain rule,
		\begin{equation*}
			\text{Jac}(D_xf^{-n}|_{E^u(x)})
			= \prod_{k=1}^{n}\frac{1}{J^u(f^{-k}x)},
		\end{equation*}
		and similarly for $z$.
		Therefore, we should estimate
		\begin{equation*}
		\sum_{k=1}^n |\log J^u(f^{-k}x) - \log J^u(f^{-k}z)|.
		\end{equation*}
		By \cref{JuHol}, the above is bounded by
		\begin{equation*}
			C\sum_{k=1}^n \mathfrak{r}_{\delta}^{-2}(f^{-k}x)d(f^{-k}x,f^{-k}z)^{\alpha_5}.
		\end{equation*}
		Finally, we use \eqref{SlowExp} and \eqref{UnstableContraction} to bound the above by
		\begin{align*}
			C \sum_{k=1}^n \mathfrak{r}_{\delta}(x)^{-2}e^{2\delta k}
			\mathfrak{r}_{\delta}(x)^{-\alpha_5}e^{\alpha_5(-\lambda+\sqrt{\delta})k}d(x,z)^{\alpha_5} \\
			\leq C \left(\frac{e^{-\alpha_5\lambda+2\delta+\alpha_5\sqrt{\delta}}}{1-e^{-\alpha_5\lambda+2\delta+\alpha_5\sqrt{\delta}}}\right)
			\mathfrak{r}_{\delta}(x)^{-(2+\alpha_5)} d(x,z)^{\alpha_5},
		\end{align*}
		which shows \eqref{JacHolder}.
		\eqref{JacHolder2} is an immediate consequence of \eqref{JacHolder}.
	\end{proof}

	\subsubsection{Fake Center Stable Foliations}

	Here we construct fake center stable foliations similar to those in \cite{ExpMix}.
	In their paper, they pull back the $\mathbb{R}^{cs}$ foliation from time $n$
	through the maps $\tilde{f}^{(n)}$.
	In this paper we'll have to pull back a foliation which is close to $\mathbb{R}^{cs}$,
	but it's not exactly equal to it. Therefore, we'll redo their construction
	while controlling how the properties depend on the initial foliation that we're pulling back.

	Let $x \in LyapReg$, $n>0$ 
	and $\mathcal{F}$ a smooth foliation on $\mathbb{R}^d$ such that each leaf $\mathcal{F}(y)$ is the graph of a smooth function
	$\phi^{\mathcal{F}}_{y} \colon \mathbb{R}^{cs} \to \mathbb{R}^u$.
	The following definition is analogous to Definition 2.5 of \cite{ExpMix},
	with the differencce being that they take $\mathcal{F}$ to be the foliation by translations of $\mathbb{R}^{cs}$.
	\begin{definition}[Fake cs-foliation]
		Given $0<\delta<\delta_0$, $x \in LyapReg(\delta)$ and $n \geq 0$, define the foliation $\tilde{W}^{cs,n,\delta}_{\mathcal{F},x}$ on $\mathbb{R}^d$
		by pulling back $\mathcal{F}$ via $\tilde{f}^{(n)}_{x,\delta}$, that is,
		\begin{equation*}
			\tilde{W}^{cs,n,\delta}_{\mathcal{F},x}(y) = (\tilde{f}^{(n)}_{x,\delta})^{-1}\mathcal{F}(\tilde{f}^{(n)}_{x,\delta}(y)).
		\end{equation*}
		
	\end{definition}

	The following lemma, similar to Lemma 2.6 of \cite{ExpMix}, allows us to view the leaves of the above foliation as graphs of well behaved functions:
	\begin{lemma}
		\label{fakecs}
		Suppose that $\underset{w\in \mathbb{R}^{cs},y \in \mathbb{R}^d}{\sup}\|D_z\phi^{\mathcal{F}}_y\|\leq 1$.
		There exists $\alpha_3,\delta_f>0$ such that 
		for $0<\delta<\delta_f$, $x \in LyapReg(\delta)$ and $y \in \mathbb{R}^d$,
		there exists a $C^{1+\alpha_3}$ function $\eta^{cs,n,\delta}_{\mathcal{F},x,y} \colon \mathbb{R}^{cs} \to \mathbb{R}^u$ satisfying
		\begin{enumerate}
			\item $\tilde{W}^{cs,n,\delta}_{\mathcal{F},x}(y) = \text{graph}(\eta^{cs,n,\delta}_{\mathcal{F},x,y})$;
			\item $\underset{w \in \mathbb{R}^{cs}}{\sup}\|D_w\eta^{cs,n,\delta}_{\mathcal{F},x,y}\| \leq \frac{2\delta}{1-e^{-\lambda+\sqrt{\delta}}}
				+ e^{(-\lambda+\sqrt{\delta})n}\|D\phi^{\mathcal{F}}_{\tilde{f}^{(n)}_{x,\delta}(y)}\|_{C^0}$.
		\end{enumerate}
		In particular,
	there exists $n_{\delta} >0$ such that
	if $\mathcal{F}$ is $\left(\alpha_3,\frac{3\delta}{1 - e^{-\lambda + \sqrt{\delta}}}\right)$-mostly vertical
	and $n>n_{\delta}$, then $\tilde{W}^{cs,n,\delta}_{\mathcal{F},x}$ is also $\left(\alpha_3,\frac{3\delta}{1 - e^{-\lambda + \sqrt{\delta}}}\right)$-mostly vertical.
	\end{lemma}

	The proof of this lemma relies on the following Lemma (Lemma A.5 of \cite{ExpMix}):
	\begin{lemma}
		\label{lemmaA5}
		Let $x \in LyapReg(\delta)$, with $\delta$ sufficiently small.
		Let $z \in \mathbb{R}^d$ and $L \colon \mathbb{R}^{cs} \to \mathbb{R}^u$ be a linear map with $\|L\| \leq 1$.
		Define the linear map $\Gamma^{cs}_{x,z}(L) \colon \mathbb{R}^{cs} \to \mathbb{R}^u$ by
		\begin{equation*}
			\text{graph}(\Gamma^{cs}_{x,z}(L)) = (D_z\tilde{f}_{x,\delta})^{-1} \text{graph}(L).
		\end{equation*}
		Then,
		\begin{enumerate}
			\item $\|\Gamma^{cs}_{x,z}(L)\| \leq e^{-\lambda + \sqrt{\delta}}\|L\| + 2\delta\min\{1,|z|^{\alpha_2}\}$;
			\item $\|\Gamma^{cs}_{x,z_1}(L_1)-\Gamma^{cs}_{x,z_2}(L_2)\| \leq e^{-\lambda + \sqrt{\delta}}\|L_1-L_2\|
				+ 6\delta|z_1-z_2|^{\alpha_2}$.
		\end{enumerate}
	\end{lemma}
	\begin{proof}[Proof of \cref{fakecs}]
		The proof will closely follow the proof of Lemma 2.6 of \cite{ExpMix}.
		Fix $x \in LyapReg(\delta)$, $z \in \mathbb{R}^{cs}$, $y \in \mathbb{R}^d$ and $n\geq 0$.
		Define inductively
		\begin{equation*}
			\eta_0 = \phi^{\mathcal{F}}_{\tilde{f}^{(n)}_{x,\delta}(y)}
		\end{equation*}
		and for $0<k \leq n$,
		\begin{equation*}
			\text{graph}(\eta_k) = (\tilde{f}^{(k)}_{f^{n-k}x,\delta})^{-1}\text{graph}(\eta_{k-1}).
		\end{equation*}
		We need to show that the right hand side is always transverse to the $\mathbb{R}^{cs}$ direction to make it a well defined graph of a function.
		We want to therefore estimate $D_w\eta_n$, which will prove both items in the Lemma.
		Let
		\begin{equation*}
			z_k = \tilde{f}^{(n-k)}_{x,\delta}(\eta_n(w),w) \ \text{and} \
			w_k = \pi^{cs}(z_k).
		\end{equation*}
		Let 
		\begin{equation*}
			L_k = D_{w_k} \eta_k \colon \mathbb{R}^{cs} \to \mathbb{R}^u.
		\end{equation*}
		Following the notation from \cref{lemmaA5},
		\begin{equation*}
			L_{k+1} = \Gamma^{cs}_{f^{n-(k+1)}x,z_{k+1}}(L_k).
		\end{equation*}
		Applying (1) of \cref{lemmaA5} inductively,
		\begin{equation*}
			\begin{split}
				\|L_k\| \leq& 2\delta\sum_{i=0}^{k-1} e^{(-\lambda+\sqrt{\delta})i}
				+ \|L_0\|e^{(-\lambda+\sqrt{\delta})k} \\
				& \leq \frac{2\delta}{1-e^{-\lambda+\sqrt{\delta}}} 
				+ e^{(-\lambda+\sqrt{\delta})k}\|L_0\|
			\end{split}.
		\end{equation*}
		In particular, if $\frac{2\delta}{1-e^{-\lambda + \sqrt{\delta}}}< 1-e^{-\lambda+\sqrt{\delta}}$
		and $\|L_0\| < 1$,
		then $\|L_k\| <1$.
		Therefore, we may still inductively apply \cref{lemmaA5} to get the above for all $k\leq n$.
		By construction, $L_n = D_w\eta_n$ 
		and $L_0 = D_{w_0}\phi^{\mathcal{F}}_{\tilde{f}^{(n)}_{x,\delta}(y)}$.
		These prove the Lemma since $w$ was chosen arbitrarily.
	\end{proof}

	A crucial property for the fake cs-foliations is their subexponential growth up to time $n$,
	which follows from the following (Lemma A.6 of \cite{ExpMix}).
	\begin{lemma}
		\label{subexpchart}
		Let $x \in LyapReg(\delta)$ and $z_1,z_2 \in \mathbb{R}^d$.
		Let $K = 2 \max\{\|Df\|_{C^0},\|Df^{-1}\|_{C^0}\}$.
		Then
		\begin{equation*}
			|\tilde{f}_{x,\delta}(z_1)-\tilde{f}_{x,\delta}(z_2)|
			\leq 2K|z_1^u - z_2^u|+ e^{2\delta}|z_1^{cs}- z_2^{cs}|.
		\end{equation*}
		Moreover, if $|z_1^u-z_2^u| \leq \frac{3\delta}{1-e^{-\lambda + \sqrt{\delta}}}|z_1^{cs}-z_2^{cs}|$,
		then
		\begin{equation}
			\label{subexp}
			|\tilde{f}_{x,\delta}(z_1) - \tilde{f}_{x,\delta}(z_2)|
			\leq e^{\sqrt{\delta}}|z_1-z_2|.
		\end{equation}
	\end{lemma}

	\begin{lemma}
		\label{subexpM}
		Suppose that $x \in \mathcal{P}_{\tau}$ and $\mathcal{F}$ is such that
		\begin{equation*}
			\underset{w \in \mathbb{R}^{cs}, y \in \mathbb{R}^d}{\sup} \|D_w \phi^{\mathcal{F}}_y\|
			<\frac{\delta}{e^{-\lambda+\sqrt{\delta}}(1-e^{-\lambda+\sqrt{\delta}})}.
		\end{equation*}
		Suppose that $w_1,w_2 \in \mathbb{R}^{cs}$ and $\overline{z} \in \tilde{W}^u(x)$
		satisfy
		\begin{equation}
			\label{fakecscondition}
			\begin{split}
				|w_1|,|w_2| < r_n \ \text{and} \ 
				|\tilde{f}^{(n)}_{x,\delta}\overline{z}| < \xi_n,\\
				\text{with} \ r_n \ll e^{-n(\delta+\sqrt{\delta})} \
				\text{and} \ \xi_n \ll e^{-\delta n}.
			\end{split}
		\end{equation}
		Let $\overline{z}_i = \eta^{cs,n,\delta}_{\mathcal{F},x,\overline{z}}(w_i)$,
		and $z_i = h_{x,\delta}(\overline{z}_i)$, for $i=1,2$.
		Then for $0 \leq k \leq n$ and $i=1,2$,
        \begin{equation}
            \label{fcommute}
            \tilde{f}^{(k)}_{x,\delta}(\overline{z}_i,w_i) 
            =  h_{f^kx,\delta}^{-1} \circ f^k \circ h_{x,\delta}(\overline{z}_i,w_i).
        \end{equation}
        In particular,
		\begin{equation}
			\label{subexpM2}
			d(f^k(z_1),f^k(z_2)) \leq \tau^{-2}e^{2\sqrt{\delta}k}d(z_1,z_2).
		\end{equation}
	\end{lemma}
	\begin{proof}
		By \cref{fakecs} and the condition on $\mathcal{F}$,
		we have for all $0 \leq k \leq n$
		\begin{equation*}
			|\pi^u(\tilde{f}^{(k)}_{x,\delta}(\overline{z}_1)) 
			- \pi^u(\tilde{f}^{(k)}_{x,\delta}(\overline{z}_2))|
			\leq \frac{3\delta}{1-e^{-\lambda+\sqrt{\delta}}}
			|\pi^{cs}(\tilde{f}^{(k)}_{x,\delta}(\overline{z}_1))
			- \pi^{cs}(\tilde{f}^{(k)}_{x,\delta}(\overline{z}_2))|.
		\end{equation*}
		Therefore, we may inductively apply \cref{subexpchart} obtaining for $0 \leq k < n$,
		\begin{equation}
			\label{subexpcharts1}
			|\tilde{f}^{(k)}_{x,\delta}(\overline{z}_1) - \tilde{f}^{(k)}_{x,\delta}(\overline{z}_2)|
			\leq e^{k\sqrt{\delta}}|\overline{z}_1 - \overline{z}_2|.
		\end{equation}
		Since $w_i$ are arbitrary,
		we have the same estimate picking some $w \in \mathbb{R}^{cs}$ such that $\overline{z} = \eta^{cs,n,\delta}_{\mathcal{F},x,\overline{z}}(w)$.
		That is, for $i=1,2$
		\begin{equation}
			\label{subexpchartsz}
			|\tilde{f}^{(k)}_{x,\delta}(\overline{z}_i) - \tilde{f}^{(k)}_{x,\delta}(\overline{z})|
			\leq e^{k\sqrt{\delta}}|\overline{z}_i-\overline{z}|.
		\end{equation}
		Recall that we extended $\tilde{f}_{x,\delta}$ to the entirety of $\mathbb{R}^d$.
		Therefore, the information from the dynamics of $f$ can only be obtained in its original domain.
		We consequently need for $i=1,2$ and $0\leq k \leq n$,
		\begin{equation*}
			f^k(z_i) \in h_{f^kx,\delta}(\text{domain}(\tilde{f}_{f^kx,\delta})).
		\end{equation*}
		By (3) of \cref{LyapCharts}
		and \eqref{SlowExp}, it's enough to show that for every $0 \leq k \leq n$,
		\begin{equation}
			\label{insidedomain}
			|\tilde{f}^{(k)}_{x,\delta}(\overline{z}_i)| < e^{-\delta k}\tau.
		\end{equation}
		By \cref{unstable} and \eqref{subexpchartsz},
		\begin{equation*}
			\begin{split}
				\label{subexpMpf1}
				|\tilde{f}^{(k)}_{x,\delta}(\overline{z}_i)|
				\leq & |\tilde{f}^{(k)}_{x,\delta}(\overline{z}_i) 
				-\tilde{f}^{(k)}_{x,\delta}(\overline{z})|
				+|\tilde{f}^{(k)}_{x,\delta}(\overline{z})| \\
				\leq &
				e^{k\sqrt{\delta}}|\overline{z}_i - \overline{z}|
				+ e^{(-\lambda+\sqrt{\delta})(n-k)}\xi_n.
			\end{split}
		\end{equation*}
		We may estimate the first term above by the following
		\begin{equation*}
			\begin{split}
				|\overline{z}-\overline{z}_i|
				\leq &
				|\eta^{cs,n,\delta}_{\mathcal{F},x,\overline{z}}(w_i)-\eta^{cs,n,\delta}_{\mathcal{F},x,\overline{z}}(w)|+|w_i-w| \\
				\leq &
				\left(\frac{3\delta}{1-e^{-\lambda+\sqrt{\delta}}}+1\right)
				|w_i-w| \\
				\leq &
				\left(\frac{3\delta}{1-e^{-\lambda+\sqrt{\delta}}}+1\right)
				(|w_i|+|w|) \\
				\leq &
				\left(\frac{3\delta}{1-e^{-\lambda+\sqrt{\delta}}}+1\right)
				(r_n+e^{(-\lambda+\sqrt{\delta})n}\xi_n).
			\end{split}
		\end{equation*}
		By \eqref{fakecscondition}, there exists $n_0>0$ such that
		for $n>n_0$ and $0\leq k \leq n$, 
		\begin{equation*}
			e^{k(\sqrt{\delta}+\delta)}
			\left(\frac{3\delta}{1-e^{-\lambda+\sqrt{\delta}}}+1\right)
				(r_n+e^{(-\lambda+\sqrt{\delta})n}\xi_n)
				+ e^{(-\lambda+\sqrt{\delta})(n-k)}e^{\delta k}\xi_n < \tau.
		\end{equation*}
		This implies \eqref{insidedomain} for all $0 \leq k \leq n$.
		Therefore, we may write for each such $k$,
		\begin{equation*}
			\tilde{f}^{(k)}_{x,\delta} 
			= h_{f^kx,\delta}^{-1} \circ f^k \circ h_{x,\delta}.
		\end{equation*}
		Applying \eqref{subexpcharts1},
		(2) of \cref{LyapCharts}
		and \eqref{SlowExp},
		\begin{equation*}
			\begin{split}
				d(f^kz_1,f^kz_2)
				\leq & \tau^{-1}e^{k\delta}
				|\tilde{f}^{(k)}_{x,\delta}(h_{x,\delta}^{-1}(z_1))
				- \tilde{f}^{(k)}_{x,\delta}(h_{x,\delta}^{-1}(z_2))| \\
				\leq & \tau^{-1} e^{k(\delta+\sqrt{\delta})}
				|h_{x,\delta}^{-1}(z_1) 
				- h_{x,\delta}^{-1}(z_2)| \\
				\leq & \tau^{-2}e^{(\delta+\sqrt{\delta})k}d(z_1,z_2).
			\end{split}
		\end{equation*}
		This proves the Lemma.
	\end{proof}
	The following estimate is essentially contained in the proof of Proposition 6.4 in \cite{ExpMix}.
    \begin{lemma}
    \label{UsefulJac}
    There exists $\alpha_7>0$ satisfying the following:
    Let $x \in \text{LyapReg}(\delta)$,
    $\overline{z},\overline{z}' \in \mathbb{R}^d$ 
    and for $k\geq0$, denote by $\overline{z}_k = \tilde{f}^{(k)}_x(\overline{z})$
    and $\overline{z}'_k = \tilde{f}^{(k)}_x(\overline{z}')$.
    Let $n \in \mathbb{N}$ and suppose that for $0 \leq k \leq n$ we have 
    	$|\overline{z}_k|,|\overline{z}'_k| < 1$ and $\overline{z}_k,\overline{z}'_k \in \text{domain}(\tilde{f}_{f^kx})$.
    Let $\overline{W},\overline{W}' \subset \mathbb{R}^d$ be horizontal submanifolds with $\overline{z} \in \overline{W}$ and $\overline{z}' \in \overline{W}'$.
    Let $L,L' \colon \mathbb{R}^u \to \mathbb{R}^{cs}$ be linear maps such that $\text{graph}(L) = T_{\overline{z}}\overline{W}$ and $\text{graph}(L') = T_{\overline{z}'}\overline{W}'$.
    Then,
    \begin{align*}
        &\left|\log \text{Jac}(D_{\overline{z}}\tilde{f}_x^{(n)}|_{T_{\overline{z}}\overline{W}}) - 
        \log \text{Jac}(D_{\overline{z}'}\tilde{f}_x^{(n)}|_{T_{\overline{z}'}\overline{W}'})\right| \\
        &\leq C\sum_{k=0}^{n-1}\delta|\overline{z}_k-\overline{z}_k'|^{\alpha}
        + e^{(-\lambda+\sqrt{\delta})k}\|L-L'\|
        +6\delta \sum_{i=1}^k e^{(-\lambda+\sqrt{\delta})(k-i)}|\overline{z}_i-\overline{z}'_i|^{\alpha_7}
    \end{align*}
	\end{lemma}
	
	\begin{proof}
		Denote by $\overline{W}_k = \tilde{f}^{(k)}_x\overline{W}$ and $\overline{W}'_k = \tilde{f}^{(k)}_x\overline{W}'_k$.
		Since
		$\overline{z}_k,\overline{z}'_k \in \text{domain}(\tilde{f}_{f^kx})$,
		then 
		$\tilde{f}^{(n)}_x(\overline{z}) = \tilde{f}_{f^{n-1}x}(\overline{z}_{n-1}) \circ ...\circ \tilde{f}_x(\overline{z})$ and similarly for $\overline{z}'$.
		By the chain rule, we wish to estimate
		\begin{equation}		
		\label{Jacest1}	
        \sum_{k=0}^{n-1}
        \left| \log \text{Jac}(D_{\overline{z}_k} \tilde{f}_{x_k}|_{T_{\overline{z}_k}\overline{W}_k}) - \log \text{Jac}(D_{\overline{z}'_k} \tilde{f}_{x_k}|_{T_{\overline{z}'_k}\overline{W}'_k})\right|
		\end{equation}
		For $x \in \text{LyapReg}(\delta)$ and $z \in \mathcal{N}_{\mathfrak{r}(x)}(0) \subset \mathbb{R}^d$ by (A.1) of Lemma A.3 of \cite{ExpMix} there exists $K>0$ such that
		\begin{equation*}
			\|D_0\tilde{f}_x\|,\|D_0\tilde{f}_x\| \leq K.
		\end{equation*}	
		 By (5) of \cref{LyapCharts},
		\begin{equation*}
			\|D_z\tilde{f}_x\|, \|D_z\tilde{f}_x\|^{-1} 
			\leq \delta|z|^{\alpha_2} + K \leq K+1 = K'.
		\end{equation*}
		Let $L_k,L_k' \colon \mathbb{R}^u \to \mathbb{R}^{cs}$ be linear maps such that $\text{graph}(L_k) = T_{\overline{z}_k}\overline{W}$ and $\text{graph}(L'_k) = T_{\overline{z}_k}\overline{W}'$.
		By Lemma 2.7 of \cite{ExpMix}, the $k$-th term in \eqref{Jacest1} is bounded by
		\begin{equation}			
		\label{Jacest2}
		N(K')(\|D_{\overline{z}_k}\tilde{f}_{x_k} - D_{\overline{z}'_k}\tilde{f}_{x_k}\|
			+ \|L_k - L'_k\|).
		\end{equation}
		The first term of \eqref{Jacest2} is bounded using (5) of \cref{LyapCharts}
		\begin{equation*}
			\|D_{\overline{z}_k}\tilde{f}_{x_k} -  D_{\overline{z}'_k}\tilde{f}_{x_k}\|
			\leq \delta |\overline{z}_k-\overline{z}'_k|^{\alpha}.
		\end{equation*}
		For the second term of \eqref{Jacest2}, we use Lemma A.10 of \cite{ExpMix} to obtain
		\begin{equation*}
			\|L_k-L'_k\|
			\leq e^{(-\lambda+\sqrt{\delta})k}\|L-L'\|
			+ 6\delta
			\sum_{i=1}^k e^{(-\lambda+\sqrt{\delta})(k-i)}|\overline{z}_i-\overline{z}'_i|^{\alpha_7}.
		\end{equation*}
		Putting these estimates together proves the Lemma.
	\end{proof}

	\subsection{SRB measures}
    \label{SRBsection}
    
    We say that a $\mu$-measurable partition of $M$ $\{W(x)\}_{x \in M}$ is \emph{subordinate to the unstable foliation} if for $\mu$-almost every $x \in M$,
    \begin{enumerate}
    	\item $W(x) \subset W^u(x)$;
    	\item $W(x)$ contains a neighborhood of $x$ in the submanifold topology of $W^u(x)$.
    \end{enumerate}

	Let $W(x)$ be a $\mu$-measurable partition
	subordinate to the unstable foliation.
	Let $\{\mu_{W(x)}\}_{x \in M}$ be a disintegration of $\mu$ with respect to the $W$ partition.

	We say that an invariant measure $\mu$ with at least one positive Lyapunov exponent is an \textit{SRB measure} (Sinai-Ruelle-Bowen) if for $\mu$-almost every $x$,
	$\mu_{W(x)} \ll m^u_x$,
	where $m^u_x$ is the measure induced by the Riemannian metric restricted to the immersed submanifold $W^u(x)$.
	In this case, we may write for $\mu$-almost every $x \in M$ and $A \subset M$ Borel measurable,
	\begin{equation}
		\label{SRB}
		\mu_{W(x)}(A) = \int_{A\cap W(x)} \rho_x(y) dm^u_x(y).
	\end{equation}
	notice that $\rho_x$ depends on the partition $W$.
	Despite that, it is known that for $y,z \in W(x)$,
	\begin{equation}
		\label{SRBdensity}
		\Delta_x(y,z) \coloneqq \frac{\rho_x(y)}{\rho_x(z)} = \frac{\prod_{n \geq 1}J^u(f^{-n}(z))}{\prod_{n \geq 1}J^u(f^{-n}(y))}.
	\end{equation}
    See (6.1) of \cite{LY1} for more details.

	\begin{lemma}
		\label{SRBdensityHol}
		Suppose $\mu$ is an SRB measure. 
		Let $\{W(x)\}_{x \in M}$ be a u-subordinate partition of $M$ and $\{\mu_{W(x)}\}_{x \in M}$ a disintegration with respect to that partition.
		Then for any $\tau >0$ small enough and $ \varepsilon >0$, 
		there exists $ \varepsilon_0 = \varepsilon_0( \varepsilon,\tau)>0$ such that
		for $\mu$-almost every $x \in \mathcal{P}_{\tau}$, if $\text{diam}(W(x)) < \varepsilon_0$,
		then for any $y \in W(x)$
		\begin{equation*}
			\rho_x(y) \sim_{(1+ \varepsilon)^{\frac{1}{2}}} \frac{1}{m^u_x(W(x))}.
		\end{equation*}
		In particular, for $A \subset W(x)$,
		\begin{equation*}
			\mu_{W(x)}(A) \sim_{1+ \varepsilon} \frac{m^u_x(A)}{m^u_x(W(x))}.
		\end{equation*}
	\end{lemma}
	\begin{proof}
	By Theorem 13.1.2 of \cite{BarPes}, $\Delta(x,y)$ is H\"older continuous for $x \in \mathcal{P}_{\tau}$ and $y \in W^u(x)$.
	Therefore, there exists $C_{\tau}>0$ and $\alpha$ such that
	\begin{equation*}
		\left|1-\frac{\rho_x(y)}{\rho_x(x)}\right|
		= \left| \frac{\rho_x(x)}{\rho_x(x)}-\frac{\rho_x(y)}{\rho_x(x)}\right|
		\leq C_{\tau}d(x,y)^{\alpha} \leq C_{\tau}\text{diam}(W(x))^{\alpha}.
	\end{equation*}
		Making $\text{diam}(W(x))$ small enough, we have for $y \in W(x)$,
		\begin{equation*}
			\rho_x(y) \sim_{(1+ \varepsilon)^{\frac{1}{2}}} \rho_x(x).
		\end{equation*}
		Since $\{ \mu_{W(x)} \}_{x \in M}$ is a disintegration of $\mu$, then
		\begin{equation*}
			\label{isprob}
			1 = \mu_{W(x)}(W(x)) = \int_{W(x)} \rho_x(y) dm^u_x(y).
		\end{equation*}
		Since $\rho_x(y)$ is continuous for $y \in W^u(x)$, there must exist $z \in W(x)$ such that $\rho_x(z) = \frac{1}{m^u_x(W(x))}$.
		In particular,
		\begin{equation*}
			\frac{1}{m^u_x(W(x))} = \rho_x(z) 
			\sim_{(1+ \varepsilon)^{\frac{1}{2}}} \rho_x(x) 
			\sim_{(1+ \varepsilon)^{\frac{1}{2}}} \rho_x(y),
		\end{equation*}
		which proves the Lemma.
	\end{proof}

	\section{Constructing good boxes}
	\label{GoodBox}
    
	Let $f \colon M \to M$ be a $C^{1+\alpha}$ diffeomorphism and $\mu$ be an ergodic invariant probability measure with at least one positive Lyapunov exponent.
    Let $m$ be normalized volume of $M$.
    Recall that for $z$ in a full $\mu$-measure set of $M$, $h_z$ denotes the Lyapunov chart around $z$, constructed in \cref{LyapCharts}.
    
	We will construct a disjoint collection of exponentially small boxes covering a big portion of the space and such that their indicator functions are well $L_1$-approximated by Lipschitz functions with controlled Lipschitz norm ((P3) and (P4) respectively in \cref{goodbox}).
    
    If $\mu$ is an SRB measure, then they also present an approximate local product structure with respect to foliations in the center stable direction with good holonomies (see (P8) in \cref{goodbox}).
    These foliation will be constructed in \cref{fakecsconstructionsec}.

	\begin{lemma}[Cover by Good Boxes]
		\label{goodbox}
		Let $\ep>0$ and sequences $l_n,r_n,\xi_n$ satisfying 
        \begin{equation*}
            \xi^{1+\alpha}_n \ll l_n \ll r_n \ll \xi_n.
        \end{equation*}
        There exists $n_1 = n_1(l,r,\xi,\ep)$ and $\tau = \tau(\ep)$ such that for each $n>n_1$ we can find a family of disjoint sets $\{C_i\}_{i\in I_n}$ satisfying the following properties
		\begin{itemize}
			\item[(P1)] $C_i = h_{z_i}(\hat{C}_i)$ where $\hat{C}_i = \mathcal{N}_{(\xi,r)}(c_i)$ with $c_i \in \mathcal{N}_{2\tau^{-2}\xi}(0)$, for some $z_i \in \mathcal{P}_{\tau}$;
			\item[(P2)] There exists $D = D( \varepsilon)>0$ such that for all $i \in I_n$, $D\xi^{2d} \leq \mu(C_i)$;
			\item[(P3)] $\sum_i \mu(C_i) > 1-\ep$;
			\item[(P4)] $\hat{\mu}_i(\partial_l \hat{C}_i) \leq \varepsilon \hat{\mu}_i(\hat{C}_i)$ where $\hat{\mu}_i = (h^{-1}_{z_i})_* \mu$;
			\item[(P5)] $\mu(C_i \cap \mathcal{P}^c_{\tau}) \leq \varepsilon \mu(C_i)$;
			\item[(P6)] If $z \in h_{z_i}(\hat{C}_i\setminus \partial_l \hat{C}_i) \cap \mathcal{P}_{\tau}$, 
            then $h^{-1}_{z_i}(W^u_{\tau}(z)\cap C_i)$ is the graph of a $C^{1+\alpha}$ function $\eta^u_{z} \colon \mathbb{R}^u \to \mathbb{R}^{cs}$;
			\item[(P7)] Let $T^{cs}_i = h_{z_i}((c_i+\mathbb{R}^{cs})\cap \hat{C}_i)$
            and define
            \begin{equation*}
                T_i = T_i^{cs} \cap \underset{z \in h_{z_i}(\hat{C}_i\setminus\partial_l\hat{C}_i)\cap \mathcal{P}_{\tau}}{\bigcup} W^u_{\tau}(z).
            \end{equation*}
            For $x \in T_i$, there exists $z \in h_{z_i}(\hat{C}_i\setminus \partial_l \hat{C}_i) \cap \mathcal{P}_{\tau}$ such that
            $x \in W^u_{\tau}(z)$.
            Denote by $W_i(x) = W^u_{\tau}(z) \cap C_i$.
            Let
            \begin{equation*}
                \tilde{T_i} = \bigcup_{x \in T_i} W_i(x).
            \end{equation*}
            Then, $\mu(C_i \cap \tilde{T_i}^c) \leq \ep \mu(C_i)$;
			\item[(P8)] 
            Let $\{\mu_{W_i(x)}\}_{x \in T_i,i \in I_n}$ be a disintegration of $\mu$ with respect to the measurable partition $\{W_i(x)\}_{x \in T_i,i \in I_n}$.
            Let $\nu_i$ be the measure on $T_i$ induced by that partition, that is,
            \begin{equation*}
                \nu_i(B) = \mu\left(\bigcup_{x \in B}W_i(x)\right).
            \end{equation*}
			If $\mu$ is an SRB measure,
			and $\mathcal{F}^{cs}_i$ is a family of submanifolds $\mathcal{F}^{cs}_i(y) \subset C_i$, for $y \in Y_i$ some indexing set such that
			\begin{itemize}
				\item[(F1)] The induced foliation $\tilde{\mathcal{F}}^{cs}_i$ on $\mathcal{N}^u_{\xi}(c_j^u)\times \mathcal{N}^{cs}_r(c_j^{cs})$
            is made up of $(\alpha,1)$-mostly vertical leaves;
            \item[(F2)] There exists $\tilde{A} = \tilde{A}(\varepsilon,\tau)>0$ such that if $W_1,W_2 \subset C_j$
            $(\alpha,1)$-mostly horizontal submanifolds satisfying $\|D\eta_{W_1}\|,\|D\eta_{W_2}\| \leq \tilde{A}$, then denoting by $\pi^{\tilde{\mathcal{F}}}_{W_1,W_2}$
            the holonomy along $\tilde{\mathcal{F}}^{cs}_i$ from $W_1$ to $W_2$, for $x \in W_1 \cap \tilde{\mathcal{F}}$,
        \begin{equation*}
            \text{Jac}_x(\pi^{\tilde{\mathcal{F}}}_{W_1,W_2}) \sim_{1 + \varepsilon} 1.
        \end{equation*}
        \end{itemize}
			Then, for $A \subset C_i $ $\mathcal{F}^{cs}_i$-saturated and $\nu_i$-almost every $x \in T_i$, we have 
			
			\begin{equation*}
				\frac{\mu(A\cap\tilde{T}_i)}{\mu(\tilde{T}_i)} \sim_{1+\varepsilon} 
				\mu_{W(x)}(A)
			\end{equation*}
            and
            \begin{equation*}
                \frac{m(A)}{m(C_i)}
                \sim_{1+ \varepsilon} \mu_{W(x)}(A).
            \end{equation*}
			\end{itemize}

	\end{lemma}
	
	The remainder of this section will prove this lemma.
	
	\subsection{D-Nice partitions}

	\begin{definition}[D-Nice Partition]
		A \emph{nice partition} $\mathcal{Q}$ is a finite collection of disjoint open sets $ \{Q_1,...,Q_N\}$ such that for each $i$, $\partial Q_i$ is a finite union of embedded closed disks and $\bigcup_i \bar{Q_i} = M$. We denote by $\partial \mathcal{Q} = \bigcup_i \partial Q_i$.
		It is a \emph{$D$-Nice partition} if it also satisfies
		\[
			\mu(\partial_t \mathcal{Q}) \leq Dt
		\]
	\end{definition}
	
    The following lemma constructs $D$-Nice partitions from perturbations of nice partitions.
    A remark following Lemma 4.1 of \cite{OW98} establishes the existence of partitions with arbitrarily small atoms satisfying the above property of boundary control in any metric space.
    However, they are constructed differently than ours,
    as they aren't realized as perturbations of existing partitions.
    A detailed statement and proof in their context can be found in Corollary 5.2 of\cite{BernSubadd}.

	\begin{lemma}[Existence of $D$-nice partitions]
	\label{Dnicelemma}
		Let ${\mathcal{Q}} = \{Q_1,...,Q_N\}$ be a nice partition.  Given $\ep>0$, we may perturb each submanifold in $\partial{\mathcal{Q}}$ inside an $\ep$-neighborhood of itself such that the new collection of submanifold form a $D$-nice partition, for some $D>0$.
	\end{lemma}
	\begin{proof}
		$\partial \mathcal{Q}$ is a finite union of codimension 1 submanifolds, enlarge them a little so that they intersect transversally and let $W$ be one of them. 
        Define the function $\varphi(t) = \mu(\mathcal{N}_t(W))$.
        It is a monotone map on the interval $[0,diam M]$, therefore, it is differentiable Lebesgue almost everywhere.
        Let $\varepsilon$ be small enough so that $\psi_W(x) = d(x,W)$ is smooth on $\mathcal{N}_{ \varepsilon}(W)\setminus W$.
        Let $t_0 < \ep$ be such that $\varphi$ is differentiable at $t_0$ and let $C = \varphi'(t_0)$. Let $W' = \psi_W^{-1}(t_0) $, which is a submanifold by smoothness of $\psi_W$.
        Then, for $h>0$ small enough,
		\[
		\mu(\mathcal{N}_h(W')) \leq \varphi(t_0+h) - \varphi(t_0-h) 
		\]
		\[
		\leq \left[\frac{\varphi(t_0+h) - \varphi(t_0)}{h} + \frac{\varphi(t_0) - \varphi(t_0-h)}{h}\right] h \leq 2Ch.
		\] 
        Therefore,
		\[
		D_{W} \coloneqq \underset{h>0}{\sup} \frac{\mu(\mathcal{N}_h(W'))}{h} < \infty.
		\]
		Do the same procedure on each submanifold composing $\partial \mathcal{Q}$ and take $D = \sum_W D_W$. The new submanifolds form the boundary of a $D$-nice partition, if $ \varepsilon>0$ is small enough to guarantee that they still intersect transversally.
	\end{proof}

	\subsection{The shifting trick}

	The following lemma will be used many times. We therefore state it in a more general way.
	
	\begin{lemma}
	\label{shift}
		Let $A \subset \R^d$ be an open set with finite volume, $\nu$ a probability measure supported on $\bar{A}$ and let $S^{(1)},...,S^{(d)}>0$ be positive numbers.
		Denote by $S_M = \max \{S^{(1)},...,S^{(d)}\}$
		and $S_m = \min \{S^{(1)},...,S^{(d)}\}$.
		Given $\varepsilon >0$, if $s \leq \frac{\varepsilon^2}{8d}S_m$, then we can find pairwise disjoint boxes $A_i = \mathcal{N}_{S^{(1)}}(a^{(1)}_i)\times ... \times \mathcal{N}_{S^{(d)}}(a^{(d)}_i) \subset A$ satisfying
		\begin{enumerate}
			\item $1 - \nu(\partial_{3S_M} A) - \ep \leq \nu(\bigcup_i A_i) $;
			\item $\nu(\partial_s A_i) \leq \ep\nu(A_i)$;
			\item $ \frac{\varepsilon}{2} \frac{2^d\prod_{k=1}^d S^{(k)}}{vol(A)} \leq \nu(A_i)$.
		\end{enumerate}
	\end{lemma}
	\begin{proof}
		A box in $\R^d$ is defined by $2d$ sides. We will find d families of hyperplanes whose $s$-neighborhood have small measure. Consider the projection on the $k$-th coordinate $\pi_k \colon \R^d \to \R$ and $\nu_k = (\pi_k)_*\nu$. Define for $0 \leq j \leq \floor{\frac{S^{(k)}}{s}}-1$
		\[
		H_j^k = \bigcup_{l \in \Z} \mathcal{N}_s((2j+1)s)+2lS^{(k)} \subset \R.
		\]
		Notice that $\{H_j^k\}_j$ are pairwise disjoint.
		Since $\sum_{j=0}^{\floor{\frac{S^{(k)}}{s}}-1} \nu_k(H_j^k) \leq 1$ and we have $\floor{\frac{S}{s}}$ many terms, one of them should be less than $\frac{1}{
			\floor{\frac{S^{(k)}}{s}}} < \frac{s}{S^{(k)}-s}$. Pick such a set and call it $H^k$. Then $H = \bigcup_k\pi_k^{-1}(H^k)$ is a union of neighborhoods of size $s$ of hyperplanes and
			\begin{equation}
				\label{nuH}
				\nu(H) \leq \sum_{k=1}^d \frac{s}{S^{(k)}-s}.
			\end{equation}
			The connected components of the complements of these hyperplanes define boxes $(A_i)_{i \in I_1}$ of diameter $2S_M$.
			Let $I_2 = \{ i \in I_1 \ | \ A_i \subset A\}$.
			Note that
		\[
		\bar{A} \setminus \partial_{3S_M} A \subset \bigcup_{i \in I_2} A_i \cup \partial_s A_i.
		\]
		Therefore, by \eqref{nuH},
		\begin{equation}
			\label{eq1}
			1 - \nu(\partial_{3S_M} A) \leq \sum_{i \in I_2}\nu(A_i) + \sum_{k=1}^d \frac{s}{S^{(k)}-s}.
		\end{equation}
		Let 
		\[
		g = \sum_{i \in I_2} \frac{\nu(A_i\cap \partial_s A_i)}{\nu(A_i)} \mathbbm{1}_{A_i},
		\]
		where we're summing over the boxes of positive measure. 
		By Markov's inequality and \eqref{nuH},
		\begin{equation}
		\label{nuH2}
			\nu \left(x  \ \Big| \ a \leq g(x)\right) \leq a \int g d\nu
			\leq \frac{1}{a} \nu(H) \leq \frac{1}{a}\sum_{k=1}^d \frac{s}{S^{(k)}-s} \leq \frac{1}{a}d\frac{2s}{S_m}.
		\end{equation}
		where we are assuming without loss of generality that $s \leq \frac{1}{2}S_m$.
		Note that $a \leq g(x)$ is equivalent to $x \in A_i$ for some $i \in I_2$ satisfying $\nu(A_i)>0$ and $\nu(A_i\cap \partial_s A_i) \geq a\nu(A_i)$.
		Let
		\begin{equation*}
			I_3(a) = \left\{ i \in I_2  \ \Big| \   \nu(A_i \cap \partial_s A_i) < a\nu(A_i)\right\}.
		\end{equation*}
		Since
		\begin{equation*}
			\nu\left(\underset{i \in I_2}{\bigcup}A_i\right)
			-\nu\left(\underset{i \in I_2\setminus I_3(a)}{\bigcup}A_i\right) = 
			\nu\left(\underset{i \in I_3(a)}{\bigcup}A_i\right),
		\end{equation*}
		then by \eqref{eq1}, \eqref{nuH2} and assuming without loss of generality that $a<1$,
		\begin{equation}
		\label{I3a}
			1-\nu(\partial_{3S_M}A) - \frac{4d}{a}\frac{s}{S_m}
			\leq \nu\left(\underset{i \in I_3(a)}{\bigcup}A_i\right).
		\end{equation}
		Now let 
		\begin{equation*}
			h = \sum_{i \in I_3(a)} \frac{1}{\nu(A_i)}\mathbbm{1}_{A_i},
		\end{equation*}
		where we are summing over elements with positive measure.
		By Markov's inequality and using $vol(A_i) = 2^d\prod_{k=1}^d S^{(k)}$,
		\[
		\nu \left(x \ | \ \frac{1}{b} \leq h(x)\right) \leq b \int g d\nu \leq b \# I_2 \leq b \frac{vol(A)}{vol(A_i)}
		\leq b \frac{vol(A)}{2^d\prod_{k=1}^d S^{(k)}}.
		\]
		But $\frac{1}{b} \leq h(x) \iff 0< \nu(A_i) \leq b$ with $x \in A_i$ for some $i \in I_3(a)$. Therefore,
		\begin{equation}
		\label{nuH3}	
		\nu \left(x \ \Big| \ \frac{1}{b} \leq g(x)\right)  = \sum_{\nu(A_i)\leq b} \nu(A_i) \leq b \frac{ vol(A)}{2^d\prod_{k=1}^d S^{(k)}}.
		\end{equation}
		Let 
		\begin{equation*}
			I(a,b) = \{ i \in I_3(a) \ | \ b < \nu(A_i) \}.
		\end{equation*}
		By \eqref{I3a} and \eqref{nuH3},
		\begin{equation}
			\label{shifttrick}
			1-\nu(\partial_{3S_M}A) - \frac{4d}{a}\frac{s}{S_m}
			- b \frac{ vol(A)}{2^d\prod_{k=1}^d S^{(k)}}
			\leq \nu\left(\underset{i \in I(a,b)}{\bigcup} A_i\right).
		\end{equation}
		We finish the proof by taking $a = \varepsilon$ and
		$b = \frac{\varepsilon}{2} \frac{2^d\prod_{k=1}^d S^{(k)}}{vol(A)}$.
		Items (2) and (3) follow from the choice of $a$ and $b$ while item (1) follows from \eqref{shifttrick}, 
		since $s \leq \frac{\varepsilon^2}{8d}S_m $.	
	\end{proof}
	
	\subsection{Construction of good boxes}
    \label{sec:goodboxconstruct}
	We  now prove \cref{goodbox}.
	The construction will be carried out in 4 steps.
	In \nameref{Step1}, we fix a Pesin set and fix a $D$-Nice partition of the space such that each atom intersecting the Pesin set is contained in a Lyapunov chart, those are indexed by $k$.
    
	In \nameref{Step2}, we refine this partition in each chart using the \cref{shift} in order to get boxes of size $R$ small that have controlled measure of a neighborhood of their boundary, indexed by $j$.
    
	In \nameref{Step3}  we use the \cref{shift} once more to define covers by boxes a slightly larger than the desired size in each box of size $R$, indexing them by $i$. The problem with these boxes is that they all come from a single Lyapunov chart, while we want a chart for each box that we construct.
	This is done in \nameref{Step4}, where we show that we can take a chart for each box and the resulting construction is the desired cover, still indexed by $i$.
    The proof of the desired properties will follow right after.
	\subsubsection*{Step 1}
    \label{Step1}
	Let $\tau>0$ be such that $\mu(\mathcal{P}_{\tau}) > 1 - \ep^3$. Cover $\mathcal{P}_{\tau}$ with finitely many Lyapunov charts of size $\tau$, which we denote by $h_{w_k} \colon \R^u \times \R^{cs} \to M$ with $k \in K$. For each $k \in K$, take $Q_k \subset Im(h_{w_k})$ open disk such that  $\mathcal{P}_{\tau} \subset \bigcup_{k \in K} \bar{Q_k}$ and,
	by \cref{Dnicelemma}, we may assume that $\mathcal{Q} = \{Q_k\}_{k\in K} \cup \{\left(\bigcup_{k \in K} Q_k\right)^c$\} is a $D$-Nice partition, for some $D>0$.  We may also assume that each $Q_k$ has positive measure, otherwise we discard it from our collection.
	\subsubsection*{Step 2}
 \label{Step2}
	For each $k \in K$, define $\hat{Q}_k = h_{w_k}^{-1}(Q_k)\subset \R^d$ and $\nu_k = (h_{w_k}^{-1})_*(\mu_{Q_k})$ probability measure on $\R^d$ supported on $\overline{\hat{Q}_k}$.
    Let $R_n = \sqrt{\xi_n}$, which satisfies $\xi_n \ll R_n$. 
    Applying \cref{shift} for $s = 4\tau^{-1}\xi_n$, $S^{(1)}=...=S^{(d)} = R_n$ and $A = \hat{Q}_k$, for $n$ large enough, we get a family of disjoint boxes $\hat{A}_{k,j} = \hat{A}^{(n)}_{k,j} = \mathcal{N}_{R_n}(a_{k,j}) \subset \hat{Q}_k \subset\R^d$, $j \in J(k)$ satisfying
	\begin{enumerate}
		\item $1 - \nu_k(\partial_{3R_n} \hat{Q}_k) -\ep^2 \leq \nu_k(\bigcup_{j \in J(k)} \hat{A}_{k,j})$;
		\item $\nu_k(\partial_{4\tau^{-1}\xi_n} \hat{A}_{k,j}) \leq \ep^2 \nu_k(\hat{A}_{k,j})$;
		\item $\frac{ \varepsilon^2}{2} \frac{2^dR_n^d}{vol(\hat{Q}_k)} \leq \nu_k(\hat{A}_{k,j})$.
	\end{enumerate}
	Since $\mathcal{Q}$ is a $D$-nice partition and we have finitely many Lipschitz charts, then for $D' = D \max_k\{Lip(h_{w_k})\}$ we have 
	\[
		\mu(Q_k)\nu_k(\partial_{3R_n} \hat{Q}_k) \leq 3D'R_n.
	\]
	Therefore, multiplying the inequality in property (1) above by $\mu(Q_k)$, summing over $k$ and denoting by $A_{k,j} = h_{w_k}(\hat{A}_{k,j})$, we get
	\begin{equation}
		\label{1stPart}
		(1-\ep^3)(1-\ep^2) - 3D'(\#K)R_n \leq \sum_k\sum_{j \in J(k)} \mu(A_{k,j}).
	\end{equation}
	We relabel the boxes $A_{k,j}$ to just $A_j$ for $j \in J_1$. Define $J_2 = \{j \in J_1 \ | \ \mu(A_j\cap \mathcal{P}_{\tau}^c) \leq \ep^2\mu(A_j)\}$, then by Markov's inequality applied to the function
	\[
		g = \sum_{j \in J_1} \frac{\mu(A_j \cap \mathcal{P}_{\tau}^c)}{\mu(A_j)} \mathbbm{1}_{A_j}
	\]
	we get
	\[
	\sum_{j \notin J_2} \mu(A_j) \leq \ep.
	\]
	This implies
	\begin{equation}
		\sum_{j \in J_1} \mu(A_j) - \varepsilon \leq \sum_{j \in J_2} \mu(A_j). \label{As}
	\end{equation}
	For $j \in J_2$, pick $w_j \in A_j \cap \mathcal{P}_{\tau}$ and let $h_{w_j} \colon \R^u \times \R^{cs} \to M$  be the Lyapunov chart around $w_j$.
	\subsubsection*{Step 3}
 \label{Step3}
	For each $j\in J_2$, define $\hat{A}_j = h_{w_j}^{-1}(A_j) \subset \mathbb{R}^d$ and $\nu_j = (h_{w_j}^{-1})_*(\mu_{A_j})$ probability measure on $\mathbb{R}^d$ supported on $\overline{\hat{A}_j}$, similar to the previous step.
    Applying \cref{shift} for $s = 100l_n$,
    $S^{(1)}=...=S^{(u)} = \xi_n+50l_n$, $S^{(u+1)}=...=S^{(d)} = r_n+50l_n$ and $A = \hat{A}_j^{(n)}$, for each $n$ large enough, we get boxes $\hat{B}_{j,i} = \hat{B}_{j,i}^{(n)} = \mathcal{N}_{(\xi_n+50l_n,r_n+50l_n)}(b_{j,i})\subset \hat{A}_j$, $i \in I_1(j)$ satisfying
	\begin{enumerate}
		\item $1 - \nu_j(\partial_{3(\xi_n+50l_n)} \hat{A}^{(n)}_j)) - \ep \leq \nu_j\left(\bigcup_{i \in I_1(j)} \hat{B}^{(n)}_{j,i}\right)$;
		\item $\nu_j(\partial_{100l_n} \hat{B}_{j,i}) \leq \ep \nu_j(\hat{B}_{j,i})$;
		\item $ \frac{\varepsilon}{2} \frac{2^d(\xi_n+50l_n)^u(r_n+50l_n)^{cs}}{vol(\hat{A}_j^{(n)})} \leq \nu_j(\hat{B}_{j,i})$.
	\end{enumerate}
	By (2) of \eqref{LyapCharts} and the choice of $w_j \in \mathcal{P}_{\tau}$, $\text{Lip}(h^{-1}_{w_j})\leq \tau^{-1}$.
	Therefore, since $vol(A_j) = 2^dR_n^d$, then
	$vol(\hat{A}^{(n)}_j) \leq 2^d\tau^{-2}R_n^d$ and we have
	\begin{equation*}
		 \frac{\varepsilon\tau^2}{2}\frac{(\xi_n+50l_n)^u(r_n+50l_n)^{cs}}{R_n^d}\leq \frac{\varepsilon}{2} \frac{2^d(\xi_n+50l_n)^u(r_n+50l_n)^{cs}}{vol(\hat{A}_j^{(n)})}
		 \leq \nu_j(\hat{B}_{j,i}).
	\end{equation*}
	Denote by $B_{j,i} = h_{w_j}(\hat{B}_{j,i})$.
	Define for $j \in J_2$, $I_2(j) = \{i \in I_1(j) \ | \ \mu(B_{j,i} \cap \mathcal{P}_{\tau}^c) \leq \ep \mu(B_{j,i})\}$ which by the same argument as above satisfies
	\[
		\sum_{i \notin I_2(j)} \mu(B_{j,i} ) \leq \frac{1}{ \varepsilon} \mu(A_j\cap K^c) \leq \ep \mu(A_j).
	\]
	Multiplying item (1) above by $\mu(A_j)$, we get
	\begin{equation*}
		(1- 2\varepsilon)\mu(A_j) - \mu(h_{w_j}(\partial_{3(\xi_n+50l_n)} \hat{A}_j)) \leq \sum_{i \in I_2(j)} \mu(B_j,i).
	\end{equation*}
	Using (2) of \cref{LyapCharts} again, we have
	$h_{w_j}(\partial_{3(\xi_n+50l_n)} \hat{A}_j) \subset \partial_{3\tau^{-1}(\xi_n+50l_n)} A_j \subset \partial_{4\tau^{-1}\xi_n} A_j$
	for $n$ large enough.
	Using this fact and summing the above inequality over $j \in J_2$,
	\begin{equation}
		(1- 2\varepsilon - \varepsilon^2)\sum_{j \in J_2} \mu(A_j) \leq \sum_j\sum_{i \in I_2(j)} \mu(B_{j,i}). \label{tildes}
	\end{equation}
	\subsubsection*{Step 4}
 \label{Step4}
	For $i \in I_2(j)$, pick $z_{j,i} \in B_{j,i} \cap \mathcal{P}_{\tau}$ and $h_{z_{j,i}} \colon \R^u \times \R^{cs} \to M$ the Lyapunov chart around that point. We wish to pick an image of a box with respect to the $h_{z_{j,i}}$ having nice control for the measure of a neighborhood of its boundary. Let $\varphi_{j,i} = (h_{z_{j,i}})^{-1} \circ h_{w_j}$ and denote by $c_{j,i} = \varphi_{j,i}( b_{j,i})$. For now we're fixing $i,j$ and taking $n$ large, so we will omit them in our notation. We wish to show that
	\begin{equation}
		\partial_l\mathcal{N}_{(\xi,r)}(c) \subset \varphi(\partial_{100l}\mathcal{N}_{(\xi+50l,r+50l)}( b)). \label{P3}
	\end{equation}
	It is enough to show the following inclusions
	\begin{equation}
		\label{disjoint}
		\varphi^{-1}(\mathcal{N}_{(\xi+l,r+l)}(c)) \subset \mathcal{N}_{(\xi+2l,r+2l)}(b),
	\end{equation}
	\begin{equation}
		\label{P2}
		\varphi(\mathcal{N}_{(\xi-2l,r-2l)}( b)) \subset \mathcal{N}_{(\xi-l,r-l)}(c).
	\end{equation}
	Since $\varphi$ and $\varphi^{-1}$ have similar properties, we will just prove \eqref{P2}, while \eqref{disjoint} follows by the same argument.  If $x \in \mathcal{N}_{(\xi-2l,r-2l)}( b)$, then
	\begin{equation}
    \label{P2sol}
		d(\varphi(x),\varphi( b)) \leq \underset{z}{\sup} \|D_z \varphi\| d(x, b) \leq \underset{z}{\sup} \|D_z \varphi \| (\xi-2l)
	\end{equation}
	where the supremum is taken over $\mathcal{N}_{(\xi-2l,r-2l)} ( b)$. But $\varphi$ is given by
	\begin{equation*}
		\varphi = h_q^{-1} \circ h_p.
	\end{equation*}
	By \cref{changecoordinates},
    \begin{equation*}
        \|D_z\varphi\| \leq 1 + C(d_M(p,q)^{\alpha_2} + |z|).
    \end{equation*}
    Since $d_M(p,q) \leq \tau^{-1}R$,
	\begin{equation*}
		d(\varphi(x),\varphi(b)) \leq (1+C(\tau^{-1}R^{\alpha_2}+\xi))(\xi-2l).
	\end{equation*}
	Since $R^{\alpha_2}\xi \ll l$ and $\xi^2 \ll l$, then for $n$ large enough
	\begin{equation*}
		(1+C(\tau^{-1}R^{\alpha_2}+\xi))(\xi-2l) \leq \xi-l.
	\end{equation*}
	Therefore, \eqref{P2} follows from \eqref{P2sol} for $n$ large enough.
	Denote by $\hat{C}_{j,i} = \mathcal{N}_{(\xi,r)}(c_{j,i})$. They're disjoint by \eqref{disjoint}. 
 
    \subsection{Verifying properties}
    We now verify that the boxes constructed in \ref{sec:goodboxconstruct} satisfy properties (P1) to (P8) stated in \cref{goodbox}.

	\subsubsection*{Property (P1)}
    By (2) of \cref{LyapCharts},
	\begin{equation*}
	d(c_{j,i},0) \leq \tau^{-1}d(h_{z_{j,i}}(c_{j,i}),z_{j,i}) \leq \tau^{-2} d(b_{j,i},h_{w_j}^{-1}(z_{j,i})) \leq \tau^{-2}2\xi.
	\end{equation*}
	This means that $c_{j,i} \in \mathcal{N}_{2\tau^{-2}\xi}(0)$.

	 \subsubsection*{Property (P2)}
	 By \eqref{P2}, we have
	\begin{equation*}
		h_{w_j}(\hat{B}_{j,i} \setminus \partial_{2l} \hat{B}_{j,i}) \subset h_{w_{j,i}}(\hat{C}_{j,i}) = C_{j,i}.
	\end{equation*}
    Since $\mu(h_{w_j}(\mathcal{N}_{2l}\hat{B}_{j,i})) \leq \varepsilon\mu(B_{j,i})$, then
	\begin{equation}
		\label{CbigB}
		(1- \varepsilon)\mu(B_{j,i}) \leq \mu(C_{j,i}).	
	\end{equation}
	It follows from the third property of the construction of $B_{j,i}$ that
	\begin{equation*}
		\frac{ \varepsilon}{2} (\xi_m+50l_m)^d \leq \frac{\mu(B_{j,i})}{\mu(A_j)}.
	\end{equation*}
	Using the third property of the construction of $A_j$
    together with the definition of $\nu_k$, we have
	\begin{equation*}
		\frac{\text{min}\{\mu(Q_k) \ | \ k \in K \}}{\text{max} \{ vol(\hat{Q}_k) \ | \ k \in K\}}\frac{ \varepsilon^2}{2}R^d_m \leq \mu(A_j).
	\end{equation*}
    Combining with the above,
	\begin{equation*}
		\frac{ \varepsilon^3}{4}\frac{\text{min}_k \mu(Q_k)}{\text{max}_k vol(\hat{Q}_k)} R_m^d (\xi_m+50l_m)^d \leq \mu(B_{j,i}).
	\end{equation*}
	This implies (P2).

	\subsubsection*{Property (P3)}
	Summing over $j,i$ in \eqref{CbigB}, we get
	\begin{equation}
		(1- \varepsilon) \sum_{i,j} \mu(B_{j,i}) \leq \sum_{i,j} \mu(C_{j,i}). \label{Bs}
	\end{equation}
	Combining \eqref{1stPart},\eqref{As},\eqref{tildes} and \eqref{Bs}, we get (P3).

	\subsubsection*{Property (P4)}
	It follows from \eqref{P3} and the properties of $B_{j,i}$ that
	\begin{equation*}
		\mu(h_{z_{j,i}}(\partial_l \hat{C}_{j,i})) \leq \mu(h_{w_j}(\partial_{100l} \hat{B}_{j,i})) \leq \varepsilon \mu(B_{j,i}) \leq \frac{ \varepsilon}{1 - \varepsilon} \mu(C_{j,i}).
	\end{equation*}

	\subsubsection*{Property (P5)}
	Since $i \in I_2(j)$, we have by \eqref{disjoint}
	\begin{equation*}
		\mu(C_{j,i} \cap \mathcal{P}_{\tau}^c) \leq \mu(B_{j,i} \cap \mathcal{P}_{\tau}^c) \leq \varepsilon \mu(B_{j,i}) \leq \frac{ \varepsilon}{1- \varepsilon} \mu(C_{j,i}).
	\end{equation*}
	We now relabel the boxes to simply $C_i$ for $i \in I$.

	\subsubsection*{Property (P6)}
	Let $z \in h_{z_i}(\hat{C}_i\setminus \partial_l\hat{C}_i) \cap \mathcal{P}_{\tau}$.
    By \cref{unstablegraph},
	\begin{equation*}
		h_{z_i}^{-1}(W^u_{\tau}(h_{z_i}(z))) = \text{graph}( \eta^u_{z_i,z}),
	\end{equation*}
	where $\eta^u_{z_i,z} \colon \mathcal{N}^u_{\tau}(0) \to \mathbb{R}^{cs}$ is $C^{1+\alpha_4}$ with $\|\eta\|_{C^1+\alpha_4} \leq C = C(\tau)$.
    Since $z_i \in \mathcal{P}_{\tau}$, it also satisfies
    \begin{equation}
        \label{tanclose}
        \|D_w\eta^u_{z_i,z_i}-D_w\eta^u_{z_i,z}\| \leq C|(w,\eta^u_{z_i,z_i}w)-(w,\eta^u_{z_i,z}w)|^{\alpha_4},
    \end{equation}
    where $\text{graph}(\eta^u_{z_i,z_i}) = \tilde{W}^u(z_i)$.
    We wish to show that for all $x \in \mathcal{N}^u_{\xi}(\pi^u(c_i))$, $|\eta(x)- \pi^{cs}(c_i)| < r$.
    We have,
	\begin{equation}
    \label{distbound}
		|\eta^u_{z_i,z}(x)-\eta^u_{z_i,z}(x')| \leq \sup_{z \in \mathcal{N}_{\xi}(\pi^u(c_i))}\|D_z \eta\| 2\xi.
	\end{equation}
    Substituting $w = 0$ in \eqref{tanclose} and recalling that $\tilde{W}(z_i)$ is tangent to $\mathbb{R}^u$ at the origin, we get
    \begin{equation*}
        \|D_0\eta^u_{z_i,z}\| \leq C|\eta^u_{z_i,z}(0)|^{\alpha_4} \leq Cr^{\alpha_4}.
    \end{equation*}
	Using the bound on $\|\eta^u_{z_i,z}\|_{C^{1+\alpha_4}}$, we have
	\begin{equation*}
		\|D_z \eta^u_{z_i,z}\|
        \leq \|D_z\eta^u_{z_i,z} - D_0\eta^u_{z_i,z}\| + \|D_0\eta^u_{z_i,z}\|
        \leq C|z|^{\alpha_4} + Cr^{\alpha_4}
        \leq 2C\xi^{\alpha_4}.
	\end{equation*}
	Substituting the above in \eqref{distbound},
	\begin{equation*}
		|\eta(x)-\eta(x')| \leq 4C \xi^{1+\alpha_4}. 
	\end{equation*}
     The hypothesis says that for some $x'\in \mathcal{N}^u_{\xi -l}(\pi^u(c_i))$, $|\eta^u_{z_i,z}(x')-\pi^{cs}(c_i)| < r-l$.
     Therefore,
    \begin{equation*}
        |\eta(x)-\pi^{cs}(c_i)| \leq r-l + 4C \xi^{1+\alpha_4}.
    \end{equation*}
	We get the result as long as we have
	\begin{equation*}
		4C \xi^{1+\alpha_4} \leq l
	\end{equation*}
	which we have for $n$ large enough by the assumption on the sequences.

	\subsubsection*{Property (P7)}
	Notice that 
    \begin{equation*}
        C_i\cap \tilde{T}_i^c \subset \mathcal{P}_{\tau}^c \cup  h_{z_i}(\partial_l \hat{C}_i).
    \end{equation*}
    Property (P7) then follows by properties (P4) and (P5).

	\subsubsection*{Property (P8)}

	Since we will work with a fixed $i$, we'll omit it from the notation.
	By the definition of SRB measure, we may write for $A \subset C$
	\begin{equation}
		\label{SRBdecomp}
		\mu(A \cap \tilde{T}) = 
		\int_{T} \int_{W(x)} \rho_x(y)\mathbbm{1}_{A}(y) dm^u_x(y) d\nu(x) .
	\end{equation}
	By \cref{SRBdensityHol},
	and recalling that the size of $W(x)$ goes to $0$ as $\xi_n$ goes to $0$,
	we may take $n$ large enough so that 
	\begin{equation*}
		y \in W(x) \implies
		\rho_x(y) \sim_{(1+ \varepsilon)^{\frac{1}{4}}} \frac{1}{m^u_x(W(x))}.
	\end{equation*}
	Substituting the above in \eqref{SRBdecomp}, we get
	\begin{equation}
		\label{SRBdecomp2}
		\mu(A\cap \tilde{T}) \sim_{(1+ \varepsilon)^{\frac{1}{4}}} \int_T \frac{m^u_x(A)}{m^u_x(W(x))} d\nu(x).
	\end{equation}
	By \cref{unstablegraph}, increasing $n$ if necessary, the unstable manifolds satisfy the condition in (F2).
	Therefore, if we further assume that $A$ is $\mathcal{F}^{cs}$-saturated, then for any $x,x' \in T$ we have
	\begin{equation*}
		m^u_x(A) \sim_{(1+ \varepsilon)^{\frac{1}{4}}} m^u_{x'}(A).
	\end{equation*}
	For $n$ large enough, we also have
	\begin{equation*}
		m^u_x(W(x))  \sim_{(1+ \varepsilon)^{\frac{1}{4}}}m^u_{x'}(W(x')).
	\end{equation*}
	Combining the above with \eqref{SRBdecomp2}, we obtain
	\begin{equation*}
		\mu(A\cap \tilde{T}) \sim_{(1+ \varepsilon)^{\frac{3}{4}}}
		\nu(T)\frac{m^u_{x'}(A)}{m^u_{x'}(W(x'))}.
	\end{equation*}
	This proves the first part of (P8) since $\nu(T) = \mu(\tilde{T})$ and, by \cref{SRBdensityHol},
	\begin{equation*}
	\frac{m^u_{x'}(A)}{m^u_{x'}(W(x'))} \sim_{(1+ \varepsilon)^{\frac{1}{4}}} \mu_{W(x')}(A).
	\end{equation*}
    For the second part of (P8), let $u = \pi^{cs}\circ h_z^{-1} \colon C \to \mathbb{R}^{cs}$.
    Then by the coarea formula from differential geometry, for $\varphi \colon C \to \mathbb{R}$ an $L^1$ function:
    \begin{equation}
        \label{coarea}
        \int_C \varphi dm = \int_{\mathcal{N}^{cs}_r(c^{cs})}\int J_u(x)^{-1}\varphi(x)dm_{u^{-1}(y)}(x)dm_{\mathbb{R}^{cs}}(y),
    \end{equation}
    where $J_u(x) = \sqrt{\text{det}(D_xu)(D_xu)^*}$ is the Jacobian of $u$.
    The precise formula for $J_u$ is not important.
    But note that it is H\"older continuous and its H\"older constant only depends on that of $h_z^{-1}$, which is uniformly bounded on $\mathcal{P}_{\tau}$.
    Note also that, by (L3) of \cref{LyapCharts}, $J_u$ is uniformly bounded away from $0$ for $z \in \mathcal{P}_{\tau}$.
    Recall that $\text{diam}(C) \leq 2\tau^{-1}\xi_n$.
    Therefore, for $x,y\in C$,
    \begin{equation*}
        \left|\frac{J_u(x)}{J_u(y)}-1\right| \leq \underset{y \in C}{\sup}J_u(y)^{-1}\|J_u\|_{\text{H\"ol}} (2\tau^{-1}\xi_n)^{\alpha}.
    \end{equation*}
    Therefore, for $n$ large enough and $x,y \in C$
    \begin{equation}
        \label{Jucte}
        J_u(x) \sim_{(1+\varepsilon)^{\frac{1}{4}}} J_u(y).
    \end{equation}
    Now for $A \subset C$ $\mathcal{F}^{cs}$-foliated, using \eqref{coarea}, \eqref{Jucte} and the fact that $u^{-1}(y)$ is a $(\alpha,1)$-mostly horizontal submanifold satisfying the condition in (F2) for each $y \in \mathcal{N}^{cs}(c^{cs})$, we get
    \begin{equation}
        \label{Asat}
        m(A) \sim_{(1+\varepsilon)^{\frac{1}{2}}} J_u(z)^{-1}m_{\mathbb{R}^{cs}}(\mathcal{N}^{cs}_r(c^{cs}))m_{u^{-1}(c^{cs})}(A).
    \end{equation}
    In particular, since $C$ is saturated by vertical lines $h_z(\mathbb{R}^{cs})$, then for $n$ large enough
    \begin{equation}
        \label{Csat}
        m(C) \sim_{(1+\varepsilon)^{\frac{1}{2}}} J_u(z)^{-1}m_{\mathbb{R}^{cs}}(\mathcal{N}^{cs}_r(c^{cs}))m_{u^{-1}(c^{cs})}(C).
    \end{equation}
    Dividing \eqref{Asat} by \eqref{Csat} we get
    \begin{equation}
        \label{ACsat}
        \frac{m(A)}{m(C)} \sim_{1+\varepsilon} \frac{m_{u^{-1}(c^{cs})}(A)}{m_{u^{-1}(c^{cs})}(C)}.
    \end{equation}
    Now given $x \in T$ with $\mu_{W(x)}$ well defined, by \cref{SRBdensityHol},
    \begin{equation*}
        \mu_{W(x)} \sim_{1+\varepsilon} \frac{m^u_x(A)}{m^u_x(C)}.
    \end{equation*}
    Using once again that $A$ is $\mathcal{F}^{cs}$-saturated, $C$ is $h_z(\mathbb{R}^u)$-saturated, the fact that $u^{-1}(c^{cs})$ is $(\alpha,1)$-mostly horizontal and \eqref{ACsat} together with the above,
    \begin{equation*}
        \mu_{W(x)} \sim_{1+\varepsilon} \frac{m_{u^{-1}(c^{cs})}(A)}{m_{u^{-1}(c^{cs})}(C)} \sim_{1+\varepsilon}  \frac{m(A)}{m(C)}.
    \end{equation*}
    This shows the second part of (P8) and concludes the proof of \cref{goodbox}.

	\subsection{Mixing of exponentially small good boxes}

	In this subsection, we pass from exponentially mixing of Lipschitz functions to mixing of exponentially small good boxes from \cref{goodbox}.
    The crucial property from that lemma is (P4), since it allows us to $L^1$-approximate the indicator function of those good boxes by Lipschitz functions with controlled Lipschitz norm.
	\begin{proposition}
		\label{mixtubes}
		Suppose that $(f,\mu)$ is exponentially mixing with $\mu$ nonatomic.
		Let $ \varepsilon>0$ and $l_n,r_n,\xi_n,\overline{l}_n,\overline{r}_n,\overline{\xi}_n$
		be sequences of positive numbers converging to $0$ satisfying
        \begin{equation*}
            \xi_n^{1+\alpha} \ll l_n \ll r_n \ll \xi_n,
        \end{equation*}
        \begin{equation*}
             \overline{\xi}_n^{1+\alpha} \ll  \overline{l}_n \ll  \overline{r}_n \ll  \overline{\xi}_n
        \end{equation*}
        and
		\begin{equation}
			\label{ScsMix}
			e^{-\beta n} \ll \xi_n^{2d}\overline{\xi}_n^{2d}l_n\overline{l}_n.
		\end{equation}
		Then, there exists $n_0 \in \mathbb{N}$ and $ \varepsilon_0 >0$ such that if
		$\left(C_j\right)_{j \in J(n)}$ and $\left(\overline{C}_i\right)_{i \in I(n)}$ are good boxes from \cref{goodbox} 
		for the sequences $\xi,r,l$ and $\overline{\xi},\overline{r},\overline{l}$ respectively and value $ \varepsilon_0$, then
		for all $n> n_0$, all $i \in I(n)$ and $j \in J(n)$,
		\begin{equation*}
			\mu\left(C_j \cap f^{- n}\overline{C}_i\right)
   \sim_{1+ \varepsilon}
   \mu\left(C_j\right)
   \mu\left(\overline{C}_i\right).
		\end{equation*}
	\end{proposition}
	\begin{proof}
		We will omit $i$ and $j$ from the notation since they're fixed.
        By (P1) from \cref{goodbox}, we'll denote $C = h_z(\hat{C})$ and $\overline{C} = h_{\overline{z}}(\hat{\overline{C}})$,
        with $z,\overline{z} \in \mathcal{P}_{\tau}$.
		For $h,g$ Lipschitz functions approximating $\mathbbm{1}_C$ and $\mathbbm{1}_{\overline{C}}$ respectively, we have
		\begin{equation}
			\label{1stapprox}
			\frac{\mu(C \cap f^{- n}\overline{C} )}{\mu(C)\mu(\overline{C})}
			= \left(\frac{\mu(C \cap f^{ -  n}\overline{C})}{\mu(h g \circ f^{ - n})}\right)
			\left(\frac{\mu(h g \circ f^{ n})}{\mu(h)\mu(g)}\right)\left(\frac{\mu(h)}{\mu(C)}\right)\left(\frac{\mu(g)}{\mu(\overline{C})}\right).
		\end{equation}
		We'll use the above formula to prove that our desired quantity is close to 1.
        We'll first show that
        \begin{equation*}
            \frac{\mu(C \cap f^{- n}\overline{C} )}{\mu(C)\mu(\overline{C})} \leq 1+ \varepsilon.
        \end{equation*}
		Choose $h,g$ satisfying $\mathbbm{1}_C \leq h \leq \mathbbm{1}_{h_z(\hat{C} \cup \partial_l \hat{C})}$ and $\mathbbm{1}_{\overline{C}} \leq g \leq \mathbbm{1}_{h_{\overline{z}}(\hat{\overline{C}} \cup \partial_{\overline{l}} \hat{\overline{C}})}$.
		Notice that this implies that
		\begin{equation*}
			\mathbbm{1}_C \mathbbm{1}_{\overline{C}} \circ f^{ n} \leq h g \circ f^{n}.
        \end{equation*}
        Therefore,
        \begin{equation*}
            \frac{\mu(C \cap f^{- n}(\overline{C}))}{\mu(h g \circ f^{ n})} \leq 1,
		\end{equation*}
		which bounds the first term in \eqref{1stapprox}.
		For the second term, by construction of $h$ and $g$ and property (2) of \cref{LyapCharts},
        \begin{equation*}
            \|h\|_{\text{Lip}} \leq \frac{2\tau^{-1}}{l_n} \ \text{and} \ \|g\|_{\text{Lip}}\leq \frac{2\tau^{-1}}{\overline{l}_n}.
        \end{equation*}
        It follows from \eqref{H1} and (P2) from \cref{goodbox} that
        \begin{equation*}
            \left|\frac{\mu(hg\circ f^n)}{\mu(h)\mu(g)}-1\right| \leq 4C\tau^{-2}D^{-2}\frac{e^{-\beta n}}{\xi^{2d}_n\overline{\xi}^{2d}_nl_n\overline{l}_n}.
        \end{equation*}
		By \eqref{ScsMix}, the above goes to $0$ as $n$ goes to infinity.
		For the third term, since $h \leq \mathbbm{1}_{h_z(\hat{C} \cup \partial_l \hat{C})}$, then
        \begin{equation*}
            \mu(h) \leq \mu(C) + \mu(h_z(\partial_l\hat{C})).
        \end{equation*}
        By (P4) of \cref{goodbox}, we have
        \begin{equation*}
		  \frac{\mu(h)}{\mu(C)} \leq 1+ \varepsilon_0.
		\end{equation*}
        Choosing $\varepsilon_0$ small enough, we get the desired bound.
		The fourth term is analogous to the third, so we omit its bound.
		We'll now use \eqref{1stapprox} to show that
        \begin{equation*}
            \frac{1}{1+\varepsilon} \leq \frac{\mu(C \cap f^{- n}\overline{C} )}{\mu(C)\mu(\overline{C})}.
        \end{equation*}
		It is the same strategy but now choose $h,g$ satisfying
		$\mathbbm{1}_{h_z(\hat{C}\setminus\partial_l \hat{C})} \leq h \leq \mathbbm{1}_C$
		and $\mathbbm{1}_{h_{\overline{z}}(\hat{\overline{C}} \setminus \partial_{\overline{l}} \hat{\overline{C}})} \leq g \leq \mathbbm{1}_{\overline{C}}$.
		Notice that
		\begin{equation*}
			h g \circ f^{  n} \leq \mathbbm{1}_C \mathbbm{1}_{\overline{C}} \circ f^{  n}.
        \end{equation*}
        Therefore,
        \begin{equation*}
            1 \leq \frac{\mu(C \cap f^{-n}\overline{C})}{\mu(h g \circ f^{ n})},
		\end{equation*}
        which bounds the first term of \eqref{1stapprox}.
		Bounding the second term is similar to what was done above, but now we have
		$(1- \varepsilon_0)D\xi_n^{2d} \leq\mu(h)$ and $ (1- \varepsilon_0)D\overline{\xi}_n^{2d} \leq \mu(g)$.
		For the third term, since $\mathbbm{1}_{h_z(\hat{C}\setminus\partial_l \hat{C})} \leq h$, then
		\begin{equation*}	
			\mu(C) - \mu(h_z(\partial_l C))
			\leq \mu(h).
        \end{equation*}
        Once again by (P4) of \cref{goodbox},
        \begin{equation*}
	        1- \varepsilon_0 \leq \frac{\mu(h)}{\mu(C)}.
		\end{equation*}
		The fourth term is, once again, analogous.
	\end{proof}
    We also need the following version of the proposition in order to prove \cref{thmB}.
    \begin{proposition}
		\label{mixtubesB}
		Let $m$ be the normalized volume of $M$.
        Suppose that volume is almost exponentially mixing and let $\mu$ be the limit SRB measure.
		Let $ \varepsilon>0$ and $l_n,r_n,\xi_n,\overline{l}_n,\overline{r}_n,\overline{\xi}_n$
		be sequences of positive numbers converging to $0$ satisfying
        \begin{equation*}
            \xi_n^{1+\alpha} \ll l_n \ll r_n \ll \xi_n,
        \end{equation*}
        \begin{equation*}
             \overline{\xi}_n^{1+\alpha} \ll  \overline{l}_n \ll  \overline{r}_n \ll  \overline{\xi}_n
        \end{equation*}
        and
		\begin{equation*}
			e^{-\beta n} \ll \xi_n^{2d}\overline{\xi}_n^{2d}l_n\overline{l}_n.
		\end{equation*}
		Then, there exists $n_0 \in \mathbb{N}$ and $ \varepsilon_0 >0$ such that if
		$\left(C_j\right)_{j \in J(n)}$ and $\left(\overline{C}_i\right)_{i \in I(n)}$ are good boxes from \cref{goodbox} 
		for the sequences $\xi,r,l$ and $\overline{\xi},\overline{r},\overline{l}$ respectively and value $ \varepsilon_0$, then
		for all $n> n_0$, all $i \in I(n)$ and $j \in J(n)$,
		\begin{equation*}
			m\left(C_j \cap f^{- n}\overline{C}_i\right)
   \sim_{1+ \varepsilon}
   m\left(C_j\right)
   \mu\left(\overline{C}_i\right).
		\end{equation*}
	\end{proposition}
    It follows from the proof of \cref{mixtubes}, where
    instead of \eqref{1stapprox}, we use
    \begin{equation*}
        \frac{m(C \cap f^{- n}\overline{C} )}{m(C)\mu(\overline{C})}
			= \left(\frac{m(C \cap f^{ -  n}\overline{C})}{m(h g \circ f^{ - n})}\right)
			\left(\frac{m(h g \circ f^{ n})}{m(h)\mu(g)}\right)\left(\frac{m(h)}{m(C)}\right)\left(\frac{\mu(g)}{\mu(\overline{C})}\right)
    \end{equation*}
    and instead of \eqref{H1} we use \eqref{H2} when estimating the second term.
    The fact that $h$ $L^1(m)$-approximates the indicator function of $C$ follows from standard estimates using volume.

 \section{Construction of fake center stable manifolds}
 \label{fakecsconstructionsec}

We now wish to construct fake center stable foliations, similar to what is done in section 5 of \cite{ExpMix}.
In their paper, they pick a collection of good reference points $x$ and pull back the $h_{f^nx}(\mathbb{R}^{cs})$ foliation from time $n$ to time $0$ in a neighborhood of $x$.
The problem is that if $x$ and $x'$ are two good reference points close to each other, then they might define two different foliations in the intersections of their neighborhoods.
To overcome this problem, they create buffers between these neighborhoods. 
This comes at a price that the foliation will not be defined everywhere, but instead it will be defined on a set of measure close to $1$.
The smoothness of measure is heavily used for that step.

What we do instead is use the good boxes constructed in \cref{goodbox}.
In these boxes, we have the notion of center stable direction given by the reference point $z_j$ of the box.
Therefore, we pull back the $h_{z_j}(\mathbb{R}^{cs})$ foliation from time $n$ to time $0$ and intersecting it with good boxes at time $0$.
Since good boxes cover most of the space, the resulting foliation will also cover most of the space.

\begin{proposition}
	\label{fakecsconstruction}
    Suppose that $\mu$ is an ergodic SRB measure.
    Let $\varepsilon>0$ be small,
    $(\xi_n)_{n \in \mathbb{N}},(\overline{l}_n)_{n \in \mathbb{N}}, (\overline{r}_n)_{n \in \mathbb{N}},  (\overline{\xi}_n)_{n \in \mathbb{N}}$ be sequences of positive numbers
    satisfying
    \begin{equation}
    \begin{split}
	    \label{scaleslemma1}
        e^{-\beta n} 
        \ll r_n^{4d}l_n^2
        \ll (\xi_n)^{1+\alpha}
        \ll l_n
        \ll r_n
        \ll \xi_n \\
        \ll e^{2\sqrt{\delta}n}\xi_n
        \ll (\overline{\xi}_n)^{1+ \alpha}
        \ll \overline{l}_n
        \ll \overline{r}_n
        \ll \overline{\xi}_n
        \ll e^{- \delta n},
    \end{split}
    \end{equation}
    where $\delta = \varepsilon^{100^{100}}$.
    For $n$ large enough,
    let $\{C_j\}_{j \in J_n}$ and $\{\overline{C}_i\}_{i \in I_n}$ be good boxes from \cref{goodbox}
    for $(l,r,\xi,\varepsilon^{1000})$ and $(\overline{l},\overline{r},\overline{\xi},\varepsilon^{1000})$ respectively.
    For $i \in I_n$, let $\overline{\mathcal{F}}^{cs}_i$ be the foliation by vertical lines $h_{z_i}(\mathbb{R}^{cs})$
    on $\overline{C}_i$.
    Then, there exists $n_0 \in \mathbb{N}$ such that for $n>n_0$, there exists $J_n' \subset J_n$ and for $j \in J_n$,
    $I_n(j) \subset I_n$ satisfying 
    \begin{equation}
    \label{primesize}
        \mu\left(\underset{j \in J_n'}{\bigcup}C_j\right) > 1- \varepsilon^{80} \ \text{and} \ 
        \mu\left(\underset{i \in I_n(j)}{\bigcup}\overline{C}_i\right) > 1- \varepsilon^{3}
    \end{equation}
    such that, for $j \in J_n$, when breaking the following into its connected components:
    \begin{equation*}
        C_j \cap f^{-n}\left(\underset{i \in I_n(j)}{\bigcup} \overline{C}_i\right) = \underset{l \in L(j)}{\bigsqcup} G_l,
    \end{equation*}
    there exists $L'(j) \subset L(j)$ 
    satisfying $\mu_{C_j}(\bigcup_{l \in L'(j)}G_l) \geq 1-\varepsilon^2$ such that, for $x \in \bigsqcup_{l \in L'(j)}G_l$, denoting by
    \begin{equation*}
        \mathcal{F}^{cs,n}_j(x) = f^{-n} (\overline{\mathcal{F}}^{cs}_i(f^nx))\cap G_l,
    \end{equation*}
    then $\mathcal{F}^{cs,n}_j$ satisfies
     \begin{itemize}
        \item[(F1)] The induced foliation $\tilde{\mathcal{F}}^{cs,n}_j$ on $\mathcal{N}^u_{\xi}(c_j^u)\times \mathcal{N}^{cs}_r(c_j^{cs})$
        is made up of $(\alpha_3,O(\delta))$-mostly vertical leaves;
        \item[(F2)] 
        There exists $A = A(\varepsilon,\tau)>0$ such that if $W_1,W_2 \subset C_j$ are
        $(\alpha,1)$-mostly horizontal submanifolds satisfying $\|D\eta_{W_1}\|,\|D\eta_{W_2}\| \leq A$, then denoting by $\pi^{\tilde{\mathcal{F}}}_{W_1,W_2}$
        the holonomy along $\tilde{\mathcal{F}}^{cs,n}_j$ from $W_1$ to $W_2$, for $x \in W_1 \cap \tilde{\mathcal{F}}$,
        \begin{equation*}
            \text{Jac}_x(\pi^{\tilde{\mathcal{F}}}_{W_1,W_2}) \sim_{1 + \varepsilon} 1
        \end{equation*}
        \item[(F3)] For $x,y \in \mathcal{F}^{cs,n}_j(w)$ and $0 \leq k \leq n$,
        \begin{equation*}
            d(f^kx,f^ky) \leq \tau^{-2} e^{2\sqrt{\delta}k}d(x,y).
        \end{equation*}
    \end{itemize}
    We call such a foliation a $(\varepsilon,n)$-fake cs foliation.
\end{proposition}

		\begin{proof}

		For $i \in I$, denote by
		\begin{equation*}
			\overline{C}_i' = 
			h_{\overline{z}_i}(\mathcal{N}_{(\overline{\xi},\overline{r})}
			(\overline{c}_i)
			\setminus
			\partial_{\overline{l}}
			\mathcal{N}_{(\overline{\xi},\overline{r})}(\overline{c}_i)).
		\end{equation*}
		By (P4) of \cref{goodbox}, it satisfies
		\begin{equation*}
			(1- \varepsilon_I)\mu(\overline{C}_i)
			\leq \mu(\overline{C}_i').
		\end{equation*}
		We'll denote by 
		\begin{equation*}
			H = H(n,\tau,(\overline{C_i})_{i \in I}) = 	
			\mathcal{P}_{\tau}^c
			\cup f^{-n}\mathcal{P}_{\tau}^c
			\cup f^{-n}\left(\underset{i \in I}{\bigcup} \overline{C}_i'\right)^c.
		\end{equation*}
		We may also assume that $\mu(H) < \varepsilon^{100}$ is extremely small,
		by making $\tau$ and $ \varepsilon_I$ small.

		\subsubsection*{Excluding bad j-boxes}
		\label{badjbox}

		We first want the sets $\tilde{T}_j$ to intersect $H$ on a small set.
		By (P7) of \cref{goodbox} we have
		\begin{equation*}
			(1- \varepsilon_I)^2 
			< (1- \varepsilon_I) \sum_{i \in I} \mu(C_i)
			< \sum_{i \in I} \mu(\tilde{T}_i)
		\end{equation*}
		suppose that $ \varepsilon_I < \varepsilon^{100}$, then
		\begin{equation}
			\label{tildesbig}
			1- \varepsilon^{99} 
			<\sum_j \mu(\tilde{T}_j)
		\end{equation}
		Let 
        \begin{equation}
        \label{Jprimedef}
            J_n'= \{ j \in J_n \ | \ \mu_{\tilde{T}_j}(H) \leq \varepsilon^{10}\}.
        \end{equation}
		By \cref{AppendixLemma},
		\begin{equation*}
			1- \varepsilon^{80} < 1- \varepsilon^{100} - \varepsilon^{90} \leq \mu\left(\bigcup_{j \in J_n'} \tilde{T}_j \right),
		\end{equation*}
		which verifies the first part of \eqref{primesize}.
		For $j \in J_n'$, denote by 
		\begin{equation}
		\label{Kjdef}
			K_j = \tilde{T}_j \cap H^c \cap h_{z_j}( \mathcal{N}_{((1- \varepsilon_1)\xi,r)}(c_j))
		\end{equation}
		By (P8) of \cref{goodbox},
		the fact that $h_{z_i}(\mathbb{R}^{cs} + y)$, $y \in \mathbb{R}^u$ defines a foliation on $C_i$
		with the properties on (P8) for $n$ large enough 
		(since each $W_i(x)$ is the graph of a H\"older function with uniformly bounded H\"older constant),
		\begin{equation}
			\label{Leb1}
			\begin{split}
				\frac{\mu(\tilde{T}_j \cap h_{z_j}( \mathcal{N}_{((1 - \varepsilon_1)\xi,r)}(c_j))^c)}{\mu(\tilde{T}_j)}
			\leq O( \varepsilon_1) 
			\end{split}
		\end{equation}
		where the constant depends only on $\tau$ and the dimension of the unstable direction.
		Notice that by (P7), the condition on $j \in J'$ and \eqref{Leb1}, for $ \varepsilon_1>0$ small enough
		\begin{equation}
			\label{Kj}
			(1- \varepsilon^8) \mu(C_j) \leq (1- \varepsilon^{9})\mu(\tilde{T}_j)
			\leq \mu(K_j).
		\end{equation}

		\subsubsection*{Excluding bad i-boxes}
		\label{badibox}

		For each $j \in J_n'$,
		we want the sets $f^{-n}\overline{C}_i$ to intersect $K_j$ in a large portion.
		Notice that by \eqref{Jprimedef} and (P7) of \cref{goodbox},
		\begin{equation*}
			(1-\varepsilon^{10})(1-\varepsilon^{1000})\mu(C_j) <
            (1- \varepsilon^{10})\mu(\tilde{T}_j) < \mu\left(C_j \cap \underset{i \in I_n}{\bigcup}f^{-n}\overline{C}_i\right)
		\end{equation*}
        \begin{equation*}
            \implies (1-\varepsilon^9) < \mu_{C_j}\left(\underset{i \in I_n}{\bigcup}f^{-n}\overline{C}_i\right).
        \end{equation*}
		Let
        \begin{equation}
        \label{Iprime}
        I_n(j) = \{ i \in I_n \ | \ (\mu_j)_{f^{-n}\overline{C}_i}
		(K_j^c) \leq \varepsilon^{5}\}.
        \end{equation}
		Then by \cref{AppendixLemma},
		\begin{equation*}
			1 - \varepsilon^3
			\leq 1- \varepsilon^{9} - \frac{ \varepsilon^{9}}{ \varepsilon^{5}}
			\leq\mu_{C_j}\left(\bigcup_{i \in I'(j)} f^{-n}\overline{C}_i\right).	
		\end{equation*}
		This verifies the second part of \eqref{primesize}.

		\subsubsection*{Excluding bad components}
		\label{badcomponents}
		For $ j \in J_n'$ and $i \in I_n(j)$, we break the following sets into its connected components
        \begin{equation*}
            C_j \cap f^{-n}\overline{C}_i 
		  = \bigsqcup_{l \in L(i,j)} G_l \ \text{and} \ 
        \bigcup_{i \in I'_n(j)}C_j \cap f^{-n}\overline{C}_i = 
        \bigsqcup_{l \in L(j)} G_l
        \end{equation*}
		Let 
		\begin{equation}
		\label{goodcomponents}
			L'(j) = \{ l \in L(j) \ | \ G_l \cap K_j \neq \emptyset\}.
		\end{equation}
        Denote by $\mathcal{G}(j) = \bigcup_{l \in L'(j)} G_l$.
		Notice that $l \in L(j) \setminus L'(j)$ means that $G_l \subset K_j^c$.
		Therefore, by \eqref{Iprime} and \eqref{Kj},
		\begin{equation*}
			1 - 2\varepsilon^3 \leq \mu_{C_j}(\mathcal{G}(j)).
		\end{equation*}

		\subsubsection*{Good components are contained in fake cs-foliations}
		\label{containedinfakecs}
		Now fix $j \in J'_n$, $l \in L'(j)$ and take $i \in I_n'(j)$ such that $G_l \subset C_j \cap f^{-n}\overline{C}_i$.
		By assumption, there exists $w_l \in K \cap G_l$. In particular,
		$w_l \in \mathcal{P}_{\tau}$ and $f^n(w_l) \in \mathcal{P}_{\tau}$.
		Consider the foliation on $\mathbb{R}^d$:
		\begin{equation*}
			\mathcal{F}(y) = h_{f^nw_l}^{-1}\circ h_{z_i} 
			(\mathbb{R}^{cs}+ h_{z_i}^{-1} \circ h_{f^nw_l}(y)).
		\end{equation*}
		By \cref{changecoordinates}, 
		it satisfies the hypothesis of \cref{fakecs} for $n$ large enough.
		Therefore, the foliation pulled back by $\tilde{f}^{(n)}_{w_l}$ is such that 
		each leaf is a graph of a $C^{1+\alpha_3}$ function
		\begin{equation*}
			\eta^{cs,n}_{\mathcal{F},w_l,y} \colon \mathbb{R}^{cs} \to \mathbb{R}^u.
		\end{equation*}
		Given $\overline{z} \in \tilde{W}_{\tau}^u(w_l)$ such that
		$\tilde{f}^{(n)}_{w_l}(\overline{z}) 
		\in h_{f^nw_l}^{-1} \circ h_{z_i} (\mathcal{N}_{\overline{\xi},\overline{r}}(c_i))$
		, we have
		\begin{equation*}
			|\tilde{f}^{(n)}_{w_l} \overline{z}| \leq 2\overline{\xi} \ll e^{-\delta n}
		\end{equation*}
		Given $w \in \mathcal{N}^{cs}_{2\overline{\xi}}(0) \subset \mathbb{R}^{cs}$,
		then $w$ and $\overline{z}$ satisfy \eqref{fakecscondition},
		for $n$ large enough.
		Therefore, by \cref{fcommute},
		\begin{equation*}
			\tilde{f}^{(n)}_{w_l}(\eta^{cs,n}_{\mathcal{F},w_l,\overline{z}}(w),w)
			= h^{-1}_{f^nw_l} \circ f^n \circ h_{w_l} (\eta^{cs,n}_{\mathcal{F},w_l,\overline{z}}(w),w).
		\end{equation*}
		This implies that for such $\overline{z}$,
		\begin{equation*}
			 h_{w_l}(\text{graph}(\eta^{cs,n}_{\mathcal{F},w_l,\overline{z}}|_
			{|w| < 2\overline{\xi}}))
			\subset f^{-n} \circ h_{f^nw_l}
			(\mathcal{F}(\tilde{f}^{(n)}_{w_l}\overline{z})).
		\end{equation*}
		Now given $z \in G_l$, it is in
		$f^{-n} \circ h_{f^nw_l}(\mathcal{F}(\tilde{f}^{(n)}_{w_l}\overline{z}))$
		for some $\overline{z}$ as above.
		Since $|\pi^{cs}(h_{w_l}^{-1}(z))| < 2\overline{\xi}$,
		\begin{equation*}
			z = h_{w_l}(\eta^{cs,n}_{\mathcal{F},w_l,\overline{z}}
		(\pi^{cs}(h_{w_l}^{-1}(z)),\pi^{cs}(h_{w_l}^{-1}(z))).
		\end{equation*}
		This shows that
		\begin{equation*}
			G_l \subset \underset{\overline{z}}{\bigcup} h_{w_l}
			(\text{graph}(\eta^{cs,n}_{\mathcal{F},w_l,\overline{z}}|_{\mathcal{N}_{2\overline{\xi}}(0)})).
		\end{equation*}
		Using \cref{changecoordinates} for $h^{-1}_{z_j} \circ h_{w_l}$,
        we have that
		$h_{z_j}^{-1}(G_l)$ is contained in the union of graphs of $C^{1+\alpha_3}$ functions 
		$\eta^{cs}_y \colon \mathbb{R}^{cs}\cap \mathcal{N}_r(c_j^{cs}) \to \mathbb{R}^u$,
		for $y \in W_i(f^n(w_l)) = W^u(f^n(w_l)) \cap \overline{C}_i$
		satisfying the following properties:
		\begin{enumerate}
			\item there exists $u_l \in \mathcal{N}_r(c_j^{cs})$  such that
				\begin{equation*}
					h_{z_j}(\eta^{cs}_{f^n(w_l)}(u_l),u_l) = w_l
				\end{equation*}
				and consequently, since $w_l \in K_j$,
				\begin{equation*}
				|\eta^{cs}_{f^n(w_l)}(u_l) - c_j^u| < (1- \varepsilon_1)\xi;
				\end{equation*}

			\item for each $y$, $\underset{w \in \mathbb{R}^{cs}}{\sup}\|D_w\eta_y^{cs}\|
				< C\frac{3 \delta}{1 - e^{-\lambda + \sqrt{\delta}}}$.
		\end{enumerate}
        This will show (F1).
		For $y \in W_i(f^nw_l)$, let $u_y \in \mathcal{N}_r{c_j^{cs}}$ be such that
		\begin{equation*}
			f^n(h_{w_l}(\eta^{cs}_y(u_y),u_y)) = y.
		\end{equation*}
		By \cref{unstable} and using that $f^n(w_l) \in \mathcal{P}_{\tau}$,
		\begin{equation*}
			|\eta^{cs}_y(u_y)-\eta^{cs}_{f^nw_l}(u_l)|
			\leq \tau^{-2} e^{(-\lambda+\sqrt{\delta})n}\overline{\xi}.
		\end{equation*}
		In particular, for $w \in \mathcal{N}^{cs}_r{c_j}$ and $y \in W_i(f^nw_l)$,
		\begin{equation*}
			\begin{split}
			|\eta_y^{cs}(w)-c_j^u| 
			\leq &|\eta_y^{cs}(w) - \eta_y^{cs}(u_y)| 
			+ |\eta_y^{cs}(u_y) - \eta_{f^nw_l}^{cs}(u_l)|
			+ |\eta_{f^nw_l}^{cs}(u_l) - c_j^u| \\
			\leq & C\frac{3 \delta}{1 - e^{-\lambda + \sqrt{\delta}}}r 
			+ \tau^{-2}e^{(-\lambda+\sqrt{\delta})n}\overline{\xi} 
			+ (1- \varepsilon_1)\xi.
			\end{split}
		\end{equation*}
		The value above is smaller than $\xi$, for $n$ large enough by \eqref{scaleslemma1}.
		This shows that the graphs are actually contained in
		$\mathcal{N}_{(\xi,r)}(c_j)$.

		\subsubsection*{Components fully cross $C_j$}
		For $z \in W^u(w_l)\cap f^{-n}(W^u(f^nw_l)\cap\overline{C}_i) = W^u_l$, denote by 
		\begin{equation*}
		\mathcal{F}^{cs}(z) = h_{w_l}(\text{graph}(\eta^{cs,n}_{\mathcal{F},w_l,h_{w_l}^{-1}(z)}))\cap C_j.
		\end{equation*}
		We showed above that
		\begin{equation}
		      \label{insidefols}
			G_l \subset \underset{z \in W^u_l}{\bigcup} \mathcal{F}^{cs}(z).
		\end{equation}
		We will now show equality. This will imply that $G_l$ fully crosses $C_j$,
		since each leaf is will be the graph of a function inside that box.
		If $z \in W^u_l$, then $z \in C_j$ and $f^n(z) \in \overline{C}_i$. Therefore, $z \in G_l$ for some $l$.
		Now for such $z$, let
		$\hat{\mathcal{F}}^{cs}(z)$ be the connected component of $C_j \cap f^{-n}(\mathcal{F}(f^n(z))\cap \overline{C}_i)$ that contains $z$.
		By construction, $\hat{\mathcal{F}}^{cs}(z) \subset \mathcal{F}^{cs}(z)$.
		If they're not equal, then there exists $q \in \mathcal{F}^{cs}(z)$ such that
		$f^n(q) \in \partial \overline{C}_i$. Then by \cref{subexpM} and \eqref{scaleslemma1}, for $n$ large enough
		\begin{equation*}
			d(f^nz,f^nq) \leq \tau^{-2}e^{2\sqrt{\delta}k}r_n < \tau^2\overline{l}_n.
		\end{equation*}
		However, if $w_l \in K$, then, $f^n(w_l) \notin \overline{C}_i'$. Therefore, we must have $\hat{\mathcal{F}}^{cs}(z) = \mathcal{F}^{cs,n}(z)$, which shows equality in \cref{insidefols}.
        (F3) follows from \cref{subexpM}.

		\subsubsection*{Fake cs holonomies}
		\label{fakecsholonomies}
        Let $j \in J'$ and
        $W_1,W_2 \subset C_j$ be $(\alpha,1)$-horizontal submanifolds.
        Let $\pi_j \colon W_1 \to W_2$ be the holonomy along $\mathcal{F}^{cs,n}_j$, wherever it is well defined.
        Let $x \in W_1$ such that $\pi(x)$ is well defined.
        Then $x \in G_l$, for some $l \in L'(j)$.
        Let $\tilde{\mathcal{F}}_l^{cs,n}$ be the foliation on $\mathbb{R}^d$ induced via $h_{w_l}$ and $\tilde{W}_i = h_{w_l}^{-1} (W_i)$, $i=1,2$.
        Let $\tilde{\pi}_l \colon \tilde{W}_1 \to \tilde{W}_2$ be the holonomy along $\tilde{\mathcal{F}}_l^{cs,n}$ leaves.
        It suffices to show that $\text{Jac}_{h_{w_l}^{-1}(x)}(\tilde{\pi}_l)$ is close to 1.
        
        Let $h_{w_l}^{-1}(x) = \overline{x}_1$, $\tilde{\pi}_l(\overline{x}_1) = \overline{x}_2$ and $\mathcal{F}^{cs}_l$ be the foliation on $\mathbb{R}^d$ obtained by pushing forward the $\mathbb{R}^{cs}$ foliation via $ h_{f^nw_l}^{-1} \circ h_{z_i}$. 
        Let also $\pi^{cs} \colon \tilde{f}_{w_l}^{(n)}(\tilde{W}_1) \to \tilde{f}_{w_l}^{(n)}(\tilde{W}_2)$ be the holonomy map along this foliation.
        Let $L_i \colon \mathbb{R}^u \to \mathbb{R}^{cs}$ be the linear map such that $\text{graph}(L_i) = T_{\overline{x}_i}\tilde{W}_i$, for $i=1,2$.
        
        Since $\tilde{\pi}_l = \tilde{f}^{(-n)}_{w_l} \circ \pi^{cs} \circ \tilde{f}_{w_l}^{(n)}$,
        \begin{equation}
        	\label{Jacobianest}
        	\text{Jac}_{\overline{x}_1}(\tilde{\pi}_l) = 
        	\text{Jac}_{\tilde{f}_{w_l}^{(n)}\overline{x}_1}(\pi^{cs})
       		\frac{\text{Jac}(D_{\overline{x}_1}\tilde{f}_{w_l}^{(n)}|_{T_{\overline{x}_1}\tilde{W}_1})}
       		{\text{Jac}(D_{\overline{x}_2}\tilde{f}_{w_l}^{(n)}|_{T_{\overline{x}_2}\tilde{W}_2})}.
        \end{equation}
        For the second term, we use \cref{UsefulJac} to bound the absolute value of its $\log$ by
        \begin{equation}
        \label{Jacobianest2}
        	C\sum_{k=0}^{n-1}\delta|\tilde{f}_{w_l}^{(k)}\overline{x}_1-\tilde{f}_{w_l}^{(k)}\overline{x}_2|^{\alpha}
        + e^{(-\lambda+\sqrt{\delta})k}\|L_1-L_2\|
        +6\delta \sum_{i=1}^k e^{(-\lambda+\sqrt{\delta})(k-i)}|\tilde{f}_{w_l}^{(i)}\overline{x}_1-\tilde{f}_{w_l}^{(i)}\overline{x}_2|^{\alpha_7}.
        \end{equation}
        By (F3) and (2) of \cref{LyapCharts},
        \begin{equation*}
        	|\tilde{f}_{w_l}^{(k)}\overline{x}_1-\tilde{f}_{w_l}^{(k)}\overline{x}_2| \leq \tau^{-6} r_n e^{2\sqrt{\delta}k}.
        \end{equation*}
        Let $A = \frac{\varepsilon^{100}}{C}(1-e^{-\lambda+\sqrt{\delta}})$.
        Since $W_1,W_2 \subset C_j$ are $(\alpha,1)$-horizontal whose graphs have derivatives bounded by $A>0$,
        then by \cref{graphid}, there exists $K>0$ such that $\tilde{W}_1,\tilde{W}_2$ are graphs of functions with derivative bounded by
        \begin{equation}
        \label{derivative}
        	\frac{\varepsilon^{100}}{C}(1-e^{-\lambda+\sqrt{\delta}})+4K\xi_n^{\alpha_2}.
        \end{equation}
        Therefore, for $n$ large $\|L_1-L_2\| \leq \frac{\varepsilon^{50}}{C}(1-e^{-\lambda+\sqrt{\delta}})$.
        Putting these all together, \eqref{Jacobianest2} is bounded by
        \begin{equation*}
        	C\sum_{k=0}^{n-1}\delta(\tau^{-6} r_n e^{2\sqrt{\delta}k})^{\alpha}
        + \frac{\varepsilon^{50}}{C}(1-e^{-\lambda+\sqrt{\delta}})e^{(-\lambda+\sqrt{\delta})k}
        +6\delta \sum_{i=1}^k e^{(-\lambda+\sqrt{\delta})(k-i)}(\tau^{-6} r_n e^{2\sqrt{\delta}k})^{\alpha_7}.
        \end{equation*}
        Therefore, there exists some constant $C' = C'(\tau)$ such that the above is bounded by
        \begin{equation*}
        	C'\left(\frac{e^{2\alpha\sqrt{\delta}n}-1}{e^{2\alpha\sqrt{\delta}}-1}r_n^{\alpha}   	
        	+ \frac{e^{2\alpha_7\sqrt{\delta}n}-1}{e^{2\alpha_7\sqrt{\delta}}-1}r_n^{\alpha_7}\right)
        	+ \varepsilon^{50}.
        \end{equation*}
        Since $r_n \ll e^{2\sqrt{\delta}n}$, for $n$ large enough the above is less than $\varepsilon^{40}$.
        
        For the first term of \eqref{Jacobianest},
        let $\varphi = h_{f^nw_l}^{-1}\circ h_{z_i}$ and $\pi^{cs}_i \colon \varphi^{-1}\circ \tilde{f}^{(n)}_{w_l}(\tilde{W}_1) \to \varphi^{-1}\circ \tilde{f}^{(n)}_{w_l}(\tilde{W}_2)$ be the holonomy along the $\mathbb{R}^{cs}$ foliation.
        Then, omitting at which point we're taking the jacobian,
        \begin{equation*}
        	\text{Jac}(\pi^{cs})
        	= \text{Jac}(\varphi|_{\varphi^{-1}\tilde{f}^{(n)}_{w_l}(\tilde{W}_1)})
        	\text{Jac}(\pi^{cs}_i)
        	\text{Jac}(\varphi^{-1}|_{\tilde{f}^{(n)}_{w_l}(\tilde{W}_1)}).
        \end{equation*}
        The jacobians of $\varphi$ and $\varphi^{-1}$ are close to $1$ for $n$ large by \cref{changecoordinates}.
        Let $\eta_{n,i}$ be the function whose graph is $\tilde{f}^{(n)}_{w_l}(\tilde{W}_i)$, for $i=1,2$.
        By Lemma A.10 of \cite{ExpMix} and \eqref{derivative},
        \begin{equation*}
        	\|D\eta_{n,i}\|
        	\leq e^{n(-\lambda +\sqrt{\delta})}\left(\frac{\varepsilon^{100}}{C}(1-e^{-\lambda+\sqrt{\delta}})+4K\xi_n^{\alpha_2}\right)
        	+ 2\delta\overline{\xi}_n^{\alpha_7}\frac{1}{ 
        	1-e^{ -\lambda+\sqrt{\delta} } }.
        \end{equation*}
        Therefore, for $n$ large enough, the jacobian of $\pi^{cs}_i$ is close to $1$ since the manifolds $\varphi^{-1}\tilde{f}^{(n)}_{w_l}(\tilde{W}_1)$ and $\varphi^{-1}\tilde{f}^{(n)}_{w_l}(\tilde{W}_2)$ are horizontal enough.
        This concludes the proof. 
        \end{proof}

	\section{From mixing to equidistribution}

	\subsection{Equidistribution of unstable manifolds along exponentially small boxes}	
	We will now also assume that $\mu$ is an SRB measure.
	We wish to show that equidistribution of thin tubes implies equidistribution of unstables. The following statement is analogous to the Main Proposition in \cite{ExpMix}.
\begin{lemma}
	\label{MainLemma}
    Suppose either that $f$ is exponentially mixing with respect to an SRB measure $\mu$
    or that volume is almost exponentially mixing and let $\mu$ be the limit SRB measure.
    Let $\varepsilon>0$ be small,
    $H_n \subset M$ be a sequence of subsets with $\mu(H_n) < \varepsilon^{1000}$,
    $(\xi_n)_{n \in \mathbb{N}},(\overline{l}_n)_{n \in \mathbb{N}}, (\overline{r}_n)_{n \in \mathbb{N}},  (\overline{\xi}_n)_{n \in \mathbb{N}}$ be sequences of positive numbers
    satisfying
    \begin{equation}
	    \label{scaleslemma}
        e^{-\beta n} 
        \ll \xi_n^{4d+2}
        \ll e^{2\sqrt{\delta}n}\xi_n
        \ll (\overline{\xi}_n)^{1+ \alpha}
        \ll \overline{l}_n
        \ll \overline{r}_n
        \ll \overline{\xi}_n
        \ll e^{- 2\sqrt{\delta} n}
    \end{equation}
    where $\delta = \varepsilon^{100^{100}}$.
    Let $\{\overline{C}_i\}_{i \in I_n}$ be good boxes from \cref{goodbox} for $(l,r,\xi,\varepsilon^{1000})$.
    Then, there exists $n_0 = n_0(\overline{l},\overline{r},\overline{\xi},\xi,\varepsilon)>0 $
    such that for each $n>n_0$, 
    there exists $\mathcal{K}_n \subset M$
    and a measurable u-subordinate partition $\{W_n(x)\}_{x \in T_n}$ of $\mathcal{K}_n$ indexed by a a family of transversals $T_n$ with
    \begin{enumerate}
	    \item[(U1)] $\mu(\mathcal{K}_n) > 1 - \varepsilon$;
	    \item[(U2)] $W_n(x) \subset W^u(x)$ is a piece of unstable manifold of size $\sim_{1+ \varepsilon} \xi_n$, for each $x \in T_n$;
	    \item[(U3)] For each $x \in T_n$, $\mu_{W_n(x)}(H_n) < \varepsilon$;
	\end{enumerate}
	such that if $\{\overline{C}_i\}_{i \in I_n}$ are $(\overline{l},\overline{r},\overline{\xi},\varepsilon^{1000})$-Nice boxes from \cref{goodbox}, then
	\begin{enumerate}
		\item[(E1)] For each $x \in T_n$, there exists $I^*_n(x) \subset I_n $ such that
        \begin{equation*}
            \mu\left(\underset{i \in I_n(x)}{\bigcup} \overline{C}_i\right) > 1 - \varepsilon
        \end{equation*}
    such that for  $i \in I^*_n(x)$,
        \begin{equation*}
            \mu_{W_n(x)}(f^{-n}\overline{C}_i) 
            \sim_{1+ \varepsilon} \mu(\overline{C}_i).
        \end{equation*}
    \end{enumerate}
 
\end{lemma}    

		\begin{proof}
        Let $(l_n)_{n \in \mathbb{R}},
		(r_n)_{n \in \mathbb{R}}$ satisfying
		\begin{equation*}
			\xi^{1+\alpha} \ll l_n \ll r_n \ll \xi_n
		\end{equation*}
		and let $\{C_j\}_{j \in J_n}$
		be $(l,r,\xi, \varepsilon^{1000})$-nice boxes.
        Let
        \begin{equation*}
            \mathcal{K}_n^* = \bigcup_{j \in J_n'}\tilde{T}_j \ \text{and} \ T_n^* = \bigcup_{j \in J_n'} T_j. 
        \end{equation*}
        Just as (P8) of \cref{goodbox}, let $\{\mu_{W(x)}\}_{x \in T_n^*}$ be a disintegration of $\mu$ with respect to the measurable partition $\{W_n(x)\}_{x \in T_n^*}$ and $\nu$ be the measure induced on $T_n^*$.
        Note that $\mu(H_n) < \varepsilon^{1000}$ and, by (P7) of \cref{goodbox}, $\mu(K_n^*)>(1-\varepsilon^{1000})^2$. Let
        \begin{equation}
            \label{Tndef}
            T_n = \{x \in T_n^* \ | \ \mu_{W_n(x)}(H_n) < \varepsilon \}.
        \end{equation}
        Denote by $K_n = \bigcup_{x \in T_n} W_n(x)$. By \cref{AppendixLemma},
        \begin{equation*}
            \mu(K_n) > 1- \varepsilon.
        \end{equation*}
        The constructed $K_n$ along with the measurable partition defined satisfies properties (U1), (U2) and (U3).
		Let $j \in J_n'$.
        Now let $n$ be large so that \cref{fakecsconstruction} holds for the families $\{C_j\}_{j \in J_n}$ and $\{\overline{C}_i\}_{i \in I_n}$.
        Let
		\begin{equation*}
			C_j \cap f^{-n}(\overline{C}_i) = \bigsqcup_{l \in L(j,i)} G_l.
		\end{equation*}
        Recall the definition of $K_j$ in \eqref{Kjdef} and the construction of $L'(j)$ in \eqref{goodcomponents}.
        Let $\mathcal{G}(j) = \bigcup_{l \in L'(j)}G_l$ and for $i \in I_n(j)$, $\mathcal{G}(j,i) = \bigcup_{l \in L'(j,i)}G_l$, where $L'(j,i) = L'(j) \cap L(j,i)$. We will now separate the proof into the two cases in our hypothesis.
        \subsubsection*{SRB exponentially mixing}
        For $x \in T_n$ $\nu$-generic, let
        \begin{equation}
            \label{IstarA}
            I_n^*(x) = 
            \{ i \in I_n(j) \ | \ 
            (\mu_{W(x)})_{f^{-n}\overline{C}_i}(\mathcal{G}(j,i)^c) \leq \varepsilon \}.
        \end{equation}
        Since $\mu_{W(x)}(\mathcal{G}(j)^c) <  \varepsilon^2$ and 
        $\mu_{W(x)}\left(\bigcup_{i \in I_n(j)}f^{-n}\overline{C}_i\right) > 1-\varepsilon^2$, then
        \begin{equation}
            \label{IstarAbig}
            \mu_{W(x)}\left(\bigcup_{i \in I_n^*(x)}f^{-n}\overline{C}_i\right) > 1-\varepsilon.
        \end{equation}
        Note that $i \in I_n^*(x)$ means
        \begin{equation*}
            (1-\varepsilon)\mu_{W(x)}(f^{-n}\overline{C}_i) < \mu_{W(x)}(\mathcal{G}(j,i)) \leq \mu_{W(x)}(f^{-n}\overline{C}_i).
        \end{equation*}
        In particular, for $i \in I_n^*(x)$,
        \begin{equation}
            \label{Istarcond}
            \mu_{W(x)}(\mathcal{G}(j,i)) \sim_{1+2\varepsilon}\mu_{W(x)}(f^{-n}\overline{C}_i).
        \end{equation}
        Fix $i \in I_n^*(j)$.
		Notice that $l \in L(j)\setminus L'(j)$ implies that $G_l \subset K_j^c$.
		Therefore, by \eqref{Iprime},
			\begin{equation*}
				\mu_{C_j}\left(\underset{l \in L(j) \setminus L'(j)}{\bigcup}G_l
					\cap f^{-n}\overline{C}_i\right) 
					\leq \mu_{C_j}(f^{-n}\overline{C}_i) \varepsilon^5
			\end{equation*}
			\begin{equation*}
				\implies
				\mu\left(\underset{l \in L(j,i) \setminus L'(j)}{\bigcup}G_l\right) 
				\leq \mu(C_j\cap f^{-n}\overline{C}_i) \varepsilon^5.
			\end{equation*}
            Since $i \in I_n(j)$,
            \begin{equation}
            \label{saturatedGl3}
                \mu_{C_j}\left(\underset{l \in L'(j,i)}{\bigcup}G_l\right)
                \sim_{1+2\varepsilon} \mu_{C_j}(f^{-n}\overline{C}_i).
            \end{equation}
			Since $\mathcal{G}(j,i) = \bigcup_{l \in L'(i,j)}G_l$ is fake cs-saturated, then by (P8) of \cref{goodbox},
		\begin{equation}
        \label{equiGl}
			\mu_{\tilde{T}_j}(\mathcal{G}(j,i))
			\sim_{1+ \varepsilon} \mu_{W}(\mathcal{G}(j,i)).
		\end{equation}
        By (P7) of \cref{goodbox},
        \begin{equation*}
            \mu(C_j) \sim_{1+O(\varepsilon)} \mu(\tilde{T}_j).
        \end{equation*}
        Inverting the above and multiplying by $\mu(\mathcal{G}(j,i)\cap \tilde{T}_j)$, we obtain
        \begin{equation}
            \label{TjCjsame}
            \mu_{\tilde{T}_j}(\mathcal{G}(j,i)) \sim_{1+O(\varepsilon)} \mu_{C_j}(\mathcal{G}(j,i)\cap \tilde{T}_j).
            \end{equation}
        Using \eqref{Iprime} once more and recalling that $\tilde{T}_j^c \subset K_j^c$, we get
        \begin{equation*}
            (\mu_{C_j})_{f^{-n}\overline{C}_i}(\tilde{T}_j) \sim_{1+O(\varepsilon)} 1.
        \end{equation*}
        Combining the above with \eqref{saturatedGl3}, we get
        \begin{equation*}
            (\mu_{C_j})_{f^{-n}\overline{C}_i}(\mathcal{G}(j,i)\cap \tilde{T}_j) \sim_{1+O(\varepsilon)} 1.
        \end{equation*}
        Therefore,
        \begin{equation}
            \label{saturatedGl4}
            \mu_{C_j}(\mathcal{G}(j,i)\cap \tilde{T}_j) \sim_{1+O(\varepsilon)} \mu_{C_j}(f^{-n}(C_i)).
        \end{equation}
        By \cref{mixtubes},
		\begin{equation}
			\label{CjCimix}
			\mu_{C_j}(f^{-n}\overline{C}_i)
			\sim_{1+ \varepsilon} \mu(\overline{C}_i).
		\end{equation}
        Finally, combining \eqref{Istarcond},  \eqref{equiGl}, \eqref{TjCjsame}, \eqref{saturatedGl4},  and  \eqref{CjCimix} we get
        \begin{equation*}
            \mu_{W(x)}(f^{-n}\overline{C}_i) \sim_{1+O(\varepsilon)} \mu(\overline{C}_i).
        \end{equation*}
        Summing over $i \in I_n^*(x)$ and using \eqref{IstarAbig},
        \begin{equation*}
            \mu\left(\bigcup_{i \in I_n^*(x)} \overline{C}_i\right)
            \sim_{1+O(\varepsilon)} \mu_{W(x)}\left(\bigcup_{i \in I_n^*(x)}f^{-n}\overline{C}_i\right) \sim_{1+2\varepsilon} 1.
        \end{equation*}
        This shows (E1) and concludes the proof for exponentially mixing SRB measures.
        \subsubsection*{Volume almost exponentially mixing}
        Let
        \begin{equation}
            \label{Istar}
            I_n^*(j) = 
            \{ i \in  I_n(j) \ | \ (m_{C_j})_{f^{-n}(\overline{C}_i)}(\mathcal{G}(j)^c)\leq \varepsilon^{\frac{1}{2}}\}.
        \end{equation}
        We have
        \begin{align*}
            m_{C_j}(\mathcal{G}(j)) & \sim_{1+\varepsilon^{999}} \mu_{W(x)}(\mathcal{G}(j)), & \text{by (P8) of \cref{goodbox},} \\
            & \sim_{1+ \varepsilon^{999}} \mu_{\tilde{T}_j}(\mathcal{G}(j)), & \text{by (P8) of \cref{goodbox},} \\
            & \sim_{1+O(\varepsilon)}\mu_{C_j}(\mathcal{G}(j)\cap \tilde{T}_j), & \text{by \eqref{TjCjsame},} \\
            & \sim_{1+O(\varepsilon)}\mu_{C_j}\left(f^{-n}\bigcup_{i \in I_n(j)}\overline{C}_i\right), & \text{by \eqref{saturatedGl4},}\\
            & \sim_{1+\varepsilon^2} 1, & \text{by \cref{fakecsconstruction}}.
        \end{align*}
        Since $m_{C_j}(\mathcal{G}(j)) \leq m_{C_j}\left(\bigcup_{i \in I_n(j)}f^{-n}\overline{C}_i\right) \leq 1$, then we also have
        $m_{C_j}\left(\bigcup_{i \in I_n(j)}f^{-n}\overline{C}_i\right) \sim_{1+O(\varepsilon)} 1$.
        By \cref{AppendixLemma},
        \begin{equation}
            \label{IstarBbig1}
            m_{C_j}\left(\bigcup_{i \in I_n^*(j)}f^{-n}\overline{C}_i\right) \sim_{1+O(\varepsilon^{\frac{1}{2}})} 1.
        \end{equation}
        For $x \in T_n \cap T_j$ $\nu$-generic, let
        \begin{equation}
            \label{IstarB2}
            I_n^*(x) = 
            \{ i \in I_n^*(j) \ | \ 
            (\mu_{W(x)})_{f^{n}\overline{C}_i}(\mathcal{G}(j,i)^c) \leq \varepsilon^{\frac{1}{4}}\}.
        \end{equation}
        By \eqref{IstarBbig1} and \cref{AppendixLemma},
        \begin{equation}
            \label{IstarBbig2}
            \mu_{W(x)}\left(\bigcup_{i \in I_n^*(x)}f^{-n}\overline{C}_i\right) \sim_{1+O(\varepsilon^{\frac{1}{4}})}1.
        \end{equation}
        If $i \in I_n^*(x)$, then
        \begin{equation}
            \label{IstarBstep1}
            \mu_{W(x)}(f^{-n}\overline{C}_i) \sim_{1+O(\varepsilon^{\frac{1}{4}})} \mu_{W(x)}(\mathcal{G}(j,i)).
        \end{equation}
        Since $\mathcal{G}(j,i)$ is $\mathcal{F}^{cs,n}$-saturated, then applying (P8) of \cref{goodbox},
        \begin{equation}
            \label{IstarBstep2}
            \mu_{W(x)}(\mathcal{G}(j,i))\sim_{1+ \varepsilon}m_{C_j}(\mathcal{G}(j,i))  .
        \end{equation}
        Since $i$ is also in $I_n^*(j)$, then
        \begin{equation}
        \label{IstarBstep3}
            m_{C_j}(\mathcal{G}(j,i)) \sim_{1+O(\varepsilon^{\frac{1}{2}})} 
            m_{C_j}(f^{-n}\overline{C}_i).
        \end{equation}
        By \cref{mixtubesB},
        \begin{equation}
        \label{IstarBstep4}
            m_{C_j}(f^{-n}\overline{C}_i) \sim_{1+O(\varepsilon)} \mu(C_i).
        \end{equation}
        Putting together \eqref{IstarBstep1}, \eqref{IstarBstep2}, \eqref{IstarBstep3} and \eqref{IstarBstep4}, we get
        \begin{equation*}
            \mu_{W(x)}(f^{-n}\overline{C}_i) \sim_{1+O(\varepsilon^{\frac{1}{4}})} \mu(\overline{C}_i).
        \end{equation*}
        Now \eqref{IstarBbig2} and the above imply that
        \begin{equation*}
            \mu\left(\bigcup_{i \in I_n^*(x)}\overline{C}_i\right) \sim_{1+O(\varepsilon^{\frac{1}{4}})} 1.
        \end{equation*}
        This, together with the previous estimate, implies (E1) and finishes the proof.
        \end{proof}

		\subsection{Equidistribution of unstable manifolds along thin sets}

		We now pass from equidistribution of unstables along exponentially small boxes from the previous section 
		to equidistribution along thin fake center stable saturated subsets of those exponentially small boxes.

		\begin{lemma}
		\label{MainLemma2}
        Suppose either that $f$ is exponentially mixing with respect to an SRB measure $\mu$
    or that volume is almost exponentially mixing and let $\mu$ be the limit SRB measure.
    Let $\varepsilon>0$ and $(\xi_m)_{m \in \mathbb{N}}$ be a sequence of positive numbers satisfying
\begin{equation}
\label{scales52}
    e^{-\frac{\beta}{4d+2}m}
    \ll \xi_m
    \ll e^{-(\delta+2\sqrt{\delta})m},
\end{equation}
where $\delta = \varepsilon^{100^{100}}$.
Then, there exists $n_0 = n_0(\xi,\varepsilon)>0 $
    such that for each $n>n_0$, 
    there exists $\mathcal{K}_n \subset M$,
    a measurable u-subordinate partition $\{W_n(x)\}_{x \in T_n}$ of $\mathcal{K}_n$ and 
    a finite family of pairwise disjoint subsets $\{B_j\}_{j \in J_n}$ of $M$ satisfying
    \begin{enumerate}
	    \item[(U1)] $\mu(\mathcal{K}_n) > 1 - \varepsilon$;
	    \item[(U2)] $W_n(x) \subset W^u(x)$ is a piece of unstable manifold of size $\sim_{1+ \varepsilon} \xi_{\floor{\varepsilon n}}$, for each $x \in T_n$;
			\item[(B1)]  for $x \in T_{n}$,
				there exists $J_{n}'(x) \subset J_{n}$ such that
				\begin{equation*}
					\mu\left(\underset{i \in J_{n}'(x)}{\bigcup}B_i\right) 
					> 1 - \varepsilon
				\end{equation*}
			and for $j \in J_{n}'(x)$,
			\begin{equation}
				\label{EE}
				\mu_{W_{n}(x)}(f^{-\floor{\varepsilon n}}B_j)
				\sim_{1+\varepsilon} \mu(B_j).
			\end{equation}
			\item[(B2)]  for $j \in J_{n}$, $x,y \in B_j$ and $0 \leq k \leq n$,
				\begin{equation*}
				d(f^kx,f^ky) < \varepsilon.
			\end{equation*}
		\end{enumerate}
		\end{lemma}

		\begin{proof}
        By \eqref{scales52}, we may find $(\overline{l}_m)_{m \in \mathbb{N}}, (\overline{r}_m)_{m \in \mathbb{N}},  (\overline{\xi}_m)_{m \in \mathbb{N}}$ satisfying \eqref{scaleslemma}. Denote by $n' = \floor{\varepsilon n}$.
        For $n'$ large enough, let $\{\overline{C}_i\}_{i \in I_{n'}}$ be good boxes from \cref{goodbox} for $(\overline{l}_{n'},\overline{r}_{n'},\overline{\xi}_{n'},\varepsilon^{2000})$.
        Let $H_{n'} = \mathcal{P}_{\tau}^c \cup f^{-n'}\mathcal{P}_{\tau}^c \cup f^{-n'}\left(\underset{i \in I_{n'}}{\bigcup}\tilde{T}_i\right)$.
        It satisfies $\mu(H_n) < \varepsilon^{1000}$ for $\tau>0$ small enough and $n$ large enough.
        By \cref{MainLemma}, for $n' = \floor{\varepsilon n}$ large enough, there exists $\tilde{\mathcal{K}}_{n'} \subset M$,
        a measurable u-subordinate partition
        $\{\tilde{W}_{n'}(x)\}_{x \in \tilde{T}_{n'}}$
        of $\tilde{\mathcal{K}}_{n'}$ satisfying (U1) to (U3) of \cref{MainLemma}
        and $I^*_{n'}(x) \subset I_{n'}$ satisfying (E1) of \cref{MainLemma}.
        Suppose we have chosen $\overline{l},\overline{r}$ and $\overline{\xi}$ also satisfying
        \begin{equation*}
            e^{-\beta m}
            \ll \overline{r}_{\floor{\varepsilon m}}^{4d}\overline{l}_{\floor{\varepsilon m}}^2
            \ll (\overline{\xi}_{\floor{\varepsilon m}})^{1+\alpha}
            \ll \overline{l}_{\floor{\varepsilon m}} \ \text{and} \
            e^{2\sqrt{\delta}m}\overline{\xi}_{\floor{\varepsilon m}} 
            \ll e^{-\delta m},
        \end{equation*}
        which is possible since $\delta$ is much smaller than $\varepsilon$.
        Then, we can find sequences of positive numbers $(\overline{\overline{l}}_m)_{m \in \mathbb{N}},
        (\overline{\overline{r}}_m)_{m \in \mathbb{N}}$ and $(\overline{\overline{\xi}}_m)_{m \in \mathbb{N}}$ satisfying
        \begin{equation*}
            e^{2\sqrt{\delta}m}\overline{\xi}_{\floor{\varepsilon m}}
        \ll (\overline{\overline{\xi}}_m)^{1+ \alpha}
        \ll \overline{\overline{l}}_m
        \ll \overline{\overline{r}}_m
        \ll \overline{\overline{\xi}}_m
        \ll e^{- \delta m}.
        \end{equation*}
        By \cref{fakecsconstruction}, for $n$ large enough and decreasing $I_{n'}$ if necessary,
        we may construct a $(\varepsilon,n)$-fake cs foliation on the boxes $\overline{C}_i$ for $i \in I_{n'}$.
			Let $\mathcal{F}^{cs,n}_i$ denote such a fake cs foliation on $\overline{C}_i$.
            Denote also $\mathcal{K}_n = \tilde{\mathcal{K}}_{n'}$, $T_n = \tilde{T}_{n'}$ and for $x \in T_n$, $W_n(x) = \tilde{W}_{n'}(x)$.
            Properties (U1) and (U2) follow from \cref{MainLemma}.
			For $x \in T_n$ and $i \in I^*_{n'}(x)$,
			decompose the following into its connected components
			\begin{equation*}
				W(x) \cap f^{-n'}\overline{C}_i
				= \underset{l \in L_x(i)}{\bigcup} G_l.
			\end{equation*}
			Define the following families of good and bad components
			\begin{align*}
				L^G_x(i) &= \{ l \in L_x(i) \ | \  G_l \cap H^c \neq \emptyset\} \ \text{and}\\
				L^B_x(i) &= \{ l \in L_x(i) \ | \ G_l \subset H \}.
			\end{align*}
			By (U3) of \cref{MainLemma},
			\begin{equation*}
				\mu_{W(x)}\left(\underset{i \in I^*_{n'}(x)}{\bigcup}
					\underset{l \in L^B_x(i)}{\bigcup} G_l\right)
					\leq \mu_{W(x)}(H) < \varepsilon^2.
			\end{equation*}
			By (E1) of \cref{MainLemma},
			\begin{equation*}
				\mu_{W(x)}\left(\underset{i \in I^*_{n'}(x)}{\bigcup}
					f^{-n'}\overline{C}_i\right) > \frac{1- \varepsilon}{1+ \varepsilon}.
			\end{equation*}
			Denote by 
			\begin{equation*}
				\label{goodls}
				I_n(x) = \left\{
					i \in I^*_{n'}(x) \ | \
					(\mu_{W(x)})_{f^{-n'}\overline{C}_i}
					\left(
					\underset{l \in L^B_x(i)}{\bigcup} G_l\right)
				< \varepsilon \right\}.
			\end{equation*}
			By \cref{AppendixLemma} and the two previous estimates,
			\begin{equation}
				\label{mostlgood}
				\frac{1}{1+4\varepsilon}
				\leq
				\frac{1- \varepsilon}{1+ \varepsilon} - \varepsilon
				\leq \mu_{W(x)}\left(\underset{i \in I_n(x)}{\bigcup}
					f^{-n'}\overline{C}_i\right).
			\end{equation}
            By (E1) of \cref{MainLemma}, 
            \begin{equation*}
                \sum_{i \in I_n} \mu(\overline{C}_i) \sim_{1+ \varepsilon}
                \sum_{i \in I_n} \mu_{W(x)}(f^{-n'}\overline{C}_i).
            \end{equation*}
            Therefore, by \eqref{mostlgood},
            \begin{equation}
            \label{Inxgood}
                \sum_{i \in I_n(x)}\mu(\overline{C}_i) \sim_{(1+4\varepsilon)^2} 1.
            \end{equation}
			Note that if $i \in I_n(x)$, then
			\begin{equation*}
				\label{mostlgood2}
				\mu_{W(x)}\left(\underset{l \in L^B_x(i)}{\bigcup}G_l\right)
				< \varepsilon \mu_{W(x)}(f^{-n'}\overline{C}_i).
			\end{equation*}
			This implies that
			\begin{equation}
				\label{mostlgood3}
				\mu_{W(x)}\left(\underset{l \in L^G_x(i)}{\bigcup}G_l\right)
                \sim_{(1+\varepsilon)^2} \mu_{W(x)}(f^{-n'}\overline{C}_i).
			\end{equation}
			Now fix $i \in I_n(x)$.
            For each $l \in L^G_x(i)$, take $w_l \in H_{n'}^c \cap G_l$.
			Define
			\begin{equation*}
				L^{\partial}_x(i) = \{ l \in L^G_x(i) \ | \ G_l \cap \partial W(x) \neq \emptyset\}.
			\end{equation*}
			If $l \in L^{\partial}_x(i)$, then take $x_l \in G_l\cap \partial W(x)$.
			Since $f^{n'}w_l \in \mathcal{P}_{\tau}$,
			given any other $y \in G_l$, we have 
			\begin{equation*}
				f^{n'}y,f^{n'}x_l \in W^u(f^{n'}w_l).
			\end{equation*}
            By \cref{unstable}, this implies that
			\begin{equation*}
				d(y,x_l) \leq \tau^{-1} e^{(- \lambda+\sqrt{\delta})n'}.
			\end{equation*}
			Therefore,
			\begin{equation*}
				\underset{l \in L^{\partial}_x(i)}{\bigcup}G_l \subset \partial_{\tau^{-1}e^{(-\lambda + \sqrt{\delta})n'}} W(x).
			\end{equation*}
            Since $\mu$ is an SRB measure, then
			\begin{equation*}
		          \mu_{W(x)}\left(\underset{l \in L^{\partial}_x(i)}{\bigcup}G_l\right) = O(e^{(-\lambda + \sqrt{\delta})n'}(\overline{\xi}_{n'})^{-u}).
			\end{equation*}
			By \eqref{mostlgood3},
			\begin{equation*}
				\frac{1}{(1+\varepsilon)^2}
				< \frac{\mu_{W(x)}\left(\underset{l \in L^G_x(i)\setminus L^{\partial}_x(i)}{\bigcup}G_l\right)}{\mu(\overline{C}_i)} 
				+ \frac{\mu_{W(x)}\left(\underset{l \in L^{\partial}_x(i)}{\bigcup}G_l\right)}{\mu(\overline{C}_i)} 
				< (1+\varepsilon)^2.
			\end{equation*}
			By (P2) of \cref{goodbox}, $D\overline{\xi}^{2d}_{n'} \leq \mu(\overline{C}_i)$. Therefore,
			\begin{equation*}
				\frac{\mu_{W(x)}\left(\underset{l \in L^{\partial}_x(i)}{\bigcup}G_l\right)}{\mu(\overline{C}_i)} 
				\leq \frac{1}{D\overline{\xi}_{n'}^{2d}}O(e^{(-\lambda + \sqrt{\delta})n'}(\overline{\xi}_{n'})^{-u}) \rightarrow 0.
			\end{equation*}
			Therefore, taking $n$ larger if necessary
			and denoting by $L^*_x(i) = L^G_x(i) \setminus L^{\partial}_x(i)$, we may assume that
			\begin{equation}
				\label{fullycross}
				\mu_{W(x)}\left(\underset{l \in L^*_x(i)}{\bigcup}G_l\right) \sim_{(1+\varepsilon)^3} \mu(\overline{C}_i).
			\end{equation}
			Take $\overline{x}_i \in \overline{C}_i \cap \mathcal{P}_{\tau}$.
			Recall that $W^u(\overline{x}_i) \cap \mathcal{F}_i^{cs,n}$ 
			is an open set in $W^u(\overline{x}_i)$ such that
			\begin{equation*}
				m^u_{\overline{x}_i}(\mathcal{F}_i^{cs,n})
				\sim_{1 + \varepsilon} m^u_{\overline{x}_i}(\overline{C}_i).
			\end{equation*}
			Therefore, we can find $k$ pairwise disjoint subsets of $ W^u(\overline{x}) \cap \mathcal{F}_i^{cs,n}$ denoted by
			$\{\hat{B}_{i,j}\}_{j \in J_n(i)}$ 
			with diameter at most $ \varepsilon^2\underset{x \in M}{\sup} \|D_xf\|^{-n}$
			and all of the same $m^u_{\overline{x}_i}$ measure such that
			\begin{equation}
				\label{Bsati}
				m^u_{\overline{x}_i}\left(\underset{j \in J_n(i)}{\bigcup}
					\hat{B}_{i,j}\right)
					\sim_{1 + \varepsilon}
					m^u_{\overline{x}_i}(\mathcal{F}^{cs,n}_i)
					\sim_{1+ \varepsilon}
					m^u_{\overline{x}_i}(\overline{C}_i).
			\end{equation}
			Define
			\begin{equation*}
				B_{i,j} = \underset{y \in \hat{B}_{i,j}}{\bigcup} \mathcal{F}^{cs,n}_i(y) \cap \tilde{T}_i.
			\end{equation*}
			These will be the desired collection of pairwise disjoint sets for the lemma.
			In order to verify (B2),
			let $z_1,z_2 \in B_{i,j}$ such that $z_v \in \mathcal{F}^{cs,n}_i(y_v)$,
			for $v =1,2$.
            Then using (F3) of \cref{fakecsconstruction} and the hypothesis on the diameter of the sets $\hat{B}_{i,j}$, for $0 \leq k \leq n$,
			\begin{align*}
				d(f^kz_1,f^kz_2) 
				& \leq d(f^kz_1,f^ky_1) + d(f^ky_1,f^ky_2) + d(f^ky_2,f^kz_2) \\
				& \leq 2\tau^{-2}\overline{r}_ne^{2 \sqrt{\delta}k} + \varepsilon^2 \underset{x \in M}{\sup} \|D_xf\|^{ k-n} \\
				& \leq 2\tau^{-2} \overline{r}_n e^{2 \sqrt{\delta}n} + \varepsilon^2 < \varepsilon,
			\end{align*}
			for $n$ large enough since $\overline{r}_n \ll e^{-2\sqrt{\delta} n}$.
			To begin verifying (B1), we have
			\begin{align*}
				\mu\left(\underset{j \in J_n(i)}{\bigcup} B_{i,j}\right)
				& \sim_{1+ \varepsilon} \mu(\tilde{T}_i)
				\mu_{\overline{W}(\overline{x}_i)}
				\left(\underset{j \in J_n(i)}{\bigcup}\hat{B}_{i,j}\right) 
				&  \text{by (P8) of \cref{goodbox}} \\
				& \sim_{1 + \varepsilon} 
				\mu(\tilde{T}_i) \frac{1}
				{m^u_{\overline{x}_i}(\overline{W}(\overline{x}_i))}
				m^u_{\overline{x}_i}\left( 
					\underset{j \in J_n(i)}{\bigcup}\hat{B}_{i,j}\right)
				&  \text{by \cref{SRBdensityHol}} \\
				& \sim_{(1 + \varepsilon)^2} 
				\mu(\tilde{T}_i)
				& \text{by \eqref{Bsati}} \\
				& \sim_{1 +  \varepsilon} \mu(\overline{C}_i)
				& \text{by (P7) of \cref{goodbox}}.
			\end{align*}
			Therefore,
			\begin{equation}
				\label{BsCi}
				\mu\left(\underset{j \in J_n(i)}{\bigcup} B_{i,j}\right)
				\sim_{(1+\varepsilon)^5}
				\mu(\overline{C}_i).
			\end{equation}
			Summing over $i \in I_n(x)$, by \eqref{Inxgood} we get
			\begin{equation}
				\label{Bbig}
				\mu\left(\underset{i \in I_n'(x)}{\bigcup}
					\underset{j \in J_n(i)}{\bigcup} B_{i,j}\right)
				\sim_{(1+ \varepsilon)^5}
				\sum_{i \in I_n(x)} \mu(\overline{C}_i)
				\sim_{1+ \varepsilon} 1.
			\end{equation}
			We will now verify the equidistribution condition of $W_n(x)$, that is, the second part of (B1).
			For $x \in T_n$ and $i \in I_n(x)$, let
			\begin{equation*}
				B^{x}_{i,j} = B_{i,j} \setminus \underset{l \in L^{\partial}_x(i)}{\bigcup}W^u(f^{n'}w_l).
			\end{equation*}
            Notice that since unstable manifolds have $\mu$ measure $0$, then $\mu(B_{i,j})=\mu(B_{i,j}^x)$.
            Note also that for $x \in T_n$,
            \begin{equation*}
            	W(x)\cap f^{-n'}B_{i,j}^x = 
            	\bigcup_{l \in L^*_x(i)}G_l \cap f^{-n'}B_{i,j}^x 
            	\bigcup_{l \in L^{\partial}_x(i)}G_l \cap f^{-n'}B_{i,j}^x
            	\bigcup_{l \in L^B_x(i)}G_l \cap f^{-n'}B_{i,j}^x.
            \end{equation*}
            But, we removed the components intersecting $G_l$ for $l \in L^{\partial}_x(i)$ and for $l \in L^B_x(i)$, $f^{n'}G_l \cap \tilde{T}_i = \emptyset$. Therefore,
            \begin{equation}
            	\label{ignoreL}
            	W(x)\cap f^{-n'}B_{i,j}^x =
            	\bigcup_{l \in L^*_x(i)}G_l \cap f^{-n'}B_{i,j}^x.
            \end{equation}
			For $j \in J_n(i)$, $B_{i,j}^x$ is, up to union of finitely many unstable manifolds, the intersection of a $\mathcal{F}_i^{cs,n}$-saturated set with $\tilde{T}_i$.
            Therefore, by (P8) of $\cref{goodbox}$
			\begin{equation*}
				\mu(B^x_{i,j}) 
				\sim_{1+ \varepsilon}
				\mu(\tilde{T}_i) \mu_{\overline{W}(\overline{x}_i)}
				(\hat{B}_{i,j}).
			\end{equation*}
			In partiular, for $j,j' \in J_n(i)$,
			\begin{equation}
				\label{Bsaresame}
				\mu(B^x_{i,j})
				\sim_{(1+ \varepsilon)^4} \mu(B^x_{i,j'}).
			\end{equation}
			If $l \in L_x^*(i)$, then $W^u(f^{n'}w_l)$ is a $(\alpha,1)$-mostly horizontal submanifold.
            Therefore, by (F2) of \cref{fakecsconstruction} and \eqref{Bsati}
			\begin{equation*}
				m^u_{f^{n'}w_l}\left(\underset{j \in J_n(i)}{\bigcup}B^x_{i,j}\right)
				\sim_{(1+ \varepsilon)^4}
				m^u_{f^{n'}w_l}(\overline{C}_i).
			\end{equation*}
			Increasing $n$ if necessary, by \cref{JacHolderCor} and using that $w_l \in \mathcal{P}_{\tau}$,we may assume that for $ y \in \overline{W}(f^{n'}w_l)$
			\begin{equation}
			\label{JacClose}
				\text{Jac}(D_y(f^{-n'}|_{E^u(y)}))
				\sim_{1+\varepsilon} \text{Jac}(D_{f^{n'}w_l}(f^{-n'}|_{E^u(f^{n'}w_l)})).
			\end{equation}
			Therefore,
			\begin{align*}
				m^u_x\left(G_l\cap f^{-n'}
				\left(\underset{j \in J_n(i)}{\bigcup}B^x_{i,j}\right)
				\right) 
				& = \int_{\overline{W}(f^{n'}w_l)\cap (\cup_{j \in J(i)} B^x_{i,j})}
				\text{Jac}(D_y(f^{-n'}|_{E^u(y)})) dm^u_{f^{n'}w_l}(y) \\
				& \sim_{1+ \varepsilon}
				\text{Jac}(D_{f^{n'}w_l}(f^{-n'}|_{E^u(f^{n'}w_l)})) m^u_{f^{n'}w_l}
				\left( \underset{j \in J_n(i)}{\bigcup}B^x_{i,j}\right)\\
				& \sim_{(1+ \varepsilon)^4}
				\text{Jac}(D_{f^{n'}w_l}(f^{-n'}|_{E^u(f^{n'}w_l)})) m^u_{f^{n'}w_l}
				(\overline{C}_i)  \\
				& = \int_{\overline{W}(f^{n'}w_l)} \text{Jac}(D_{f^{n'}w_l}(f^{-n'}|_{E^u(f^{n'}w_l)})) dm^u_{f^{n'}w_l}(y) \\
				& \sim_{1+\varepsilon} \int_{\overline{W}(f^{n'}w_l)} \text{Jac}(D_y(f^{-n'}|_{E^u(y)})) dm^u_{f^{n'}w_l}(y) \\
				& =
				m^u_x(f^{-n'}\overline{W}(f^{n'}w_l)) = m^u_x(G_l).
			\end{align*}
			Therefore,
			dividing the above by $m^u_x(W(x))$ and using \cref{SRBdensityHol}, we get that for $l \in L^*_x(i)$,
			\begin{equation*}
				\mu_{W(x)}\left(G_l\cap f^{-n'}
				\left(\underset{j \in J_n(i)}{\bigcup}B^x_{i,j}\right)
				\right) 
				\sim_{(1+ \varepsilon)^8} \mu_{W(x)}(G_l).
			\end{equation*}
			In particular, summing the above over $l \in L^*_x(i)$, using \eqref{ignoreL} and \eqref{fullycross},
			\begin{equation}
				\label{mostGl}
				\mu_{W(x)}\left(\bigcup_{j \in J_n(i)}f^{-n'}B^x_{i,j}\right)
				\sim_{(1+\varepsilon)^{11}}
				\mu(\overline{C}_i)
			\end{equation}
			For $l \in L^*_x(i)$, $j, j' \in J_n(i)$,
			\begin{align*}
				m^u_x(G_l \cap f^{-n'}B^x_{i,j}) & = \int_{B^x_{i,j}} \text{Jac}(D_y(f^{-n'}|_{E^u(y)})) dm^u_{f^{n'}w_l}(y),
								& \text{since} \ f^{n'}G_l = \overline{W}(f^{n'}w_l) \\
								& \sim_{1+ \varepsilon} \text{Jac}(D_{f^{n'}w_l}(f^{-n'}|_{E^u(y)})) m^u_{f^{n'}w_l}(B^x_{i,j}),
							 	& \text{by \eqref{JacClose}} \\
								& \sim_{1 + \varepsilon} \text{Jac}(D_{f^{n'}w_l}(f^{-n'}|_{E^u(y)}))  m^u_{\overline{x}_i}(\hat{B}_{i,j}),
								& \text{by (F2) of \cref{fakecsconstruction}} \\
								& = \text{Jac}(D_{f^{n'}w_l}(f^{-n'}|_{E^u(y)}))  m^u_{\overline{x}_i}(\hat{B}_{i,j'}),
								& \text{by construction of} \  \hat{B}_{i,j} \\
								& \sim_{1+ \varepsilon} \text{Jac}(D_{f^{n'}w_l}(f^{-n'}|_{E^u(y)})) m^u_{f^{n'}w_l}(B^x_{i,j'}),
							 	& \text{by (F2) of \cref{fakecsconstruction}}  \\
								& \sim_{1 + \varepsilon} \int_{B^x_{i,j'}} \text{Jac}(D_y(f^{-n'}|_{E^u(y)})) dm^u_{f^{n'}w_l}(y),
								& \text{by \eqref{JacClose}} \\
								& = m^u_x(G_l \cap f^{-n'}B^x_{i,j'}),
								& \text{since} \ f^{n'}G_l = \overline{W}(f^{n'}w_l).
			\end{align*}
			Therefore,
			\begin{align*}
				\mu_{W(x)}(f^{-n'}B^x_{i,j}) & \sim_{1 + \varepsilon} \frac{1}{m^u_x(W(x))} m^u_x(f^{-n'}B^x_{i,j}), & \text{by \cref{SRBdensityHol}} \\
							  & =  \frac{1}{m^u_x(W(x))} \sum_{l \in L_x^*(i)} m^u_x(G_l \cap f^{-n'}B^x_{i,j}),
							  & \text{by \eqref{ignoreL}} \\
							  & \sim_{(1 + \varepsilon)^4} \frac{1}{m^u_x(W(x))} \sum_{l \in L^*_x(i)} m^u_x(G_l \cap f^{-n'}B^x_{i,j'}),
							  & \text{by above estimates} \\
							  & = \frac{1}{m^u_x(W(x))} m^u_x(f^{-n'}B^x_{i,j'}),
							  & \text{by \eqref{ignoreL}} \\
							  & \sim_{1+ \varepsilon} \mu_{W(x)}(f^{-n'}B^x_{i,j'}),
							  & \text{by \cref{SRBdensityHol}}.
			\end{align*}
			This means that
			\begin{align*}
				k \mu_{W(x)}(f^{-n'}B^x_{i,j_0}) & \sim_{(1+ \varepsilon)^6} \sum_{j \in J(i)} \mu_{W(x)}(f^{-n'}B^x_{i,j}),
								& \text{by above estimates} \\
								& =  \mu_{W(x)}\left(\underset{j \in J(i)}{\bigcup}f^{-n'}B^x_{i,j} \right) \\ 
							       & \sim_{(1+\varepsilon)^{11}} \mu(\overline{C}_i),
							       & \text{by \eqref{mostGl}}\\
							       & \sim_{(1+ \varepsilon)^5} \sum_j \mu(B^x_{i,j}),
							       & \text{by \eqref{BsCi}}\\
							       & \sim_{(1+ \varepsilon)^4} k \mu(B^x_{i,j_0}),
							       & \text{by \eqref{Bsaresame}}.
			\end{align*}
			Therefore,
			\begin{equation}
				\label{Bjxequi}
				\mu_{W(x)}(f^{-n'}B^x_{i,j_0}) \sim_{(1+\varepsilon)^{26}} \mu(B_{i,j_0}).
			\end{equation}
			We want the above to hold for $B_{i,j}$ instead of $B_{i,j}^x$.
			In order to simplify notation, rename the indices to not include $i$ anymore.
			For $j\in J_n(x)$, let
			\begin{equation*}
				f^{-n'}B_j = f^{-n'}B_j^x \cup \Delta_j.
			\end{equation*}
			We may write
			\begin{equation*}
				\mu_{W(x)}(f^{-n'}B_j) = \mu_{W(x)}(f^{-n'}B_j^x) + \delta_j,
			\end{equation*}
			where $\mu_{W(x)}(\Delta_j) = \delta_j$. Then by \eqref{Bbig} and \eqref{Bjxequi}
			\begin{align*}
				\frac{1}{(1+ \varepsilon)^{32}} + \sum_{j \in J_n(x)} \delta_j
				& \leq \frac{1}{(1+ \varepsilon)^{26}}\sum_{j \in J_n(x)} \mu(B_j) + \sum_{j \in J_n(x)} \delta_j \\
				& \leq \sum_{j \in J_n(x)} \mu_{W(x)}(f^{-n'}B_j^x) + \sum_{j \in J_n(x)} \delta_j \leq 1.
			\end{align*}
			Denote by 
			\begin{equation*}
				\Delta = \underset{j \in J_n(x)}{\bigcup}\Delta_j.
			\end{equation*}
			Then,
			\begin{equation*}
				\mu_{W(x)}(\Delta) \leq 1 - \frac{1}{(1+ \varepsilon)^{32}} \ \text{and}
				\ \mu_{W(x)}\left(\underset{j \in J_n(x)}{\bigcup}f^{-n'}B_j\right) > \frac{1}{(1+ \varepsilon)^{26}}.
			\end{equation*}
			By \cref{AppendixLemma},
			denoting by
			\begin{equation*}
				\label{Jprime}
				J'_n(x) = \{ j \in J_n(x) \ | \ \mu_{W(x)}(\Delta_j) < \sqrt{ \varepsilon} \mu_{W(x)}(f^{-n'}B_j) \},
			\end{equation*}
			we have
			\begin{equation}
				\label{Jprimebig}
					 \frac{1}{(1+ \varepsilon)^{26}} - \frac{(1+ \varepsilon)^{32}-1}{\sqrt{ \varepsilon}(1+ \varepsilon)^{32}}
				\leq \mu_{W(x)}\left(\underset{j \in J'_n(x)}{\bigcup}f^{-n'}B_j\right).
			\end{equation}
			Notice that the left hand side of the above goes to $1$ as $ \varepsilon$ goes to $0$.
			Now the condition for $j \in J_n'(x)$ allows us translate equidistribution on $B_j^x$ to $B_j$. Indeed,
			\begin{equation*}
				\mu_{W(x)}(f^{-n'}B_j^x)
				\leq \mu_{W(x)}(f^{-n'}B_j) 
				\leq (1+\sqrt{ \varepsilon})\mu_{W(x)}(f^{-n'}B_j^x).
			\end{equation*}
			Therefore,
			\begin{equation}
				\label{BjsameBjx}
				\mu_{W(x)}(f^{-n'}B_j) \sim_{1+ \sqrt{ \varepsilon}} \mu_{W(x)}(f^{-n'}(B_j^x)).
			\end{equation}
			Combining \eqref{BjsameBjx} with \eqref{Bjxequi}, we get
			\begin{equation}
				\label{Bequi}
				\mu_{W(x)}(f^{-n'}B_j) \sim_{(1+ \sqrt{\varepsilon})^{27}} \mu(B_j).
			\end{equation}
			And combining \eqref{Bequi} with \eqref{Jprimebig},
			\begin{equation*}
				\frac{1}{(1+ \sqrt{ \varepsilon})^{27}} \left(\frac{1}{(1+ \varepsilon)^{26}} - \frac{(1+ \varepsilon)^{32}-1}{\sqrt{ \varepsilon}(1+ \varepsilon)^{32}}\right)
				\leq \mu\left(\underset{j \in J'_n(x)}{\bigcup} B_j\right).
			\end{equation*}
			These show (B1) and concludes the proof of the Lemma.
				\end{proof}

	\section{From equidistribution of unstables to Bernoulli}
	\label{sec:Exponential Equidistribution implies Bernoulli}
	This section is dedicated to proving the following Proposition.

	\begin{proposition}
		\label{EBB}
		Let $M$ be a compact Riemannian manifold,
		$f \colon M \to M$ be a $C^{1+ \alpha}$ diffeomorphism 
		and $\mu$ be an $f$-invariant, nonatomic, Borel, ergodic probability measure
		with at least one positive Lyapunov exponent.
		Suppose that
		for every $ \varepsilon >0$,
		there exists a sequence $\xi_n$ of positive numbers
		converging to $0$
		and $n_0 \in \mathbb{N}$ such that 
		for all $n > n_0$, there exists $K_n \subset M$,
		a measurable partition $\{W_n(x)\}_{x \in T_n}$ of $K_n$,
		a disintegration $\{\mu_{W(x)}\}_{x \in T_n}$ of the measure $\mu$ subordinate to the partition $\{W_n(x)\}_{x \in T_n}$
		and a finite family of pairwise disjoint subsets $\{B_i\}_{i \in I_n}$ of $M$ satisfying:
		\begin{enumerate}
			\item[(U1)] $\mu(K_n) > 1- \varepsilon$;
			\item[(U2)] $W_n(x)$ is a piece of unstable manifold 
				of size $\sim_{1 + \varepsilon}\xi_n$, for each $x \in T_n$;
			\item[(B1)]  for $x \in T_n$,
				there exists $I_n(x) \subset I_n$ such that
				\begin{equation*}
					\mu\left(\underset{i \in I_n(x)}{\bigcup}B_i\right) 
					> 1 - \varepsilon
				\end{equation*}
			and for $i \in I_n(x)$,
			\begin{equation}
				\label{EE}
				\mu_{W_n(x)}(f^{-\floor{\varepsilon n}}B_i)
				\sim_{1+\varepsilon} \mu(B_i).
			\end{equation}

			\item[(B2)]  for $i \in I_n$, $x,y \in B_i$ and $0 \leq k \leq n$,
				\begin{equation*}
				d(f^kx,f^ky) < \varepsilon;
			\end{equation*}
		\end{enumerate}
		Then $(f,\mu)$ is Bernoulli.
		
	\end{proposition}

	\subsection{Almost u-saturation}

	Given $\ep,\xi>0$ we say that a partition $\mathcal{Q}$ is \emph{$(\ep,\xi)$-u-saturated} (see section 4 of \cite{ExpMix}) if there exists $E \subset M$ with $\mu(E)>1- \varepsilon$ such that
	\begin{equation*}
		x \in E \implies \mathcal{W}^u_{\xi}(x) \subset \mathcal{Q}(x).
	\end{equation*}
	The following is a version of Lemma 4.2 of \cite{ExpMix} for any Borel measure. 
	The proof is exactly the same and we use the D-Nice condition to verify that $\sum_k \mu(\partial_{c^{-k}} \mathcal{Q}) \leq D \sum_k c^{-k} < \infty$.
	\begin{lemma}
		\label{usat}
		Let $\mathcal{Q}$ be a $D$-Nice partition. Given $\ep>0$, there exists $\xi = \xi( \varepsilon) >0$ and $n_4 = n_4( \varepsilon,\xi)$ such that for every $n_4 \leq N_1 \leq N_2$, the partition $\bigvee_{i=N_1}^{N_2} f^i(\mathcal{Q})$ is $( \varepsilon,\xi)$ u-saturated.
	\end{lemma}	
	
	The following corollary is crucial for proving Bernoullicity from equidistribution of unstables.

\begin{corollary}
	\label{unstdisint}
    Let $\varepsilon>0$, 
    $\mathcal{Q}$ be a $D$-Nice partition of $M$,
    $K_n$ be a sequence of subsets of $M$ and
    $\{W_n(x)\}_{x \in T_n}$ be a sequence of u-subordinate partitions 
    of $K_n \subset M$ satisfying
    \begin{equation*}
        \mu(K_n) > 1- \varepsilon^2 \ \text{and} 
        \ |W_n(x)| \sim_{1 + \varepsilon} \xi_n, \ \forall x \in T_n
    \end{equation*}
    with $\xi_n$ converging to $0$.
    Then, there exists $N_0,n_0>0$ such that
    for $N_0 \leq N_1 \leq N_2$,
    $n> n_0$ and 
    $\varepsilon$-almost every atom 
    $A \in \bigvee_{i=N_1}^{N_2}f^i(\mathcal{Q})$,
    there exists $T_A = T_A(n) \subset T_n$ such that
    for all $B \subset M$,
    \begin{equation}
        \label{atomdisint}
        \mu_A(B) = \int_{T_A} \mu_{W_n(x)}(B) d\nu_{T_A}(x) + \mu^A_{\text{error}}(B),
    \end{equation}
    where $\mu^A_{\text{error}}$ is a signed measure on $M$ with
    \begin{equation*}
       \underset{B \subset M}{\sup} |\mu^A_{\text{error}}(B)| \leq 5\varepsilon.
    \end{equation*}
    Moreover, if $B \subset \underset{x \in T_A}{\bigcup} W_n(x)$, then
    \begin{equation*}
	    |\mu^A_{\text{error}}(B)| \leq 2 \varepsilon \int_{T_A} \mu_{W_n(x)}(B)d\nu_{T_A}(x).
    \end{equation*}
\end{corollary}

\begin{proof}
    By \cref{usat},
    there exists $\xi >0$ and $N_0>0$ such that
    if $N_0 \leq N_1 \leq N_2$, we may separate the atoms of $\bigvee_{i=N_1}^{N_2} f^i(\mathcal{Q})$ into two disjoint collections
    \begin{equation*}
        \bigvee_{i=N_1}^{N_2} f^i(\mathcal{Q}) = G \sqcup B
    \end{equation*}
    such that  
    for $A \in G$, there exists $E_A \subset A$ satisfying
    \begin{equation*}
        \mu_A(E_A^c) < \varepsilon \ \text{and} \ 
        x \in E_A \implies W^u_{\xi}(x) \subset A
    \end{equation*}
    while
    \begin{equation*}
        \mu\left(\bigcup_{A \in B}A\right) < \varepsilon.
    \end{equation*}
    Let $n_0$ be such that
    \begin{equation}
        \label{unstsmall}
        n > n_0 \implies 2(1 + \varepsilon)\xi_n < \xi.
    \end{equation}
    Fix $n>n_0$, which we'll omit from the notation whenever possible.
    Notice that
    \begin{equation*}
        \mu\left(\bigcup_{A \in G} E_A\right)
        > (1- \varepsilon)\sum_{A \in G} \mu(A)
        >(1- \varepsilon)^2 > 1-3 \varepsilon.
    \end{equation*}
    Let $G' = \{ A \in G \ | \ \mu_{E_A}(K^c) \leq \varepsilon\}$
    then by \cref{AppendixLemma},
    \begin{equation*}
       1- 4 \varepsilon 
       < 1- 3\varepsilon -\frac{\varepsilon^2}{\varepsilon}
       < \mu\left(\bigcup_{A \in G'}E_A\right)
    \end{equation*}
    which means that $4 \varepsilon$-almost every atom is in $G'$.
    For $A \in G'$,
    define 
    \begin{equation*}
        T_A = \{ x \in T \ | \ W(x) \cap E_A \neq \emptyset\}.
    \end{equation*}
    Given $B \subset M$, we may write
    \begin{equation}
    \begin{split}
        \mu(A\cap B) & = \nu(T_A)
        \int_{T_A}\mu_{W(x)}(A \cap B)d\nu_{T_A}(x)  \\
        & +\left(\int_{T\setminus T_A}\mu_{W(x)}(A\cap B)d\nu(x) 
        + \mu(A \cap B \cap K^c)\right).
    \end{split}
\label{TAdisint}
\end{equation}
    By \eqref{unstsmall} and the property of $E_A$,
    $x \in T_A $ implies $ W(x) \subset A$. In particular,
    \begin{equation*}
	    \mu_{W(x)}(A \cap B) = \mu_{W(x)}(B), \ \text{for} \ x \in T_A.
    \end{equation*}
    Since $\mu_{W(x)}(A) = 1$ for $\nu$-almost every $x \in T_A$,
    \begin{equation*}
        \nu(T_A) = \int_{T_A}\mu_{W(x)}(A) d\nu(x) \leq \mu(A).
    \end{equation*}
    On the other hand, if $A \in G'$, then
    \begin{equation*}
        (1- \varepsilon)^2 \mu(A)
        \leq (1- \varepsilon)\mu(E_A)
        \leq \int_{T_A} \mu_{W(x)}(E_A)d\nu(x)
        \leq \nu(T_A).
    \end{equation*}
    Combining the two above estimates, we get
    \begin{equation}
    \label{numuA}
        \frac{\nu(T_A)}{\mu(A)} = 1 + \varepsilon_A, \ \text{where} \ 
        |\varepsilon_A| < 2\varepsilon.
    \end{equation}
    Dividing \eqref{TAdisint} by $\mu(A)$ and using \eqref{numuA},
    \begin{equation}
    \begin{split}
        \mu_A(B) & = \int_{T_A} \mu_{W(x)}(B) d\nu_{T_A}(x) 
                 + \varepsilon_A\int_{T_A}\mu_{W(x)}(B)d\nu_{T_A}(x)\\
                &+\frac{1}{\mu(A)} \int_{T\setminus T_A} \mu_{W(x)}(A \cap B) d\nu(x) 
                 +\frac{1}{\mu(A)} \mu(A\cap B \cap K^c).
    \end{split}
    \label{errorAdecomp}
    \end{equation}
    Denote by
    \begin{equation*}
        \mu^A_{\text{error}}(B) = 
        \mu_A(B) - \int_{T_A}\mu_{W(x)}(B) d\nu_{T_A}(x),
    \end{equation*}
    which is the last three terms on  \eqref{errorAdecomp}.
    For the first error term we have
    \begin{equation*}
        \underset{B \subset M}{\sup}
        \left| \varepsilon_A
        \int_{T_A} \mu_{W(x)}(B)d\nu_{T_A}(x) \right|
        \leq |\varepsilon_A| < 2 \varepsilon.
    \end{equation*}
    For the second, 
    since $x \in T\setminus T_A$ implies that  $W(x)\cap A \subset A\setminus E_A$, then
    \begin{equation*}
        \underset{B \subset M}{\sup}
        \left| \frac{1}{\mu(A)} \int_{T\setminus T_A} \mu_{W(x)}(A \cap B) d\nu(x) \right|
        \leq \frac{1}{\mu(A)} \mu(A\setminus E_A) < \varepsilon.
    \end{equation*}
    For the last term, we once again use that $A \in G'$:
    \begin{align*}
        \underset{B \subset M}{\sup} \left|\frac{1}{\mu(A)}
        \mu(A\cap B \cap K^c) \right|
        & \leq \frac{\mu(A \cap K^c)}{\mu(A)} \\
        &= \frac{\mu(E_A\cap K^c) + \mu(A\setminus E_A \cap K^c)}{\mu(A)} \\
        & \leq \varepsilon\frac{\mu(E_A)}{\mu(A)} + \varepsilon
        \leq 2 \varepsilon.
    \end{align*}
    Combining these 3 estimates together, we get
    \begin{equation*}
        \underset{B \subset M}{\sup}
        |\mu^A_{\text{error}}(B) | < 5 \varepsilon.
    \end{equation*}
    If $B \subset \underset{x \in T_A}{\bigcup}W_n(x)$, notice that when computing $\mu^A_{\text{error}}(B)$
    we just need to take the first error term into consideration, which proves the corollary.
\end{proof}

	\subsection{Matching between unstables}

	\begin{lemma}
		\label{matching}
		Suppose that the hypothesis of \cref{EBB} holds. 
		Given $ \varepsilon>0$, there exists $n_0 \in \mathbb{N}$ such that for $n>n_0$
		and $x,x' \in T_n$, there is a measure preserving map
		\begin{equation*}
			\theta_{x,x'} \colon (W_n(x),\mu_{W_n(x)}) \to (W_n(x'),\mu_{W_n(x')})
		\end{equation*}
		and a subset $\hat{W}_n(x) \subset W_n(x)$ with $\mu_{W_n(x)}(\hat{W}_n(x)) > 1 - \varepsilon$
		such that for $\floor{ \varepsilon n} \leq k \leq n$ and $z \in \hat{W}_n(x)$,
		\begin{equation*}
			d(f^kz,f^k\theta_{x,x'}(z)) < \varepsilon.
		\end{equation*}
	\end{lemma}
	\begin{proof}
		Denote by $n' = \floor{ \varepsilon n}$.
		For $x \in T_n$ and  $i \in I(x)$, using \eqref{EE} and the fact that $\mu_{W_n(x)}$ is nonatomic, there exists
		\begin{equation}
			\label{Bhati}
			\hat{B}_i(x) \subset W_n(x) \cap f^{-n'}B_i, \ \text{with} \ 
			\mu_{W_n(x)}(\hat{B}_i) = \frac{1}{1 + \varepsilon}\mu(B_i).
		\end{equation}
		Let $I_n(x,x') = I_n(x) \cap I_n(x')$. By (B3) in \cref{EBB},
		\begin{equation}
			\label{Ixxprime}
			\mu\left(\underset{i \in I_n(x,x')}{\bigcup} B_i\right) > 1- 4 \varepsilon.
		\end{equation}
		For $i \in I(x,x')$, since $\hat{B}_i(x)$ and $\hat{B}_i(x')$ are Lebesgue spaces of the same measure,
		there exists a measure preserving map $\hat{B}_i(x) \to \hat{B}_i(x')$, which defines a measure preserving map
		\begin{equation*}
			\theta_{x,x'} \colon 
			\left(\underset{i \in I(x,x')}{\bigcup} \hat{B}_i(x), \mu_{W_n(x)}\right) \to
			\left(\underset{i \in I(x,x')}{\bigcup} \hat{B}_i(x'), \mu_{W_n(x')}\right). 
		\end{equation*}
		By \eqref{Bhati} and \eqref{Ixxprime},
		\begin{equation*}
			\frac{1-4 \varepsilon}{1+ \varepsilon} < \mu_{W_n(x)}\left(\underset{i \in I(x,x')}{\bigcup} \hat{B}_i(x)\right).
		\end{equation*}
		Therefore, denoting by $\hat{W}_n(x) = \bigcup_{i \in I(x,x')} \hat{B}_i(x)$ we verify the second condition of the lemma.
		Since $\mu_{W_n(x)}$ and $\mu_{W_n(x')}$ are both probability measures,
		we may define $\theta_{x,x'}$ from $W_n(x)\setminus\hat{W}_n(x) \to W_n(x')\setminus\hat{W}_n(x)$
		to be measure preserving, thus proving the lemma.
	\end{proof}

	\subsection{Proof of \cref{EBB}}

	We just need to verify the conditions of \cref{vwBnice}.
	Let $\mathcal{P}$ be a D-nice partition, which is also regular. Given $ \varepsilon >0$, let $N_0>0, n_0$ be large 
	so that \cref{unstdisint} and \cref{EBB} are satisfied for $N_0<N_1<N_2$ and $n>n_0$.
	Let $A$ be an atom of $ \bigvee_{i=N_1}^{N_2} f^i(\mathcal{P})$ satisfying
	\begin{equation*}
		\mu_A(B) = \int_{T_A} \mu_{W_n(x)}(B) d\nu_{T_A} + \mu^A_{\text{error}}(B)
	\end{equation*}
	as in \cref{unstdisint}, which is satisfied by $ \varepsilon$-almost every atom.
	Since $(T_A,\nu_{T_A})$ and $(T,\nu_T)$ are Lebesgue spaces of the same measure,
	then we can define $\theta_A \colon (T_A,\nu_{T_A}) \to (T,\nu_{T})$ measure preserving.
	Let $K_A = \bigcup_{x \in T_A} W_n(x)$.
	Using \cref{matching}, define
	$\theta \colon K_A \to K_n$ by, for $z \in W_n(x)$ with $x \in T_A$,
	\begin{equation*}
		\theta(z) = \theta_{x,\theta_Ax}(z).
	\end{equation*}
	For $B \subset K_n$, using the proof of \cref{unstdisint}
	\begin{align*}
		\mu_A(\theta^{-1}B)  
		& =(1+ \varepsilon_A) \int_{T_A} \mu_{W_n(x)}(\theta_{x,\theta_Ax}^{-1}(B)) d\nu_{T_A}(x) \\
		& =(1+ \varepsilon_A) \int_{T_A} \mu_{W_n(\theta_Ax)}(B) d\nu_{T_A}(x) \\
		& =(1+ \varepsilon_A) \int_{T_n} \mu_{W_n(y)}(B) d\nu_T(y) \\
		& =(1+ \varepsilon_A) \mu(B).
	\end{align*}
	Therefore, $\theta$ on $K_A$ is $10 \varepsilon$-measure preserving, since $ |\varepsilon_A| < 2 \varepsilon$.
	On top of that, for $z \in K_A$ and $n' \leq k \leq n$,
	\begin{equation*}
		d(f^kz,f^k\theta z) < \varepsilon.
	\end{equation*}
	Since $\mu_A(A \setminus K_A) < \varepsilon$, we may define $\theta$ to be anything on that set.
	This verifies the conditions of \cref{vwBnice} and concludes the proof.

	\subsection{Proof of \cref{thmA} and \cref{thmB}}
	We are finally ready to prove the main theorems.
	Suppose that either that $\mu$ is an exponentially mixing SRB measure or volume is almost exponentially mixing and let $\mu$ be the limit SRB measure.
	Since in both cases $\mu$ is ergodic and nonatomic, then we just need to verify the conditions of \cref{EBB}.
	
	Let $\varepsilon>0$ and $\delta = \varepsilon^{100^{100}}$. By \cref{MainLemma2}, we can find $\xi_n \rightarrow 0$ and, for $n$ large enough, subsets $K_n \subset M$, a $u$-measurable partition $\{W_n(x)\}_{x \in T_n}$ of $K_n$ and a family of pairwise disjoint subsets $\{B_j\}_{j \in J_n}$ satisfying (U1),(U2),(B1) and (B2).
	Therefore, $(f,\mu)$ is Bernoulli.

	\appendix

	\section{SRB measures for skew products}
	\label{appendix}
	

	Let $f \colon M \to M$ be a $C^{1+\alpha}$ diffeomorphism of a compact manifold preserving an SRB measure $\mu$,
	$G$ be a Lie group acting on $N$ preserving volume $m$
	and $\tau \colon M \to G$ be smooth.
	Define the skew product system by
	\begin{align*}
		F  \colon M \times N &\to M \times N \\
		(x,y) &\mapsto (f(x),\tau(x)y).
	\end{align*}
	$F$ preserves the measure $\nu = \mu \otimes m$.
	Let us assume that $\nu$ is ergodic.
	
	\begin{proposition}
	\label{SRBskew}
		The measure $\nu$ is an SRB measure for $F$.
	\end{proposition}
	
	\begin{proof}
		Since $(M,f,\mu)$ is a factor of $(M\times N,F,\nu)$, then it has positive entropy.
		In particular, $F$ has positive Lyapunov exponents $\nu$-almost everywhere.
		
		We first claim that for $\nu$-almost every $(x,y) \in M \times N$, 
		$\pi (W^u_{F}(x,y)) \subset W^u_f(x)$.
		Indeed, for $(x',y') \in W^u_F(x,y)$,
		\begin{equation*}
			\underset{n \rightarrow \infty}{\limsup}
			\frac{1}{n}\log d_M(f^{-n}x,f^{-n}x') \leq 
			\underset{n \rightarrow \infty}{\limsup}\frac{1}{n}\log d_{M\times N}(F^{-n}(x,y),F^{-n}(x',y'))<0.
		\end{equation*}
		
		Let $\{W_f(x)\}_{x \in M}$ be a $\mu$ measurable partition of $M$ subordinate to the unstable foliation on $M$ and consider the $\nu$ measurable partition $\mathcal{W}_f = \{W_f(x) \times N\}_{x \in M}$ of $M\times N$.
		By the above claim, we may consider a u-subordinate partition $\mathcal{W}_{F} = \{W_F(x,y)\}_{(x,y)\in M\times N}$ of $M \times N$ that refines the partition $\mathcal{W}_f$.
		More precisely,
		$W_F(x,y) \subset W_f(x)$ for $\nu$-almost every $(x,y)$.
		Therefore, to disintegrate $\nu$ along $\mathcal{W}_F$, we may first disintegrate along $\mathcal{W}_f$ and then along $\mathcal{W}_F$.		
		Note that for $\mu$-almost every $x \in M$,
		\begin{equation*}
			\nu_{W_f(x)\times N} = \mu_{W_f(x)} \times m.
		\end{equation*}
		But since $\mu$ is SRB, then $\mu_{W_f(x)} \ll m_{W_f(x)}$.
		Therefore,
		\begin{equation}
		\label{verticalvol}
			\nu_{W_f(x) \times N} \ll m_{W_f(x)} \times m = m_{W_f(x)\times N}.
		\end{equation}
		Now for $\mu$-almost every $x$, we wish to disintegrate $\nu_{W_f(x)\times N}$ along the partition $\mathcal{W}_F$ restricted to the submanifold $W_f(x) \times N$ and show that $\nu_{W_F(x,y)} \ll m_{W_F(x,y)}$, for $\nu_{W_f(x)\times N}$-almost every $(x,y)$.
		But \eqref{verticalvol} implies that we are disintegrating a smooth measure on $W_f(x)\times N$ along $\mathcal{W}_F$.
		By standard arguments, the proposition follows from the fact that the unstable foliation has absolutely continuous holonomies between transversal submanifolds inside of $W_f(x)\times N$ (see section 8.6 of \cite{BarPes}).

	\end{proof}

	\bibliographystyle{alpha}  
	\bibliography{dynamics_refs}

\end{document}